\documentclass[a4paper,11pt]{amsart}
\usepackage{amsfonts}
\usepackage{amssymb}
\usepackage[utf8]{inputenc}
\usepackage{amsmath}
\usepackage{pdflscape}
\usepackage{float}
\usepackage{easyReview}
\usepackage{graphicx}
\usepackage[colorlinks=true, linkcolor=red, citecolor=blue]{hyperref}
\usepackage[]{epsfig}
\usepackage[]{pstricks}
\usepackage{tikz}

\newtheorem{theorem}{Theorem}[section]
\newtheorem{proposition}[theorem]{Proposition}
\newtheorem{lemma}[theorem]{Lemma}
\newtheorem{corollary}[theorem]{Corollary}
\newtheorem{assumption}[theorem]{Assumption}
\newtheorem{remark}{Remark}
\theoremstyle{remark}

\newtheorem{definition}{Definition}
\usepackage{mathrsfs}
\def\supp{\operatorname{{supp}}}

\numberwithin{equation}{section}
\makeatletter
\@namedef{subjclassname@2010}{\textup{2020} Mathematics Subject Classification}
\makeatother

\begin{document}
	
	\title[Critical NLS: Stabilization result]{Stabilization of a perturbed quintic defocusing Schr\"odinger equation in $\mathbb{R}^{3}$}
	\author[Braz e Silva]{Pablo Braz e Silva}
	\address{Departamento de Matem\'atica, Universidade Federal de Pernambuco, S/N Cidade Universit\'aria, 50740-545, Recife (PE), Brazil}
	\email{pablo.braz@ufpe.br}
	\author[Capistrano--Filho]{Roberto de A. Capistrano--Filho}
	\email{roberto.capistranofilho@ufpe.br}
	\author[Carvalho]{Jackellyny Dassy do Nascimento Carvalho}
	\email{jackellyny.dassy@ufpe.br}
	\author[Dos Santos Ferreira]{David dos Santos Ferreira*}
	\address{Institut \'Elie Cartan de Lorraine, UMR CNRS 7502, \'equipe SPHINX, INRIA, Universit\'e de Lorraine, F-54506 Vandoeuvre-les-Nancy Cedex, France.}
	\email{ddsf@math.cnrs.fr}
		\thanks{$^*$Corresponding author: ddsf@math.cnrs.fr}
	\thanks{Braz e Silva was partially supported by CAPES/PRINT 88887.311962/2018-00, CAPES/COFECUB 88887.879175/2023-00, CNPq (Brazil) 421573/2023-6, 305233/2021-1, 308758/2018-8, 432387/2018-8. Capistrano--Filho was partially supported by CAPES/PRINT 88887.311962/2018-00, CAPES/COFECUB 88887.879175/2023-00, CNPq (Brazil) 421573/2023-6, 307808/2021-1 and  401003/2022-1 and Propesqi (UFPE). Carvalho was partially supported by CNPq,  CAPES-MATHAMSUD 8881.520205/2020-01 and FACEPE BFD-0014-1.01/23.}
	
	\begin{abstract}		
This article addresses the stabilizability of a perturbed quintic defocusing Schrödinger equation in $\mathbb{R}^{3}$ at the $H^1$--energy level, considering the influence of a damping mechanism. More specifically, we establish a profile decomposition for both linear and nonlinear systems and use them to show that, under certain conditions, the sequence of nonlinear solutions can be effectively linearized. Lastly, through microlocal analysis techniques, we prove the local exponential stabilization of the solution to the perturbed Schrödinger equation in $\mathbb{R}^{3}$ showing an observability inequality for the solution of the system under consideration, which is the key result of this work.
	\end{abstract}
	
	\date{\today}
	\maketitle
	
	\tableofcontents
	
	\thispagestyle{empty}
	
	\section{Introduction}
	\subsection{Addressed issue}
	This article is devoted to the stabilization properties of the quintic defocusing Schr\"odinger equation in $\mathbb{R}^{3+1}$ 
		\begin{equation}\label{eq1aa}
		\left\{
		\begin{array}{ll}
			i\partial_{t} u + \Delta u  = |u|^{4}u ,& (t,x) \in [0,+\infty)\times  \mathbb{R}^{3},\\
			u(x,0) = u_{0} (x) \in {H}^{1}(\mathbb{R}^{3}),& x\in  \mathbb{R}^{3},
		\end{array}
		\right.
	\end{equation}
	where $u(t, x)$ is a complex-valued field in spacetime $[0,+\infty)\times  \mathbb{R}^{3}$ and the subscripts denote the corresponding partial derivatives.  Semilinear Schrödinger equations - with and without potentials, and with various nonlinearities - arise as models for diverse physical phenomena, including Bose-Einstein condensates \cite{Gross,Pitaevski} and as a description of the envelope dynamics of a general dispersive wave in a weakly nonlinear medium\footnote{For details, see the survey \cite[Chapter 1]{SuSu}.}.
	
	Equation \eqref{eq1aa} has a Hamiltonian structure, namely
		\begin{equation}\label{eq1aaa}
E(u(t)):=\int \frac{1}{2}|\nabla u(t, x)|^2dx+\frac{1}{6}|u(t, x)|^6 d x,
	\end{equation}
which is preserved by the flow \eqref{eq1aa}. We shall often refer to it as the energy and write $E(u)$ for $E(u(t))$. Our interest here in the defocusing quintic equation \eqref{eq1aaa} is motivated mainly by the fact that the problem concerning the energy norm is critical. 

To be precise, we are interested in internal stabilization for the perturbed defocusing critical nonlinear Schr\"odinger equation (C-NLS) on $\mathbb{R}^3$
	\begin{equation}\label{eq1a}
		\left\{
		\begin{array}{ll}
			i\partial_{t} u + \Delta u  - u = |u|^{4}u ,& (t,x) \in [0,+\infty)\times  \mathbb{R}^{3},\\
			u(x,0) = u_{0} (x) \in {H}^{1}(\mathbb{R}^{3}),& x\in  \mathbb{R}^{3},
		\end{array}
		\right.
	\end{equation}
	where $u = u(t, x)$ is a complex-valued function of two variables $x \in\mathbb{R}^3$ and $t\in[0,+\infty)$. We are mainly concerned with the following stabilizability problem for system \eqref{eq1a}.
	
	\vspace{0.2cm}
	\noindent\textbf{Stabilization problem:} Can one find a feedback control law $f(x,t)=\mathcal{K} u$ so that the resulting closedloop system
$$
i\partial_{t} u + \Delta u  - u - |u|^{4}u =\mathcal{K} u , (t,x) \in [0,+\infty)\times  \mathbb{R}^{3}
$$
is asymptotically stable as $t \rightarrow+\infty$ ?

	\vspace{0.2cm}
	
	Note that, similarly to the system \eqref{eq1aa}, system \eqref{eq1a} preserves the $L^{2}$--mass, defined as $\|u(t)\|^{2}_{L^{2}}$, and the $H^{1}$--Hamiltonian (energy) given by
$$
		E(u)= \frac{1}{2} \int_{\mathbb{R}^{3}}|u(t)|^{2} \ dx + \frac{1}{2}\int_{\mathbb{R}^{3}} |\nabla u(t)|^{2} \ dx + \frac{1}{6} \int_{\mathbb{R}^{3}} |u(t)|^{6} \ dx.
$$
Thus, to answer the previous question appropriately, we need to present an operator $\mathcal{K}$ that transforms the energy $E(u)$ into a decreasing function.  For this, consider a non-negative function $a \in C^{\infty}(\mathbb{R}^{3};[0,1])$ satisfying, almost everywhere, 
	\begin{equation}\label{eq 3}
		a(x) =	\left\{
		\begin{array}{ll}
			0, & \mbox{ if }|x| \leq R,  \\
			1, & \mbox{ if } |x| \geq R+ 1,
		\end{array}
		\right.
	\end{equation}
	for some $R>0$ and $\eta >0$ such that
	\begin{equation*}
		a(x)\geq \eta >0 \ \ \ \mbox{ for } \ \ |x| \geq R.
	\end{equation*}

From now on, the stabilization system in consideration is
\begin{equation} \label{eq 4}
	\left\{
	\begin{array}{ll}
		i\partial_{t}u  +\Delta u -u - |u|^{4}u -a(x)(1-\Delta)^{-1}a(x)\partial_{t}u=0	,& (t,x) \in [0,+\infty)\times  \mathbb{R}^{3},\\
	u(x,0) = u_{0} (x) \in {H}^{1}(\mathbb{R}^{3}),& x\in  \mathbb{R}^{3},
	\end{array}
	\right.
\end{equation} 
where $a(x)$ is given by \eqref{eq 3} and the solution $u=u(t,x)$ of the system satisfies the energy identity
$$
		E(u)(t_{2}) - E(u)(t_{1}) = - 2 \int_{t_{1}}^{t_{2}} \Big\|(1-\Delta)^{-\frac{1}{2}}a(x)\partial_{t}u\Big\|^{2}_{L^{2}} \ dt,
$$
where $E(u)(t)$ is now decreasing and, therefore, system \eqref{eq 4} is dissipative. Observe that, as noted by Laurent \cite{Laurent1}, the connection to the original equation \eqref{eq1aa} with damping can be established through the change of variable \( z = e^{-it}u \). Additionally, a more physically relevant damping term would be \( ia(x)u \), as used in the one-dimensional case in \cite{Laurent} and in a compact Riemannian manifold $(M,g)$ in \cite{BurqMoyano}. However, the damping term in \eqref{eq 4} is particularly suited to the \( H^1 \)-energy, which aligns with the regularity level at which we solve the equation.

Before presenting the contributions of this work, let us give a brief state of the art concerning the system and problems under consideration.

\subsection{Literature review} The Cauchy problem associated with system \eqref{eq1aa} has been extensively investigated, see for instance, \cite{cazenave,Gri,bourgain,bourgain1,GiVe,Kato}. It has been established \cite{CaWe,cazenave} that when the initial data $u_0(x)$ possesses finite energy, the Cauchy problem is locally well-posed. This implies the existence of a local-in-time solution to \eqref{eq1aa} belonging to the space $C_t^0 \dot{H}_x^1 \cap L_{t, x}^{10}$, and such a solution is unique within this class. Moreover, the mapping taking initial data to its corresponding solution exhibits local Lipschitz continuity in these norms. In cases where the energy is small, the solution exists globally in time and scatters to a solution $u_{ \pm}(t)$ of the free Schrödinger equation $\left(i \partial_t+\Delta\right) u_{ \pm}=0$. This scattering behavior is characterized by $\left\|u(t)-u_{ \pm}(t)\right\|_{\dot{H}^1\left(\mathbb{R}^3\right)} \rightarrow 0$ as $t \rightarrow \pm \infty$. However, for large initial data, the arguments presented in \cite{CaWe,cazenave}  fail to establish global well-posedness, even with the conservation of the energy \eqref{eq1aaa}. This limitation arises because the duration of existence predicted by the local theory depends on the data profile and the energy. This is in contrast to sub-critical equations like the cubic equation 
\begin{equation}\label{aaaaaa}
iu_t + \Delta u = |u|^2u,
\end{equation} 
where local well-posedness theory ensures global well-posedness and scattering even for large energy data, as discussed in \cite{Give1, cazenave_book}.

For large finite energy data, particularly for those assumed to be radially symmetric, Bourgain \cite{bourgain} demonstrated global existence and scattering for \eqref{eq1aa} in $\dot{H}^1\left(\mathbb{R}^3\right)$. Subsequently, Grillakis \cite{Gri} presented an alternative argument that partially recovered the results of \cite{bourgain}, focusing on global existence from smooth, radial, finite energy data. Recently, Colliander \textit{et al.} \cite{Co} obtained global well-posedness, scattering, and global $L^{10}$ space-time bounds for energy-class solutions to the quintic defocusing Schrödinger equation in $\mathbb{R}^{1+3}$, which is energy-critical. Notably, they established the global existence of classical solutions\footnote{For details about global well-posedness, scattering, and blow-up for the nonlinear Schr\"{o}dinger equation in the radial case, see \cite{KeMe}.}.

While the well-posedness theory for system \eqref{eq1aa} has been extensively explored, the study of control properties concerning the quintic critical defocusing Schrödinger equation in $\mathbb{R}^{3+1}$ is less advanced. Most research efforts have concentrated on the cubic Schrödinger equation \eqref{aaaaaa}, which has been a focal point in the past few decades. For instance, concerning control aspects, relevant literature includes \cite{Miller,Phung,RaTaTeTu} and related works. In terms of Carleman estimates and their applications to inverse problems, references such as \cite{BaPu,CaGa,CaCriGa,LaTriZhang, RoMe, YuYa} are noteworthy, along with their respective bibliographies. A comprehensive overview of contributions up to 2003 can be found in \cite{Zuazua}.

Concerning the stabilization problem, there are several results considering the equation \eqref{aaaaaa}.   Some similar results were obtained in dimension $2$ in the article of Dehman \textit{et al.} \cite{DeGeLe}, where the stabilization in $H^1$ is proved for the defocusing equation \eqref{aaaaaa} on compact surfaces considering the feedback law as $\mathcal{K}=a(x)(1-\Delta)^{-1}a(x)\partial_{t}$. Employing the same techniques for a one-dimensional case, Laurent \cite{Laurent} showed global internal controllability in large time for the system \eqref{aaaaaa} in an interval, however, in this case, with a physically relevant damping term $\mathcal{K}=ia(x)$. The strategy combines stabilization and local controllability near $0$. More recently, in a very nice article in \cite{Laurent1}, the same author gave contributions to the stabilization problem for the equation \eqref{aaaaaa} on some compact manifolds of dimension $3$. It is important to point out that in both works \cite{DeGeLe,Laurent1}, the main ingredients to achieve the results are some geometrical assumptions: geometric control and unique continuation. These are necessary due to the characteristics of the function $a(x)$. For more details about these questions, see \cite{Laurent3}.  The authors also suggest the following two references \cite{CaCaFuNa} for the $2-D$ case of the defocusing Schr\"odinger equation with locally distributed damping and \cite{BoCa} for the case of noncompact Riemannian manifolds and exterior domains. 

We also mention that  Rosier and Zhang  \cite{RoZhaMMM} (see also \cite{RoZhaSIAM}) considered the equation \eqref{aaaaaa} in the $\mathcal{R}=(0,l_1)\times\cdot\cdot\cdot\times(0,l_n)$ and investigated the control properties of the semi-linear Schr\"{o}dinger equation 
\[
i\partial_tu+\Delta u+\lambda|u|^{\alpha}u=ia(x)h(x,t), \text{  }\text{  }x\in\mathbb{T}^n\text{  }\text{  }t\in(0,T),
\]
where $\lambda\in\mathbb{R}$ and $\alpha\in 2\mathbb{N}^{\ast}$ by combining new linear controllability results in the spaces $H^s(\mathcal{R})$ with Bourgain analysis. In this case, the geometric control condition is not required (see \cite{RoZhaMMM} for more details).

Finally,  another recent work \cite{CaPa} extended the results from \cite{DeGeLe}. Therein, the authors studied global controllability and stabilization properties for the fractional Schr\"odinger equation on $d-$dimensional compact Riemannian manifolds without boundary $(M, g)$.  Using microlocal analysis, they showed the propagation of regularity, which, together with the geometric control condition and a unique continuation property, allowed them to prove global control results.

	\subsection{Main result and heuristics}  Our main theorem states that we can obtain an exponential decay for the energy of this system with a perturbation term for some solutions that are bounded in the energy space but small in a lower norm. The local stabilization result is the following.
	
	\begin{theorem}\label{mainteo}Let $\lambda_{0} > 0$. There exist constants $C,\gamma > 0$ and $\delta > 0$ such that for any $u_{0}$ in $H^{1}(\mathbb{R}^{3})$, with $\|u_{0}\|_{H^{1}(\mathbb{R}^{3})} \leq \lambda_{0}$ and $\|u_{0}\|_{H^{-1}(\mathbb{R}^{3})}\leq \delta,$ the unique strong solution of problem \eqref{eq 4} satisfies
		\begin{equation}\label{Eet}
			E(u)(t) \leq Ce^{-\gamma t}E(u)(0), \ \ \ \forall \, t\geq 0.
		\end{equation}
	\end{theorem}
	
	 Let us give a brief general idea of how we obtain our results, which provide a (local) answer to the stabilization problem.   Initially, it is important to acknowledge that the primary concern is to establish the stabilization of the energy linked with \eqref{eq1aa}. However, due to technical challenges, more specifically, due to the difficulty in  identifying suitable embeddings between nonhomogeneous Sobolev spaces, 
	 we perturb this system, transforming it in the system \eqref{eq 4}. Our inspiration for this approach comes from a result concerning the Klein-Gordon equation on a 3-dimensional compact manifold obtained by Laurent \cite{Laurent2}.
	
	To obtain Theorem \ref{mainteo}, we use a profile decomposition to describe how linear and nonlinear solutions approach each other in some sense, applying the same methodology used for the Klein-Gordon equation in the three-dimensional case. Precisely, to show that the energy of the system \eqref{eq 4} decays exponentially (even locally), it is necessary to show the \textit{observability inequality}
\begin{equation}\label{oi-ra}
E(u)(0) \leq C \int_{0}^{T}\int_{\mathbb{R}^{3}}|(1-\Delta)^{-\frac{1}{2}}a\partial_{t}u|^{2} \ dx dt,
\end{equation}	
which is obtained through propagation results for the microlocal defect measure through the strategy used in \cite{DeLeZu}. Before that, we need to prove that solutions for the nonlinear system \eqref{eq1a} behave similarly to the solutions for the linear system associated with system \eqref{eq1a}. To this end, we introduce a decomposition into profiles for both linear and nonlinear solutions as carried out by Keraani \cite{keraani}.
	
Note that, even with the addition of a perturbation term, our approach will not undergo any significant modification. Indeed, the unknown $w=e^{it}u$ is a solution of
\begin{equation*}
	\left\{
	\begin{array}{ll}
		i\partial_{t}w +\Delta w = |w|^{4}w, & (t,x) \in \mathbb{R} \times \mathbb{R}^{3}, \\
		w(x,0) = u_{0} \in \dot{H}^{1}(\mathbb{R}^{3}),& x\in \mathbb{R}^{3},
	\end{array}
	\right.
\end{equation*}
which is the original system.  Therefore, it is possible to use, in our new system, the entire profile decomposition theory developed by Keraani in \cite{keraani} as well as the scattering property.

Finally, with this decomposition of profiles in hand combined with the propagation results, which involves arguments from
microlocal analysis, we show the observability inequality \eqref{oi-ra}, ensuring the decay of the energy in the sense of estimate \eqref{Eet}.

\begin{remark} The following observations are worth mentioning:
\begin{itemize}
\item[i.] Theorem \ref{mainteo} completes the analysis begun in \cite{BrCaCaSa}, where local controllability was shown. 
\item[ii.] Our result here gives a first step to understanding the stabilization properties of system \eqref{eq 4}. Since our result is local, it is necessary to prove global stabilization (see, for instance, \cite{DeGeLe,Laurent1}) to get global controllability. In conclusion, global controllability reduces to proving that system \eqref{eq 4} is globally exponentially stabilizable, which remains an open problem (see Figure \ref{figura1}).
\begin{figure}[h!]
\tikzset{every picture/.style={line width=0.75pt}} 
\begin{tikzpicture}[x=0.70pt,y=0.70pt,yscale=-1,xscale=1]
\draw    (79.01,10) -- (79.99,263) ;
\draw [shift={(80,266)}, rotate = 269.78] [fill={rgb, 255:red, 0; green, 0; blue, 0 }  ][line width=0.08]  [draw opacity=0] (8.93,-4.29) -- (0,0) -- (8.93,4.29) -- cycle    ;
\draw [shift={(79,7)}, rotate = 89.78] [fill={rgb, 255:red, 0; green, 0; blue, 0 }  ][line width=0.08]  [draw opacity=0] (8.93,-4.29) -- (0,0) -- (8.93,4.29) -- cycle    ;
\draw    (376,229) -- (34,229) ;
\draw [shift={(31,229)}, rotate = 360] [fill={rgb, 255:red, 0; green, 0; blue, 0 }  ][line width=0.08]  [draw opacity=0] (8.93,-4.29) -- (0,0) -- (8.93,4.29) -- cycle    ;
\draw [shift={(379,229)}, rotate = 180] [fill={rgb, 255:red, 0; green, 0; blue, 0 }  ][line width=0.08]  [draw opacity=0] (8.93,-4.29) -- (0,0) -- (8.93,4.29) -- cycle    ;
\draw  [line width=3] [line join = round][line cap = round] (79,137.5) .. controls (79,137.5) and (79,137.5) .. (79,137.5) ;
\draw  [line width=3] [line join = round][line cap = round] (79,68) .. controls (79,68) and (79,68) .. (79,68) ;
\draw  [line width=3] [line join = round][line cap = round] (80,197.5) .. controls (80,197.5) and (80,197.5) .. (80,197.5) ;
\draw  [line width=3] [line join = round][line cap = round] (80,217.5) .. controls (80,217.5) and (80,217.5) .. (80,217.5) ;
\draw [color={rgb, 255:red, 208; green, 2; blue, 27 }  ,draw opacity=1 ]   (79.5,136.5) .. controls (119.5,106.5) and (129,227.5) .. (169,197.5) ;
\draw  [line width=3] [line join = round][line cap = round] (169,197.5) .. controls (169,197.5) and (169,197.5) .. (169,197.5) ;
\draw [color={rgb, 255:red, 208; green, 2; blue, 27 }  ,draw opacity=1 ]   (259,217.5) .. controls (299,187.5) and (298,98.5) .. (338,68.5) ;
\draw  [line width=3] [line join = round][line cap = round] (260,217.5) .. controls (260,217.5) and (260,217.5) .. (260,217.5) ;
\draw  [line width=3] [line join = round][line cap = round] (338,68.5) .. controls (338,68.5) and (338,68.5) .. (338,68.5) ;
\draw  [color={rgb, 255:red, 144; green, 19; blue, 254 }  ,draw opacity=1 ] (169.27,196.59) .. controls (172.43,200.29) and (175.46,203.81) .. (179.81,205.05) .. controls (184.17,206.28) and (188.64,204.89) .. (193.32,203.42) .. controls (197.99,201.95) and (202.46,200.56) .. (206.82,201.8) .. controls (211.17,203.03) and (214.2,206.56) .. (217.36,210.25) .. controls (220.53,213.94) and (223.56,217.47) .. (227.91,218.7) .. controls (232.26,219.94) and (236.74,218.55) .. (241.41,217.08) .. controls (246.09,215.61) and (250.56,214.22) .. (254.91,215.45) .. controls (256.64,215.94) and (258.16,216.79) .. (259.56,217.87) ;
\draw  [dash pattern={on 0.84pt off 2.51pt}]  (79.83,197.83) -- (169,197.5) ;
\draw  [dash pattern={on 0.84pt off 2.51pt}]  (80.17,217.5) -- (259,217.5) ;
\draw  [dash pattern={on 0.84pt off 2.51pt}]  (79.17,68) -- (338,68.5) ;
\draw  [dash pattern={on 0.84pt off 2.51pt}]  (169,197.5) -- (169.5,228.83) ;
\draw  [dash pattern={on 0.84pt off 2.51pt}]  (259,217.5) -- (259.17,228.83) ;
\draw  [dash pattern={on 0.84pt off 2.51pt}]  (338,68.5) -- (339.5,228.5) ;
\draw [color={rgb, 255:red, 155; green, 155; blue, 155 }  ,draw opacity=1 ]   (202,40) -- (125.16,147.87) ;
\draw [shift={(124,149.5)}, rotate = 305.46] [color={rgb, 255:red, 155; green, 155; blue, 155 }  ,draw opacity=1 ][line width=0.75]    (10.93,-3.29) .. controls (6.95,-1.4) and (3.31,-0.3) .. (0,0) .. controls (3.31,0.3) and (6.95,1.4) .. (10.93,3.29)   ;
\draw [color={rgb, 255:red, 155; green, 155; blue, 155 }  ,draw opacity=1 ]   (202,40) -- (291.57,128.1) ;
\draw [shift={(293,129.5)}, rotate = 224.52] [color={rgb, 255:red, 155; green, 155; blue, 155 }  ,draw opacity=1 ][line width=0.75]    (10.93,-3.29) .. controls (6.95,-1.4) and (3.31,-0.3) .. (0,0) .. controls (3.31,0.3) and (6.95,1.4) .. (10.93,3.29)   ;
\draw [color={rgb, 255:red, 155; green, 155; blue, 155 }  ,draw opacity=1 ]   (244,256.5) -- (206.24,208.57) ;
\draw [shift={(205,207)}, rotate = 51.77] [color={rgb, 255:red, 155; green, 155; blue, 155 }  ,draw opacity=1 ][line width=0.75]    (10.93,-3.29) .. controls (6.95,-1.4) and (3.31,-0.3) .. (0,0) .. controls (3.31,0.3) and (6.95,1.4) .. (10.93,3.29)   ;

\draw (10,12.4) node [anchor=north west][inner sep=0.75pt]  [xscale=0.95,yscale=0.95]  {$H^1(\mathbb{R}^3)$};
\draw (378,238.4) node [anchor=north west][inner sep=0.75pt]  [xscale=0.95,yscale=0.95]  {$t$};
\draw (62,131.4) node [anchor=north west][inner sep=0.75pt]  [font=\scriptsize,xscale=0.95,yscale=0.95]  {$\phi $};
\draw (59,61.4) node [anchor=north west][inner sep=0.75pt]  [font=\scriptsize,xscale=0.95,yscale=0.95]  {$\phi _{T}$};
\draw (64,186.4) node [anchor=north west][inner sep=0.75pt]  [font=\scriptsize,xscale=0.95,yscale=0.95]  {$\tilde{\phi }$};
\draw (63,206.4) node [anchor=north west][inner sep=0.75pt]  [font=\scriptsize,xscale=0.95,yscale=0.95]  {$\tilde{\phi }_{T}$};
\draw (160,232.4) node [anchor=north west][inner sep=0.75pt]  [font=\scriptsize,xscale=0.95,yscale=0.95]  {$T_{1}$};
\draw (252,231.4) node [anchor=north west][inner sep=0.75pt]  [font=\scriptsize,xscale=0.95,yscale=0.95]  {$T_{2}$};
\draw (332,231.4) node [anchor=north west][inner sep=0.75pt]  [font=\scriptsize,xscale=0.95,yscale=0.95]  {$T$};
\draw  [fill={rgb, 255:red, 255; green, 214; blue, 214 }  ,fill opacity=1 ]  (136,18) -- (272,18) -- (272,43) -- (136,43) -- cycle  ;
\draw (139,22) node [anchor=north west][inner sep=0.75pt]  [xscale=0.95,yscale=0.95] [align=left] {\textit{Global Stabilization}};
\draw  [fill={rgb, 255:red, 234; green, 215; blue, 255 }  ,fill opacity=1 ]  (188,252) -- (326,252) -- (326,277) -- (188,277) -- cycle  ;
\draw (191,256) node [anchor=north west][inner sep=0.75pt]  [xscale=0.95,yscale=0.95] [align=left] {\textit{Local Controllability}};
\end{tikzpicture}
\caption{Global controllability result}\label{figura1}
\end{figure}
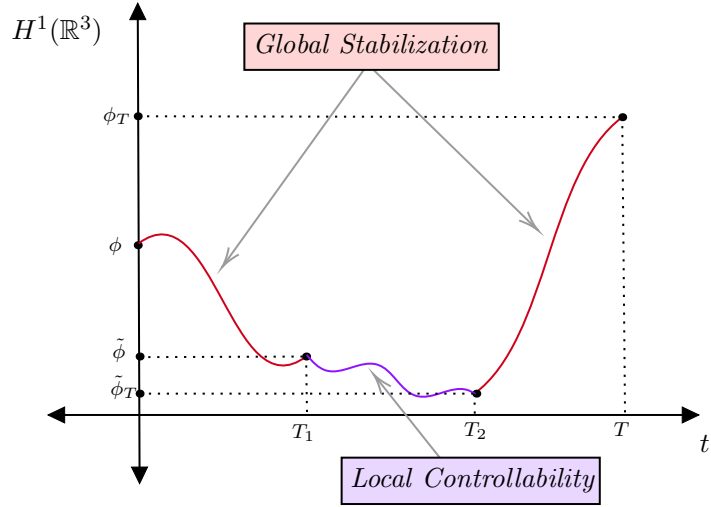
\item[iii.] Note that  $a \in C^{\infty}(\mathbb{R}^{3})$ satisfying \eqref{eq 3} act in $\omega : = \big(\mathbb{R}^{3} \backslash B_{R}(0)\big)$. Thus, as opposed to \cite{Laurent2}, the function $\omega$ satisfies a unique geometrical assumption: \ \ There exists $T_{0} > 0$ such that every geodesic travelling at speed 1 meets $\omega$ in a time $t < T_{0}$, for some $T_{0}>0.$
\item[iv.] As mentioned in \cite{Laurent1}, the most physically relevant damping term for system \eqref{eq 4} would be $ia(x)u$ instead of $a(x)(1-\Delta)^{-1}a(x)\partial_{t}u$, as used in the one-dimensional case \cite{Laurent}. For this damping term, the analysis remains open. 
\end{itemize}
\end{remark}

	\subsection{Structure of the work} We conclude our introduction by providing an outline of this work. In Section \ref{sec2}, we introduce the profile decomposition of the $H^{1}$-critical Schrödinger equation in three spatial dimensions. The nonlinear profile decomposition is detailed in Section \ref{sec3}, following Keraani's approach in \cite{keraani}. Additionally, we present a result ensuring that sequences of solutions for the nonlinear system behave similarly to sequences of solutions for the linear system, following ideas from \cite{Laurent2}. Section \ref{sec4} is dedicated to proving the observability inequality associated with the solutions of system \eqref{eq 4}, thereby providing the proof of Theorem \ref{mainteo}. Finally, two appendices are included: Appendix \ref{apA} reviews the Cauchy problem \eqref{eq 4}, while Appendix \ref{apB} compiles some results on the propagation of solutions of the linear Schrödinger equation, based on the ideas from \cite{DeGeLe}.

	 
	 \section{Profile decomposition}\label{sec2}
 In this section, we will consider the $H^{1}$-critical Schrödinger equation in three spatial dimensions
		\begin{equation} \label{eq 15}
		\left\{
		\begin{array}{ll}
			i\partial_{t} u + \Delta u -|u|^{4}u= 0, \ \ & (t,x) \in [0,T] \times \mathbb{R}^{3} \\
			u(0,x) =\varphi(x), \ \ &x \in \mathbb{R}^{3}.
		\end{array}
		\right.
	\end{equation}		
Considering $\varphi \in \dot{H}^{1}(\mathbb{R}^{3})$, the solution of the linear system associated with \eqref{eq 15} is given explicitly by $v=e^{it\Delta}\varphi$, which belongs to the class $ C (\mathbb{R}_{t},\dot{H}^{1}(\mathbb{R}_{x}^{3}))$, and satisfies the conservation law
$$
		E_{0}(v)(t) := \int_{\mathbb{R}^{3}} |\nabla v(t)|^{2} \ dx = E_{0}(\varphi).
$$
The small data theory explored by \cite{cazenave} ensures that there exists $\lambda > 0$ such that, if  
\begin{equation}\label{eq 16}
	\|\varphi\|_{\dot{H}^{1}(\mathbb{R}^{3})} \leq \lambda,
\end{equation} then there exists a unique maximal solution $u(t,x)$ of system \eqref{eq 15} satisfying
$$
	u \in C(\mathbb{R};\dot{H}^{1}(\mathbb{R}^{3})), \ \ u \in L^{10}(\mathbb{R}^{4}), \ \ \nabla u \in L^{\frac{10}{3}}(\mathbb{R}^{4}).
$$

Our first goal in this section is to prove that every sequence of solutions associated with the linear Schrödinger equation \eqref{eq 15} with bounded data in $\dot{H}^{1}(\mathbb{R}^{3})$ can be written, up to a subsequence, as an almost orthogonal sum of sequences of the type
$$h_{n}^{-\frac{1}{2}}\varphi\Big(\frac{t-t_{n}}{h_{n}^{2}},\frac{x-x_{n}}{h_{n}}\Big),$$ where $\varphi$ is a solution of the linear Schrödinger equation with a small remainder term in Strichartz norms. Using this decomposition, we show 
a similar one for system \eqref{eq 15}, assuming that the initial data belong to a ball in the energy space where the equation is solvable. This implies, in particular, the existence of an a priori estimate for the Strichartz norms in terms of energy. Let us begin with the following definition.

	\begin{definition}\label{def3}
	Let $\lambda_{0}$ be the supremum of all $\lambda$ in \eqref{eq 16} for such that one has 
	global existence of a maximal solution $u$ for \eqref{eq 15}, with $u \in C(\mathbb{R};\dot{H}^{1}(\mathbb{R}^{3})\cap L^{10}(\mathbb{R}^{4})$ and $\nabla u \in L^{\frac{10}{3}}(\mathbb{R}^{4})$.
\end{definition}  

\begin{remark}
If $\|\varphi\|_{\dot{H}^{1}(\mathbb{R}^{3})} < \lambda_{0}$, then system $\eqref{eq 15}$ admits a complete scattering theory concerning its associated linear problem. However, it is an open problem to prove that $\lambda_{0}=\infty$, i.e., to prove global well-posedness of the IVP \eqref{eq 15} for any initial data in $\dot{H}^{1}(\mathbb{R}^{3})$\footnote{Bourgain solved this problem in the particular case of radially symmetric data \cite{bourgain}.}.
\end{remark}

The following definition will be useful in the first part of the proof of the linear profile decomposition, which consists of the extraction of the scales of oscillation $h_n$.
\begin{definition}\label{def4}
\end{definition}
\begin{itemize}
	\item[i)] We call \textbf{scale} every sequence $\underline{h}=(h_{n})_{n\geq 0}$ of positive numbers and \textbf{core} every sequence $[\underline{x},\underline{t}]=(x_{n},t_{n})_{n\geq0}\subset \mathbb{R}^{3}\times \mathbb{R}$. We denote a  \textbf{scale-core} by $[\underline{h},\underline{x},\underline{t}]$.
	\item[ii)] We say that two sequences of scale-core $[\underline{h}^{(1)},\underline{x}^{(1)},\underline{t}^{(1)}]$ and $[\underline{h}^{(2)},\underline{x}^{(2)},\underline{t}^{(2)}]$ are orthogonal if either
$$
		\dfrac{h_{n}^{(1)}}{h_{n}^{(2)}} + \dfrac{h_{n}^{(2)}}{h_{n}^{(1)}}\longrightarrow + \infty, \quad \text{as} \quad n \rightarrow \infty,
$$
	or, $h_{n}^{(1)} = h_{n}^{(2)} = h_{n}$ and 
$$
		\Bigg| \dfrac{t_{n}^{(1)} - t_{n}^{(2)}}{h_{n}^{2}} \Bigg| + \Bigg| \dfrac{x_{n}^{(1)} - x_{n}^{(2)}}{h_{n}} \Bigg| \longrightarrow + \infty,\quad \text{as} \quad n \rightarrow \infty.
$$
	In each respective case above, we denote either $[\underline{h}^{(1)},\underline{x}^{(1)},\underline{t}^{(1)}] \perp [\underline{h}^{(2)},\underline{x}^{(2)},\underline{t}^{(2)}]$ or $(\underline{x}^{(1)},\underline{t}^{(1)}) \perp_{h_{n}} (\underline{x}^{(2)},\underline{t}^{(2)}).$ 
\end{itemize}
\subsection{Concentrating solutions} Now, we will introduce the concept of concentration solution, which will be extremely important for the study of the asymptotic behavior of our system.
	\begin{definition} \label{def5}
\end{definition}
\begin{itemize}
	\item [i)]	Let $f \in L^{\infty}(\mathbb{R};\dot{H}^{1}(\mathbb{R}^{3}))$,  $\underline{h} = h_{n} \in \mathbb{R}^{*}_{+}$, $\underline{x} = x_{n} \in  \mathbb{R}^{3}$ and $ \underline{t} = t_{n} \in \mathbb{R}$ such that $\lim_{n} (h_{n},x_{n},t_{n}) =(0,x_{\infty},t_{\infty})$. 
	A linear concentrating solution associated to $[f,\underline{h},\underline{x},\underline{t}]$ is a sequence $(v_{n})_{n \in \mathbb{N}}$ of solutions to
$$
		i\partial_{t} v_{n} + \Delta v_{n}  = 0, \ \ (t,x) \in \mathbb{R} \times \mathbb{R}^{3},
$$
	of the form 
$$
		v_{n}(t,x) = \frac{1}{\sqrt{h_{n}}}f \Big(\dfrac{t-t_{n}}{h_{n}^{2}},\dfrac{x-x_{n}}{h_{n}}\Big);
$$
	\item[ii)]	The associated nonlinear concentrating solution is a sequence $(u_{n})_{n \in \mathbb{N}}$ of solutions to  
$$
		\left\{
		\begin{array}{ll}
			i\partial_{t} u_{n} + \Delta u_{n} -|u_{n}|^{4}u_{n} = 0, \ \ & (t,x) \in \mathbb{R} \times \mathbb{R}^{3},\\
			u_{n}(0) = v_{n}(0), \ \ & x \in \mathbb{R}^{3} ,
		\end{array}
		\right.
$$
of the form
	\begin{equation*}
		u_{n}(t,x) = \frac{1}{\sqrt{h_{n}}}\overline{f} \Big(\dfrac{t-t_{n}}{h_{n}^{2}},\dfrac{x-x_{n}}{h_{n}}\Big),
	\end{equation*}
	where $\overline{f}(-t_{n}/h_{n}^{2}) = f(-t_{n}/h_{n}^{2})$.
\end{itemize}

The next definition is the tool that will be used to “track back” the concentrations.

	\begin{definition} \label{def6}
	Let $x_{\infty} \in \mathbb{R}^{3}$, $t_{\infty} \in \mathbb{R}$, $\underline{h} = h_{n} \in \mathbb{R}^{*}_{+}$, $\underline{x} = x_{n} \in  \mathbb{R}^{3}$ and $f \in L^{\infty}(\mathbb{R};\dot{H}^{1}(\mathbb{R}^{3}))$ 
	such that $\lim_{n} (h_{n},x_{n},t_{n}) =(0,x_{\infty},t_{\infty})$. Given a bounded sequence $(f_{n})_{n\in \mathbb{N}}$ in $L^{\infty}(\mathbb{R};\dot{H}^{1}(\mathbb{R}^{3}))$, we write $D_{h_{n}}f_{n} \rightharpoonup f$ if $h_{n}^{\frac{1}{2}}f_{n}(t_{n} + h_{n}^{2}t,x_{n} + h_{n}x) \rightharpoonup f(t,x)$ weakly in $\dot{H}^{1}(\mathbb{R}^{3}),$ for all $t \in \mathbb{R}$.
\end{definition}
Of course, this definition depends on the core of concentration $h_{n}, x_{n}$ and $t_{n}$.  When several rates of concentration $[\underline{h}^{(j)},\underline{x}^{(j)},\underline{t}^{(j)}]$,
$j \in \mathbb{N}$, are used in a proof, we use the notation $D^{(j)}_{h}$ to distinguish them.
\begin{lemma}\label{lemma3.2}
 If $f_{n}$ is a linear concentrating solution associated to $[f,\underline{h},\underline{x},\underline{t}]$, then
	$D_{h_{n}}f_{n} \rightharpoonup f.$
\end{lemma}
\begin{proof}
 Since $f_{n}$ has the form 
	$$f_{n}(t,x)= \frac{1}{\sqrt{h_{n}}} f\Big(\frac{t-t_{n}}{h_{n}^{2}},\frac{x-x_{n}}{h_{n}}\Big),$$
the change of variables
	$$\sqrt{h_{n}}f_{n}(t_{n} + h_{n}^{2}s,x_{n} + h_{n}y) =  f(s,y)$$
yields that
$$
		L_{n}  =  \sqrt{h_{n}}\int_{\mathbb{R}^{3}} \nabla_{y} f_{n}(t_{n} + h_{n}^{2}s,x_{n} + h_{n}y) \cdot \nabla_{y} \varphi(s,y) \ dx  =  \int_{\mathbb{R}^{3}} \nabla_{y} f(s) \cdot \nabla_{y} \varphi(s) \ dy .
$$
	Thus, 
	$$\int_{\mathbb{R}^{3}} \nabla_{x} f_{n}(t_{n} + h_{n}^{2}s) \cdot \nabla_{x} u_{n}(t_{n} + h_{n}^{2}s) \ dx = \int_{\mathbb{R}^{3}} \nabla_{y} f(s) \cdot \nabla_{y} \varphi(s) \ dy,$$
which gives $D_{h_{n}} f_{n} \rightharpoonup f$. 
\end{proof}
	\begin{lemma} \label{lemma3.3}
 If $u_{n}$ is a linear concentrating solution associated to $[\varphi,\underline{h}, \underline{x},\underline{t}]$, then
$$
		\|u_{n}\|_{L^{\infty}\dot{H}^{1}} = \|\varphi\|_{L^{\infty}\dot{H}^{1}}, \ \ \|u_{n}\|_{L^{10}_{t}L^{10}_{x}} = \|\varphi\|_{L^{10}_{t}L^{10}_{x}} \ \mbox{ and } \  \|\nabla u_{n}\|_{L^{\frac{10}{3}}_{t}L^{\frac{10}{3}}_{x}} = \|\nabla \varphi\|_{L^{\frac{10}{3}}_{t}L^{\frac{10}{3}}_{x}}.
$$

\end{lemma}
\begin{proof}
We prove only the first equality since the other two are similarly obtained.  
Using Definition \ref{def5} and 
the change of variables $\frac{t-t_{n}}{h_{n}^{2}} = s$ and $\frac{x-x_{n}}{h_{n}} = y$, we get
	\begin{eqnarray*}
		\|\nabla u_{n}(t)\|_{L^{2}} 
		 =  \frac{1}{\sqrt{h_{n}}}\Bigg(\int_{\mathbb{R}^{3}} \Big|\nabla_{x} \varphi\Big(\frac{t-t_{n}}{h_{n}^{2}},\frac{x-x_{n}}{h_{n}}\Big)\Big|^{2} \ dx\Bigg)^{\frac{1}{2}}
		& = &  \frac{1}{\sqrt{h_{n}}} \Bigg(\int_{\mathbb{R}^{3}} |\nabla_{x}\varphi(s,y)|^{2} \ h_{n}^{3}dy\Bigg)^{\frac{1}{2}}\\
		& = &		\|\nabla \varphi(s)\|_{L^{2}}.
	\end{eqnarray*}
\end{proof}
\subsection{Scales}
	On the Hilbert space $H^{1}(\mathbb{R}^{3})$, we define the self-adjoint operator $A$ by $Au = (-\Delta)^\frac{1}{2}u,$ with domain $D(A) = H^{2}(\mathbb{R}^{3})$.  The next definition is from \cite{GalGe}.
	\begin{definition}
	Let A be a self-adjoint (unbounded) operator on a Hilbert space $H$. Let $h_{n}$ be a sequence of positive numbers converging to 0. A bounded sequence $(u_{n})$ in $H$ is said to be $h_{n}$-oscillatory with respect to A if
\begin{equation}\label{eq 23}
		\limsup_{n \rightarrow \infty} \Big\|1_{|A| \geq \frac{R}{h_{n}}} u_{n}\Big\|_{H} \longrightarrow 0, \ \  R \rightarrow \infty,
\end{equation}
and strictly $h_{n}$-oscillatory with respect to $A$ if it satisfies \eqref{eq 23} and
$$
		\limsup_{n \rightarrow \infty} \Big\|1_{|A| \leq \frac{\epsilon}{h_{n}}} u_{n}\Big\|_{H} \longrightarrow 0, \ \  \epsilon \rightarrow 0.
$$
Moreover, $(u_{n})$ is said to be $h_{n}$-singular with respect to $A$ if
$$
		\Big\|1_{\frac{a}{h_{n}}\leq |A| \leq \frac{b}{h_{n}}} u_{n}\Big\|_{H} \longrightarrow 0, \ \  n \rightarrow \infty, \ \ \mbox{for all } a,b>0.
$$
\end{definition}

The next result ensures that the Schr\"odinger equation conserves $h_{n}$-oscillation.
	\begin{proposition}\label{prop3.4}
	Let $T > 0$. Let $\varphi_{n}$ be a bounded sequence in $ {H}^{1}(\mathbb{R}^{3})$ that is (strictly) $h_{n}$- oscillatory with respect to A. If $u_{n}$ is the solution of
	\begin{equation} \label{eq 26}
		\left\{
		\begin{array}{ll}
			i\partial_{t} u_{n} + \Delta u_{n}  = 0, \ \ &(t,x) \in [0,T] \times \mathbb{R}^{3} \\
			u_{n}(0) = \varphi_{n}, \ \ &x \in \mathbb{R}^{3} ,
		\end{array}
		\right.
	\end{equation}
	then, $(u_{n}(t))$ is (strictly) $h_{n}$-oscillatory with respect to A, uniformly on $[0, T]$. If $(\varphi_{n})$ is $h_{n}$-singular with respect to A, then $(u_{n}(t))$ is $h_{n}$-singular with respect to A, uniformly on $[0, T]$.
\end{proposition}
\begin{proof}
Consider the cut-off function $\chi \in C_{0}^{\infty}(\mathbb{R})$ such that $0 \leq \chi(s) \leq 1$ and $\chi (s) = 1$ for $|s| \leq 1$. The $h_{n}$-oscillation
	(respectively strict oscillation) is equivalent to
	$$\limsup_{n \rightarrow \infty} \Big\|\nabla (1-\chi)(\frac{h_{n}^{2}\Delta}{R^{2}})u_{n}\Big\|_{{L}^{2}} \longrightarrow 0, \ R \rightarrow \infty$$ (respectively $\limsup_{n \rightarrow \infty} \|\nabla \chi(R^{2}h_{n}^{2}\Delta)u_{n}\|_{{L}^{2}} \longrightarrow 0, \  R \rightarrow \infty$). Note that $v_{n} = (1-\chi)(\frac{h_{n}^{2}\Delta}{R^{2}})u_{n}$ is a solution of 
	\begin{equation*} 
		\left\{
		\begin{array}{lr}
			i\partial_{t} v_{n} + \Delta v_{n} = 0 , \\
			v_{n}(0) = (1-\chi)\Big(\frac{h_{n}^{2}\Delta}{R^{2}}\Big)\varphi_{n}, 
		\end{array}
		\right.
	\end{equation*}
	and the conservation of the energy gives $$\|\nabla v_{n}(t)\|_{{L}^{2}} =  \|\nabla v_{n}(0)\|_{L^{2}} \\
		 = \Big\|\nabla (1-\chi)\left(\frac{h_{n}^{2}\Delta}{R^{2}}\right)\varphi_{n}\Big\|_{{L}^{2}}.$$
	Therefore, taking the $\limsup$\ in $n$, we get the expected result uniformly in $0 \leq t \leq T$. Strict oscillation and singularity follow analogously.
\end{proof}

The following result gives us an estimation of Besov spaces. 

\begin{proposition} \label{prop3.5}
For every bounded sequence $(\varphi_{n})$ in ${H}^{1}(\mathbb{R}^{3})$, there exists $C_{T}>0$ such that 	\begin{equation*}
		\limsup_{n \rightarrow \infty} \|\nabla u_{n}\|_{L^{\infty}([0,T];\dot B^{0}_{2,\infty}(\mathbb{R}^{3}))} \leq C_{T} 	\limsup_{n \rightarrow \infty} \|\nabla \varphi_{n}\|_{\dot B^{0}_{2,\infty}(\mathbb{R}^{3})},
	\end{equation*}
	where $u_{n}$ is the solution of system \eqref{eq 26}.  Here, $\dot B^{0}_{2,\infty}(\mathbb{R}^{3})$ denotes the Besov space defined by
	$$\dot B^{0}_{2,\infty}(\mathbb{R}^{3}) =\Big\{u=u(x): \|u\|^{2}_{\dot B^{0}_{2,\infty}(\mathbb{R}^{3})} = \sup_{k \in \mathbb{Z}}\int_{2^{k} \leq |\xi|\leq 2^{k+1}} |\widehat{u}(\xi)|^{2} \ d\xi < +\infty\Big\}.$$
\end{proposition}
\begin{proof}
	Since  $u_{n}$ is the solution of system \eqref{eq 26}, the function $\sigma_{k}(D)u_{n}$ is also a solution to the same system, where $\sigma_{k}(\xi)=\textbf{1}_{2^{k}\leq |\xi|\leq 2^{k+1}}$. The conservation law for all $\sigma_{k}(D)u_{n}(t)$ gives
	$$\|\nabla u_{n}(t)\|_{\dot B^{0}_{2,\infty}(\mathbb{R}^{3})} = \|\nabla u_{n}(0)\|_{\dot B^{0}_{2,\infty}(\mathbb{R}^{3})} =  \|\nabla \varphi_{n}\|_{\dot B^{0}_{2,\infty}(\mathbb{R}^{3})}, \quad \text{for} \quad k \in \mathbb{Z},$$
	showing the result.
\end{proof}
	\subsection{Linear profile decomposition}
	The main result of this section is a combination of theories developed by Bahouri and Gerard \cite{bahouri}, Keraani \cite{keraani}, and Laurent \cite{Laurent2} and is given by the following theorem.
		\begin{theorem} \label{lineardecomp}
		Let $(v_{n})$ be a sequence of solutions to the  Schr\"{o}dinger  equation  \eqref{eq 26} on $[0,T]$ with an initial data $\varphi_{n}$, at time $t = 0$, bounded in $ H^{1}(\mathbb{R}^{3})$ and such that $\limsup_{n\rightarrow \infty} \|\varphi_{n}\|_{H^{1}} <\lambda_{0}$, where $\lambda_{0}$ was given in Definition \ref{def3}. Then, up to extraction, there exists a sequence of linear concentrating solutions $(\underline{p}^{(j)})$ associated
		to $[\varphi^{(j)},\underline{h}^{(j)},\underline{x}^{(j)},\underline{t}^{(j)}]$ such that, for any $l \in \mathbb{N}^{*}$,
$v_{n}(t,x) =  \sum_{j=1}^{l} p_{n}^{(j)}(t,x) + w_{n}^{(l)}(t,x)$ satisfies 
		\begin{equation} \label{eq 28}
 \ \limsup_{n \rightarrow \infty} \|w_{n}^{(l)}\|_{L^{\infty}_{t}L^{6}_{x}\cap L^{10}_{t}L^{10}_{x}} \longrightarrow0, \ l \rightarrow \infty,
		\end{equation}
for all $T>0,$ and
		\begin{equation} \label{eq 29}
			\|\nabla v_{n}\|_{L^{2}}^{2} =  \sum_{j=1}^{l} \|\nabla p_{n}^{(j)}\|_{L^{2}}^{2} + \|\nabla w_{n}^{(l)}\|_{L^{2}} ^{2} + o(1), \ \ n\rightarrow \infty.
		\end{equation}
		Moreover, we have $(\underline{h}^{(j)},\underline{x}^{(j)},\underline{t}^{(j)}) \perp (\underline{h}^{(k)},\underline{x}^{(k)},\underline{t}^{(k)})$  for any $j \neq k$. 
	\end{theorem}
	We split the proof of this theorem into four steps as follows. 
	\begin{proof}
	\vspace{0.2cm}
	\noindent\textit{\bf{Step 1. Extraction of scales.}} In this first part, we present the determination of the family of scales. The next result is paramount for our analysis and can be found in \cite[Proposition 3.4]{bahouri}.
	
	\begin{proposition} \label{prop3.8}
	If $(f_{n})$ is a bounded sequence in $L^{2}(\mathbb{R}^{3})$, then, up to a subsequence, there exists a family $(h_{n}^{j}) $ of pairwise orthogonal scales and a family $(g_{n}^{j})$ of bounded sequences in $L^{2}(\mathbb{R}^{3})$ such that
	\begin{enumerate}
		\item [i)] for every j, $g_{n}^{j}$ is $h_{n}^{j}$-oscillatory;
		\item[ii)] for every $l \geq 1$ and $x \in \mathbb{R}^{3}$,
		\begin{equation*}
			f_{n}(x) =  \sum_{j=1}^{l} g_{n}^{j}(x) + R_{n}^{l},
		\end{equation*}
		where $(R_{n}^{j})$ is $h_{n}^{j}$-singular for every $j \in {1, ..., l},$ and
		\begin{equation*}
			\limsup_{n \rightarrow \infty} \|R_{n}^{l}\|_{\dot{B}_{2,\infty}^{0}} \longrightarrow0, \ \ \ l \rightarrow \infty;
		\end{equation*}
		\item[iii)] for every $l \geq 1$,
		\begin{equation*}
			\|f_{n}\|_{L^{2}} =  \sum_{j=1}^{l} \|g_{n}^{j}\|^{2}_{L^{2}} + \|R_{n}^{l}\|^{2}_{L^{2}} + o(1), \ \ \ n\rightarrow \infty.
		\end{equation*}
	\end{enumerate}
\end{proposition}
	
	With this result in mind, let us present the following proposition. 
	\begin{proposition} \label{prop3.7}
		Let $T > 0$. Let $(\varphi_{n})$ be a bounded sequence in $ H^{1}(\mathbb{R}^{3})$ and $(v_{n})$ be the solution of system \eqref{eq 26}.  Then, up to an extraction, $v_{n}$ can be decomposed in the following way: for any $l \in \mathbb{N}^{*}$
		\begin{equation} \label{eq 31}
			v_{n}(t,x) =  \sum_{j=1}^{l} v_{n}^{(j)}(t,x) + \rho_{n}^{(l)}(t,x),
		\end{equation}
		where $v_{n}^{(j)}$ is a strictly $(h_{n}^{(j)})$-oscillatory solution of the linear Schrödinger equation \eqref{eq 26} on $\mathbb{R}^{3}$.
		The scales $h_{n}^{(j)}$ satisfy $h_{n}^{(j)} \rightarrow 0,$ $n\rightarrow \infty$, and are pairwise orthogonal. Additionally, we have
		\begin{equation} \label{eq 32}
			\limsup_{n \rightarrow \infty} \|\rho_{n}^{(l)}\|_{L^{\infty}([0,T];L^{6}(\mathbb{R}^{3}))\cap L^{10}([0,T];L^{10}(\mathbb{R}^{3}))} \longrightarrow 0, \ l\rightarrow \infty
		\end{equation}
		and
		\begin{equation}\label{eq 33}
			\|\nabla v_{n}(t)\|_{L^{2}}^{2} =  \sum_{j=1}^{l} \|\nabla v_{n}^{(j)}(t)\|_{L^{2}}^{2} + \|\nabla \rho_{n}^{(l)}(t)\|_{L^{2}} ^{2} + o(1), \ n\rightarrow \infty.
		\end{equation}
	\end{proposition}
\begin{proof}

	Applying Proposition \ref{prop3.8} to the sequence $(\nabla\varphi_{n})$, we obtain a family of scales $h_{n}^{(j)}$ and a family  $(\varphi_{n}^{(j)})$  of bounded sequences in $\dot{H}^{1}(\mathbb{R}^{3})$, such that
$
	\varphi_{n}(x) = \sum_{j=1}^{l} \varphi _{n}^{(j)}(x) + \Phi_{n}^{(l)}(x),
$
where $ \varphi_{n}^{(j)}$ is $h_{n}^{j}$-oscillatory with respect to $A$ for every $j \geq 1$. Moreover, $ \Phi_{n}^{(l)}$ is $h_{n}^{(j)}$-singular with respect to $A$ for every
$j \in {1, 2, ..., l},$ and
\begin{equation} \label{eq 34}
	\limsup_{n \rightarrow \infty} \|\nabla \Phi_{n}^{l}\|_{\dot{B}_{2,\infty}^{0}}\longrightarrow0, \ \ \ l\rightarrow \infty.
\end{equation}
Furthermore, the following almost orthogonality identity
\begin{equation*}
	\|\nabla \varphi_{n}\|^{2}_{L^{2}}  =  \sum_{j=1}^{l}\|\nabla \varphi_{n}^{(j)}\|^{2}_{L^{2}}  + \|\nabla \Phi_{n}^{(l)}\|^{2}_{L^{2}} + o(1), 
\end{equation*}
holds for all $ l \geq 1$, and the $h_{n}^{(j)}$ are pairwise orthogonal. This decomposition for the initial data can be extended to the solution
$
	v_{n}(t,x) =  \sum_{j=1}^{l} v_{n}^{(j)}(t,x) + \rho_{n}^{(l)}(t,x),
$
where each $v_{n}^{(j)}$ is a solution of 
\begin{equation*} 
	\left\{
	\begin{array}{ll}
		i\partial_{t} v_{n}^{(j)} + \Delta v_{n}^{(j)}  = 0,& \ (t,x) \in [0,T] \times \mathbb{R}^{3},\\
		v_{n}^{(j)}(0) = \varphi_{n}^{(j)}, &
	\end{array}
	\right.
\end{equation*}
and $\rho_{n}^{(l)}$ is a solution to the same system with initial data $\Phi_{n}^{(l)}$. 

Due to Proposition \ref{prop3.4}, each $v_{n}^{(j)}(t)$ is strictly $h_{n}^{(j)}$-oscillatory and $\rho_{n}^{(l)}(t)$ is $h_{n}^{(j)}$-singular for $1\leq j\leq l.$ So, $$\langle \nabla \rho_{n}^{(l)}(t),\nabla v_{n}^{(j)}(t) \rangle_{L^{2}} \longrightarrow 0,$$ as $n \rightarrow \infty$, uniformly in $[0,T]$. This is also true for the product between $v_{n}^{(j)}$ and $v_{n}^{k}$, $j \neq k$, by the orthogonality of the scales, i.e., $$\langle \nabla v_{n}^{(j)}(t),\nabla v_{n}^{(k)}(t) \rangle_{L^{2}} \longrightarrow 0, \quad n \rightarrow \infty.$$  Then, we get
\begin{equation*}
	\|\nabla v_{n}(t)\|_{L^{2}}^{2} =  \sum_{j=1}^{l} \|\nabla v_{n}^{(j)}(t)\|_{L^{2}}^{2} + \|\nabla \rho_{n}^{(l)}(t)\|_{L^{2}} ^{2} + o(1),
\end{equation*}
which is the desired equation \eqref{eq 33}. 

Let us now show \eqref{eq 32}. First of all, 
note that the convergence \eqref{eq 34} gives the convergence
to zero of $\nabla \rho_{n}^{(l)}(0) = \nabla \Phi_{n}^{(l)}$ in $\dot B^{0}_{2,\infty}$. We extend this convergence for all time using Proposition \ref{prop3.5} to get
$$\sup_{t \in [0,T]} \limsup_{n \rightarrow \infty} \| \nabla \rho_{n}^{(l)}(t)\|_{\dot B^{0}_{2,\infty}} \longrightarrow 0, \ l\rightarrow \infty.$$
Using \cite[Lemma 3.5]{bahouri}, we have
$$\limsup_{n \rightarrow \infty} \|\rho_{n}^{(l)}(t)\|_{L^{6}} \leq C \limsup_{n \rightarrow \infty} \|\nabla \rho_{n}^{(l)}(t)\|^{\frac{1}{3}}_{L^{2}} \limsup_{n \rightarrow \infty} \|\nabla \rho_{n}^{(l)}(t)\|^{\frac{2}{3}}_{\dot B^{0}_{2,\infty}}.$$
Observe that 
$$\|\nabla \rho_{n}^{(l)}(t)\|^{2}_{L^{2}} \leq \|\nabla v_{n}(t)\|^{2} _{L^{2}} \leq \|\nabla \varphi_{n}\|^{2} _{L^{2}} \leq C.$$
Therefore,
$$\limsup_{n \rightarrow \infty} \|\rho_{n}^{(l)}\|_{L^{\infty}_{t}L^{6}_{x}} \longrightarrow 0, \ \ l\rightarrow \infty.$$
Now, by an interpolation inequality, we obtain
$$\|\rho_{n}^{(l)}\|_{L^{10}_{t}L^{10}_{x}} \leq \|\rho_{n}^{(l)}\|_{L^{\infty}_{t}L^{6}_{x}} ^{\alpha} \|\rho_{n}^{(l)}\|_{L^{7}_{t}L^{14}_{x}} ^{\beta}.$$ 
Since $\Big(7,\frac{42}{17}\Big)$ is a $L^{2}$-admissible pair and by Sobolev's embedding, one has
$$
	\|\rho_{n}^{(l)}\|_{L^{7}_{t}L^{14}_{x}}  \leq 	\|\nabla\rho_{n}^{(l)}\|_{L^{7}_{t}L^{\frac{42}{17}}_{x}}  \leq 	\|\nabla e^{it\Delta} \Phi_{n}^{(l)}\|_{L^{7}_{t}L^{\frac{42}{17}}_{x}}  \leq \|\nabla \Phi_{n}^{(l)}\|_{L^{7}_{t}L^{\frac{42}{17}}_{x}} \leq  C \|\nabla \Phi_{n}^{(l)}\|_{L^{2}} ,
$$
which means
$$\limsup_{n \rightarrow \infty} \|\rho_{n}^{(l)}\|_{L^{10}_{t}L^{10}_{x}} \longrightarrow 0, \ \ l\rightarrow \infty.$$
This shows \eqref{eq 32} and completes the proof of Proposition \ref{prop3.7}.
\end{proof}

\noindent\textit{\bf{Step 2. Description of concentrating solutions.}} Now, we describe the “non-reconcentration” property for linear concentrating solutions. The main result can be read as follows.
	
	\begin{lemma}\label{lemma3.10} Let $\underline{v} = [\varphi,\underline{h},\underline{x},\underline{t}]$ a  linear concentrating solution and consider the interval $I = [-T,T]$ of $\mathbb{R}$ containing $t_{\infty}$.
	Set $I_{n}^{1,\Lambda} =[-T, t_{n} - \Lambda h_{n}]$ and $I_{n}^{3,\Lambda} =(t_{n} + \Lambda h_{n}, T]$.
	One has
	\begin{equation} \label{eq 35}
		\limsup_{n \rightarrow \infty}\| v_{n}\|_{L^{\infty}(I_{n}^{1,\Lambda}\cup I_{n}^{3,\Lambda},L^{6}(\mathbb{R}^{3}))} \longrightarrow 0, \mbox{  } \Lambda\rightarrow  \infty,
	\end{equation}
	and
	\begin{equation} \label{eq 36}
		\limsup_{n \rightarrow \infty}\| v_{n}\|_{L^{10}(I_{n}^{1,\Lambda}\cup I_{n}^{3,\Lambda},L^{10}(\mathbb{R}^{3}))} \longrightarrow 0, \mbox{  } \Lambda\rightarrow  \infty.
	\end{equation}
\end{lemma}
\begin{proof} Convergence \eqref{eq 36} follows directy from \eqref{eq 35} by interpolation.
To prove \eqref{eq 35}, we argue by contradiction: 
Suppose that \eqref{eq 35}  is not valid. In this case, there exists a constant $C>0$, a real subsequence $(\Lambda_{j})_{j}$ tending to $+ \infty$, and a subsequence $(t_{n_{j}})_{j}$ of $(t_{n})_{n}$ convergent to $\tau$ such that 
	\begin{equation}\label{eq 37}
		| t_{n_{j}} - t_{\infty}| > \Lambda_{j} h_{n_{j}} \mbox{ and } \lim_{j} \|v_{n_{j}}(t_{n_{j}},.)\|_{L^{6}(\mathbb{R}^{3})} \longrightarrow C.
	\end{equation}
Let us consider separately the cases $\tau \neq t_\infty$ and $\tau = t_\infty$. In case $\tau \neq t_{\infty}$, we have
	\begin{equation*}
		\left\{
		\begin{array}{ll}
			i\partial_{t} v_{n_{j}} + \Delta v_{n_{j}} = 0, \\
			v_{n_{j}}(t_{\infty}) =\frac{1}{ \sqrt{h_{n_{j}}}} \varphi\Big(\frac{x}{h_{n_{j}}}\Big).
		\end{array}
		\right.
	\end{equation*}
	Then,
	$$v_{n_{j}}(t,x) = e^{i(t-t_{\infty})\Delta} \frac{1}{ \sqrt{h_{n_{j}}}} \varphi\Big(\frac{x}{h_{n_{j}}}\Big),$$
	and so
	$$v_{n_{j}}(t_{n_{j}},x) = e^{i(t_{n_{j}}-t_{\infty})\Delta} \frac{1}{ \sqrt{h_{n_{j}}}} \varphi\Big(\frac{x}{h_{n_{j}}}\Big).$$
	By Definition \ref{def5}, we have 
	\begin{eqnarray*}
		\|v_{n_{j}}(t_{n_{j}},x) \|_{L^{6}} & \leq & C \Bigg(\int_{\mathbb{R}^{3}} \Bigg| (t_{n_{j}}-t_{\infty})^{-\frac{3}{2}} h_{n_{j}}^{\frac{5}{2}} \int_{\mathbb{R}^{3}} e^{i\frac{h_{n_{j}}^{2}|z|^{2}}{2(t_{n_{j}}-t_{\infty})}} \cdot e^{\frac{-i h_{n_{j}}\langle z,x \rangle}{(t_{n_{j}}-t_{\infty})}}  \varphi(z) \ dz \Bigg|^{6} \ dx\Bigg)^{\frac{1}{6}} \\
		& \leq &  (t_{n_{j}}-t_{\infty})^{-1} h_{n_{j}}^{2}   \Bigg(\int_{\mathbb{R}^{3}} \Bigg| \int_{\mathbb{R}^{3}} e^{i\frac{\tilde{h}_{j}^{2}|z|^{2}}{2(t_{n_{j}}-t_{\infty})}} \cdot e^{-i\langle z, w \rangle} \varphi(z) \ dz \Bigg|^{6} \ dw \Bigg)^{\frac{1}{6}} \\
		& \leq & (t_{n_{j}}-t_{\infty})^{-1} h_{n_{j}}^{2}   \Bigg(\int_{\mathbb{R}^{3}} |\hat{\varphi}(w)|^{6} \ dw \Bigg)^{\frac{1}{6}} \rightarrow 0, \ j \rightarrow \infty,
	\end{eqnarray*}
i.e., the right-hand side of this inequality converges to 0 as $j$ goes to $\infty$, which contradicts \eqref{eq 37}.
	
	Now, in case $\tau = t_{\infty}$, let $\varepsilon^{2}_{j} = |t_{\infty} - t_{n_{j}}|,$ $\tilde{h}_{j} = \frac{h_{n_{j}}}{\varepsilon_{j}}$, and define the sequence 
	$
	\tilde{f}_{j}(s,y) = \varepsilon_{j}^{\frac{1}{2}} v_{n_{j}}(t_{\infty}+ \varepsilon_{j}^{2}s, \varepsilon_{j} y).
	$
	Since $  |t_{\infty} - t_{n_{j}}| \geq \Lambda_{j} h_{n_{j}}$ and $ \lim_{j} \Lambda_{j} = +\infty$, one has $\lim_{j} \tilde{h}_{j} = 0.$ Moreover, the sequence $(\tilde{f}_{j})$ is the solution of
	\begin{equation*}
		\left\{
		\begin{array}{lr}
			i\partial_{s} \tilde{f}_{j} + \Delta_{y} \tilde{f}_{j} = 0, \\
			\tilde{f}_{j}(0) =\frac{1}{ \sqrt{\tilde{h}_{j}}} \varphi\Big(\frac{y}{\tilde{h}_{j}}\Big),
		\end{array}
		\right.
	\end{equation*}
Note that $\tilde{f}_{j}(1,y) $ is bounded, since
	\begin{equation*}
	\begin{split}
		\tilde{f}_{j}(1,y) =& e^{i\Delta} \frac{1}{ \sqrt{\tilde{h}_{j}}} \varphi\Big(\frac{y}{\tilde{h}_{j}}\Big) =\frac{1}{ \sqrt{\tilde{h}_{j}}} \int_{\mathbb{R}^{3}} e^{i\frac{|y-x|^{2}}{2}} \varphi\Big(\frac{x}{\tilde{h}_{j}}\Big) \ dx =  \tilde{h}_{j}^{\frac{5}{2}} \int_{\mathbb{R}^{3}} e^{i\frac{|\tilde{h}_{j}z-y|^{2}}{2}} \varphi(z) \ dz\\ =& \tilde{h}_{j}^{\frac{5}{2}} \int_{\mathbb{R}^{3}} e^{i\frac{\tilde{h}_{j}^{2}|z|^{2}}{2}} \cdot e^{-i\tilde{h}_{j}\langle z,y \rangle} \cdot e^\frac{i|y|^{2}}{2}\varphi(z) \ dz  \leq  C \Bigg| \tilde{h}_{j}^{\frac{5}{2}} \int_{\mathbb{R}^{3}} e^{i\frac{\tilde{h}_{j}^{2}|z|^{2}}{2}} \cdot e^{-i\tilde{h}_{j}\langle z,y \rangle} \varphi(z) \ dz \Bigg| . 
		\end{split}
	\end{equation*}
Therefore, 
	\begin{eqnarray*}
		\|\tilde{f}_{j}(1,y) \|_{L^{6}} & \leq & \Bigg(\int_{\mathbb{R}^{3}} \Bigg| \tilde{h}_{j}^{\frac{5}{2}} \int_{\mathbb{R}^{3}} e^{i\frac{\tilde{h}_{j}^{2}|z|^{2}}{2}} \cdot e^{-i\tilde{h}_{j}\langle z,y \rangle} \varphi(z) \ dz \Bigg|^{6} \ dy \Bigg)^{\frac{1}{6}} \\
		& \leq &  \tilde{h}_{j}^{2}   \Bigg(\int_{\mathbb{R}^{3}} \Bigg| \int_{\mathbb{R}^{3}} e^{i\frac{\tilde{h}_{j}^{2}|z|^{2}}{2}} \cdot e^{-i\langle z,x \rangle} \varphi(z) \ dz \Bigg|^{6} \ dx \Bigg)^{\frac{1}{6}} \\
		& \simeq & \tilde{h}_{j}^{2}   \Bigg(\int_{\mathbb{R}^{3}} |\hat{\varphi}(x)|^{6} \ dx \Bigg)^{\frac{1}{6}} \rightarrow 0, \ j \rightarrow \infty.
	\end{eqnarray*}
	Hence,
	$\|\tilde{f}_{j}(1,y)\|_{L^{6}} \rightarrow 0$, as $j\rightarrow \infty.$
	Therefore, since $	\|\tilde{f}_{j}(1,y) \|_{L^{6}}  = 	\|v_{nj}(t_{nj},.) \|_{L^{6}} ,$ 
	this contradicts \eqref{eq 37}, which finishes the proof of step 2. 
\end{proof}

\noindent\textit{\bf{Step 3. Extraction of times and cores of concentration.}}  Let $h_{n}$ be a fixed sequence in $\mathbb{R}^{*}_{+}$ converging to 0. 

Before presenting the main result of this step, we state and prove two auxiliary lemmas.
	\begin{lemma} \label{lemma3.13}
	Let $(\underline{x}^{(1)},\underline{t}^{(1)}) \not\perp_{h_{n}} (\underline{x}^{(2)},\underline{t}^{(2)}).$ Let $v_{n}$ be an (strictly) $h_{n}$-oscillatory sequence of solutions to the linear Schr\"{o}dinger equation such that
$
		D^{(1)}_{h_{n}} v_{n} \rightharpoonup \varphi^{(1)}$ as $ \ n\rightarrow \infty .
$
	There exists $\varphi^{(2)}$ such that
$
		D^{(2)}_{h_{n}} v_{n} \rightharpoonup \varphi^{(2)}$ as $\ n\rightarrow \infty .
$
	Moreover,
$
		\|\varphi^{(1)}\|_{L^{\infty}\dot{H}^{1}} = 	\|\varphi^{(2)}\|_{L^{\infty}\dot{H}^{1}}.
$
\end{lemma}
\begin{proof}
	Let $x_{n}^{(2)} = x_{n}^{(1)} + (\overrightarrow{D} + o(1))h_{n}$ and $t_{n}^{(2)} = t_{n}^{(1)} + (\overrightarrow{C} + o(1))h_{n}^{2}$, where $\overrightarrow{D} \in \mathbb{R}^{3}$, $\overrightarrow{C} \in \mathbb{R}$ are constants. We have
	$\sqrt{h}v_{n}(t_{n}^{(1)} + h_{n}^{2}s,x_{n}^{(1)}+ h_{n}y) \rightharpoonup \varphi^{(1)}(s,y), \ s \in \mathbb{R}.$
	Then,
	\begin{equation*}
	\begin{split}
		\sqrt{h_{n}}v_{n}(t_{n}^{(2)} + h_{n}^{2}s,x_{n}^{(2)}+ h_{n}y)  = & \sqrt{h_{n}} v_{n}(t_{n}^{(1)} + (\overrightarrow{C}+ o(1))h_{n}^{2} + h_{n}^{2}s,x_{n}^{(1)}\\ &+ (\overrightarrow{D} + o(1))h_{n} + h_{n}y) \\
		 = & \sqrt{h_{n}} v_{n}(t_{n}^{(1)} + (\overrightarrow{C} + s)h_{n}^{2},x_{n}^{(1)}+ (\overrightarrow{D} + y)h_{n}) \\ 
		 & \mbox{ } \rightharpoonup \varphi^{(1)}(\overrightarrow{C}+ s,\overrightarrow{D} + y), \ (s+ \overrightarrow{C}) \in \mathbb{R}.
	\end{split}
	\end{equation*}
	Taking $\varphi^{(1)}(\overrightarrow{C}+ s,\overrightarrow{D} + y) = \varphi^{(2)}(s,y),$ we have 
	$$D_{h_{n}}^{(2)}v_{n} \rightharpoonup \varphi^{(2)}, \ \ s \in \mathbb{R}.$$
	Moreover,
		$$\|\nabla \varphi^{(2)}(s)\|_{L^{2}} = \|\nabla \varphi^{(1)}(s + \overrightarrow{C})\|_{L^{2}} \leq \sup_{s' \in \mathbb{R}}\| \nabla \varphi^{(1)}(s')\|_{L^{2}} = \|\nabla \varphi^{(1)}(s)\|_{L^{\infty}L^{2}}, $$
and
$$
			\|\nabla \varphi^{(1)}(s+ \overrightarrow{C})\|_{L^{2}} =  \|\nabla \varphi^{(2)}(s)\|_{L^{2}}  \leq  \sup_{s \in \mathbb{R}}\|\nabla \varphi^{(2)}(s)\|_{L^{2}} =  \|\nabla \varphi^{(2)}(s)\|_{L^{\infty}L^{2}},
$$
		showing the lemma.
\end{proof}

The second lemma is the following one, where we keep
the notation of the construction that allowed us to extract the scales and cores.

\begin{lemma} \label{lemma3.14}
	Let $\{j,j'\} \in \{1,...,K\}^{2}$ be such that $$(\underline{x}^{(j)},\underline{t}^{(j)}) \not \perp_{h_{n}} (\underline{x}^{(K+1)},\underline{t}^{(K+1)}) \ \mbox{     and    } \ (\underline{x}^{(j)},\underline{t}^{(j)}) \perp_{h_{n}} (\underline{x}^{(j')},\underline{t}^{(j')}).$$
	If $D_{h_{n}}^{(K+1)}w_{n}^{(K+1)}\rightharpoonup 0$, then $D_{h_{n}}^{(j)}w_{n}^{(K+1)} \rightharpoonup 0.$ Moreover, $D_{h_{n}}^{(j)}p_{n}^{(j')} \rightharpoonup 0$ for any concentrating solution $p_{n}^{(j')}$ associated with $[\varphi^{(j')},\underline{h},\underline{x}^{(j')},\underline{t}^{(j')}]$.
\end{lemma}
\begin{proof}
	The first part of this lemma is a particular case of Lemma \ref{lemma3.13}. So, it remains to show that  $D_{h_{n}}^{(j)}p_{n}^{(j')} \rightharpoonup 0$
	or, equivalently, 
	$$\sqrt{h_{n}} p_{n}^{(j')}(t_{n}^{(j)}+ h_{n}^{2}s,x_{n}^{(j)} + h_{n}y) \rightharpoonup 0 \ \mbox{ in } \dot{H}^{1}(\mathbb{R}^{3}).$$
Since $p_{n}^{(j')}$ is a concentrating solution associated to $[\varphi ^{(j')},\underline{h},\underline{x}^{(j')},\underline{t}^{(j')}]$, we have
	$$p_{n}^{(j')}(t,x) = \frac{1}{\sqrt{h_{n}}} \varphi^{(j')} \Big(\frac{t-t_{n}^{(j')}}{h_{n}^{2}} ,\frac{x-x_{n}^{(j')}}{h_{n}}\Big),$$
	and 
	$$\sqrt{h_{n}} p_{n}^{(j')} (h_{n}^{2}s,x_{n}^{(j)} + h_{n}y) = \varphi^{(j')}\Big(\frac{t_{n}^{(j)}-t_{n}^{(j')}}{h_{n}^{2}} + s,\frac{x_{n}^{(j)}-x_{n}^{(j')}}{h_{n}} + y \Big).$$
Assuming $(\underline{x}^{(j)},\underline{t}^{(j)}) \perp_{h_{n}}  (\underline{x}^{(j')},\underline{t}^{(j')})$, without loss of generality, let us assume that $\varphi^{(j')}$ is continuous and compactly supported. Thus,
	\begin{equation*}
	\begin{split}
		\int_{\mathbb{R}^{3}}& \nabla \sqrt{h_{n}} p_{n}^{(j')}(t_{n}^{(j)}+ h_{n}^{2}s,x_{n}^{(j)} + h_{n}y) \cdot \nabla \psi(y) \ dy =\\ 	&\int_{\mathbb{R}^{3}} \nabla \varphi^{(j')}\Big(\frac{t_{n}^{(j)}-t_{n}^{(j')}}{h_{n}^{2}} + s,\frac{x_{n}^{(j)}-x_{n}^{(j')}}{h_{n}} + y \Big)\cdot \nabla \psi(y) \ dy,
		\end{split}
	\end{equation*}
	which tends to 0 as $n$ tends to $\infty$ if $\Big|\frac{t_{n}^{(j)} - t_{n}^{(j')}}{h_{n}^{2}}\Big| \rightarrow \infty$ or  $\Big|\frac{x_{n}^{(j)}- x_{n}^{(j')}}{h_{n}^{2}}\Big| \rightarrow \infty$, since $\varphi^{(j')}$ is compactly supported. This proves the lemma.
\end{proof}

Now, we prove the main result of this step.  Precisely, the following proposition will ensure the profile decomposition for $h_{n}$-oscillatory sequences. 

	\begin{proposition} \label{prop3.11} 
	Let $(v_{n})_{n \in \mathbb{N}}$ be an (strictly) $h_{n}$-oscillatory sequence of solutions to the linear Schr\"{o}dinger equation \eqref{eq 26}. Then, up a subsequence, there exist linear concentrating solutions $p_{n}^{k}$, as defined in Definition \ref{def5}, associated to $[\varphi^{(k)},\underline{h},\underline{x}^{(k)},\underline{t}^{(k)}]$ such that for any $l \in \mathbb{N}^{*}$, one has
	\begin{equation} \label{eq 38} 
		v_{n}(t,x) = \sum_{j=1}^{l} p_{n}^{(j)}(t,x) + w_{n}^{(l)}(t,x),
	\end{equation}
	\begin{equation} \label{eq 39}
  \limsup_{n \rightarrow \infty} \|w_{n}^{(l)}\|_{L^{\infty}([0,T];L^{6}(\mathbb{R}^{3})} \longrightarrow0, \ l\rightarrow \infty,
	\end{equation}
for all $T>0$, and
	\begin{equation}\label{eq 40}
		\|\nabla v_{n}(t)\|_{L^{2}}^{2} = \sum_{j=1}^{l} \|\nabla p_{n}(t)^{(j)}\|_{L^{2}}^{2} + \|\nabla w_{n}(t)^{(l)}\|_{L^{2}} ^{2} + o(1), \ n\rightarrow \infty,
	\end{equation}
	for $t \in [0, T ]$.
	Moreover, for any $j \neq k$, we have $(\underline{x}^{(k)},\underline{t}^{(k)}) \perp (\underline{x}^{(j)},\underline{t}^{(j)})$.
\end{proposition}
\begin{proof} Using the notation of Definition \ref{def6}, if $v_{n} \in L^{\infty}([0,T],\dot{H}^{1}(\mathbb{R}^{3}))$, consider $\widetilde{v}_{n}$ its extension in $\mathbb{R}$ by zero outside $[0,T]$ and denote
	$$ \delta(\underline{v}) = \sup_{(t_{n}, x_{n})} \Big\{\|\nabla \varphi(0)\|^{2}_{L^{2}}; D_{h_{n}}\widetilde{v}_{n} \rightharpoonup \varphi, \mbox{up to a subsequence},  \ \varphi \in L^{\infty}(\mathbb{R};\dot{H}^{1}(\mathbb{R}^{3})) \Big\},$$
	where $(t_{n},x_{n})$ are sequences in $[0,T] \times \mathbb{R}^{3}$ and this means that $h_{n}^{\frac{1}{2}}\widetilde{v}_{n}(t_{n} + h_{n}^{2}t,x_{n} + h_{n}x) \rightharpoonup \varphi(t,x)$ in $\dot{H}(\mathbb{R}^{3})$. 
	
	So, in this scenario, we consider $\varphi$ some type of weak limit of the translated sequence $\widetilde{v}_{n}$. Let
	$p_n$ be a linear concentrating solution associated to $\varphi$ $$p_{n}(t,x) = \frac{1}{\sqrt{h_{n}}} \varphi \Big(\frac{t-t_{n}}{h_{n}^{2}},\frac{x-x_{n}}{h_{n}}\Big)$$ and $\widetilde{p}_{n}$ be its extension in $\mathbb{R}$ by zero outside $[0,T]$. 
	Let $\mathcal{V}(v_{n})$ be the set of such functions $\varphi$. 	If $\delta(\underline{v})=0$, we take $p_{n}^{(j)}= \widetilde{p}_{n}^{(j)}=0$, for all $j$. 
	If $\delta(\underline{v})>0$, we choose $\varphi^{(1)} \in \mathcal{V}(v_{n})$ such that 
	$$\|\nabla \varphi^{(1)}(0)\|_{L^{2}} \geq \frac{1}{2} \delta(\underline{v}) >0.$$
	This means that there exists $(\underline{x}^{(1)},\underline{t}^{(1)}) \in [0,T] \times \mathbb{R}^{3} \rightarrow (x_{\infty}^{(1)},t_{\infty}^{(1)})$ satisfying $D_{h_{n}}\widetilde{v}_{n} \rightharpoonup \varphi^{(1)}$ as $n\rightarrow \infty.$ Equivalently 
		$$\sqrt{h_{n}} \widetilde{v}_{n}(t_{n}^{(1)} + h_{n}^{2}s,x_{n}^{(1)}+ h_{n}y) \rightharpoonup \varphi^{(1)}(s,y), \  s \in \mathbb{R} , \mbox{as}\ n\rightarrow \infty.$$
	Now, choose $p_{n}^{(1)}$ as the linear concentrating solution associated with $[\varphi^{(1)}, \underline{h},\underline{x}^{(1)},\underline{t}^{(1)}]$ and let $\widetilde{p}_{n}^{(1)}$ be its the extension to $\mathbb{R}$ by zero outside $[0,T]$. Note that the assumption $t_{n}^{(1)} \in  [0, T ]$ ensures $t_{\infty}^{(1)} \in  [0, T ]$, which will always be the case for all the concentrating solutions we consider. 
	
	To proceed, we first state a lemma that will be used for the orthogonality of energies.
	\begin{lemma}\label{lemma3.12}
		Let $w_{n}^{(1)} = \widetilde{v}_{n} - \widetilde{p}_{n}^{(1)}$. One has
		$$
		\|\nabla \widetilde{v}_{n}(t)\|^{2}_{L^{2}} = \|\nabla \widetilde{p}_{n}^{(1)}(t)\|^{2}_{L^{2}} + \|\nabla w_{n}^{(1)}(t)\|^{2}_{L^{2}} + o(1) \mbox{as} \ n \rightarrow \infty.
		$$
	\end{lemma}
	\begin{proof}
		Observe that
		\begin{equation*}
			\begin{split}
				\sqrt{h_{n}} w_{n}^{(1)}(t_{n}^{(1)} + h_{n}^{2}s,x_{n}^{(1)}+ h_{n}y)  = & \sqrt{h_{n}} \widetilde{v}_{n}(t_{n}^{(1)} + h_{n}^{2}s,x_{n}^{(1)}+ h_{n}y) \\ &- \sqrt{h_{n}}\widetilde{p}_{n}^{(1)}(t_{n}^{(1)} + h_{n}^{2}s,x_{n}^{(1)}+ h_{n}y) \\
				= & \sqrt{h_{n}} \widetilde{v}_{n}(t_{n}^{(1)} + h_{n}^{2}s,x_{n}^{(1)}+ h_{n}y) - \varphi^{(1)}(s,y) \rightharpoonup 0, \ n\rightarrow\infty,
			\end{split}
		\end{equation*}
		which means that $D_{h_{n}} w_{n}^{(1)} \rightharpoonup 0.$ Then,
		$$\|\nabla \widetilde{v}_{n}(t)\|^{2}_{L^{2}} = \|\nabla w_{n}^{(1)}(t)\|^{2}_{L^{2}} + 2\langle \nabla  w_{n}^{(1)}(t),\nabla \widetilde{p}_{n}^{(1)}(t) \rangle + \|\nabla \widetilde{p}_{n}^{(1)}(t)\|^{2}_{L^{2}}.$$
		A change of variables yields
		\begin{eqnarray*}
			\langle \nabla w_{n}^{(1)}(t),\nabla \widetilde{p}_{n}^{(1)}(t) \rangle  & = & \int_{\mathbb{R}^{3}} \nabla_{x} w_{n}^{(1)}(t,x) \cdot \nabla_{x} \widetilde{p}_{n}^{(1)}(t,x) \ dx \\
			& = & \int_{\mathbb{R}^{3}} \nabla_{x} w_{n}^{(1)}(t,x) \cdot \nabla_{x} \frac{1}{\sqrt{h_{n}}} \varphi^{(1)} \Big(\frac{t-t_{n}^{(1)}}{h_{n}^{2}},\frac{x-x_{n}^{(1)}}{h_{n}}\Big) \ dx \\
			& = & 	\int_{\mathbb{R}^{3}} \nabla_{x} w_{n}^{(1)}(t_{n}^{(1)} + h_{n}^{2}s,x_{n}^{(1)}+ h_{n}y) \cdot \nabla_{x} \frac{1}{\sqrt{h_{n}}} \varphi^{(1)} (s,y) \ h_{n}^{3} dy \\
			& = & 	\int_{\mathbb{R}^{3}} \nabla_{y} \sqrt{h_{n}}w_{n}^{(1)}(t_{n}^{(1)} + h_{n}^{2}s,x_{n}^{(1)}+ h_{n}y) \cdot \nabla_{y}  \varphi^{(1)} (s,y) \ dy,
		\end{eqnarray*}
		which goes to $ 0$, as  $ n\rightarrow\infty$, proving Lemma \ref{lemma3.12}.
	\end{proof}

	The previous lemma ensures that we can get the expansion of $v_{n}$ announced in Proposition \ref{prop3.11} by induction iterating the same process. To this end, let us assume that
	\begin{equation*}
		\widetilde{v}_{n}(t,x) = \sum_{j=1}^{l} \widetilde{p}_{n}^{(j)}(t,x) + w_{n}^{(l)}(t,x).
	\end{equation*}
Hence,
	\begin{equation*}
	v_{n}(t,x) = \sum_{j=1}^{l} p_{n}^{(j)}(t,x) + w_{n}^{(l)}(t,x),
\end{equation*}
and
$$
		\|\nabla v_{n}(t)\|_{L^{2}}^{2} = \sum_{j=1}^{l} \|\nabla p_{n}^{(j)}(t)\|_{L^{2}}^{2} + \|\nabla w_{n}^{(l)}(t)\|_{L^{2}}^{2} + o(1), \ n\rightarrow \infty,
$$
	where $p_{n}^{(j)}$ is a linear concentrating solution associated with $[\varphi^{(j)},\underline{h},\underline{x}^{(j)},\underline{t}^{(j)}]$, which are mutually orthogonal due to Lemma \ref{lemma3.12}. We now argue as before:  If $\delta(\underline{w}^{(l)}) = 0$, we just choose $p_n^{(l+1)} = 0$. If  $\delta(\underline{w}^{(l)}) >0$, choose $[\varphi^{(l+1)},\underline{x}^{(l+1)},\underline{t}^{(l+1)}]$ such that
	\begin{equation} \label{eq 42}
		\|\nabla \varphi^{(l+1)}(0)\|^{2}_{L^{2}} \geq \frac{1}{2} \delta(\underline{w}^{(l)}),
	\end{equation}
	and
	$$D_{h_{n}}w_{n}^{(l)} \rightharpoonup \varphi^{(l+1)}, \mbox{as}\ n\rightarrow \infty.$$
	Define $p_{n}^{(l+1)}$ as a linear concentrating solution associated to $[\varphi^{(l+1)},\underline{h},\underline{x}^{(l+1)},\underline{t}^{(l+1)}]$. Again, Lemma \ref{lemma3.12} applied to $w_{n}^{(l)}$ and $\widetilde{p}_{n}^{(l+1)}$ gives  \eqref{eq 40} with $w_{n}^{(l+1)} = w_{n}^{(l)} - \widetilde{p}_{n}^{(l+1)}$.
	
	Let us now show the convergence \eqref{eq 39}. Using Lemma \ref{lemma3.3} and energy estimates, we have
	$$\|\nabla \varphi^{(j)}(0)\|^{2}_{L^{2}} = \|\nabla p_{n}^{(j)}(t_{n}^{(j)})\|^{2}_{L^{2}} = \|\nabla p_{n}^{(j)}(0)\|^{2}_{L^{2}}.$$
	Using \eqref{eq 40}, we have that, for some $C(T)=C > 0$,
	\begin{eqnarray*}
		\sum_{j=1}^{l} \|\nabla \varphi^{(j)}(0)\|^{2}_{L^{2}} & = &  \sum_{j=1}^{l} \|\nabla p_{n}^{(j)}(0)\|^{2}_{L^{2}} \leq  \limsup_{n \rightarrow \infty} \|\nabla v_{n}(0)\|^{2}_{L^{2}} \leq  C.
	\end{eqnarray*}
	So, the series with general term $\|\nabla \varphi^{(j)}(0)\|^{2}_{L^{2}}$ converges and, therefore, 
	$$\|\nabla \varphi^{(j)}(0)\|^{2}_{L^{2}} \rightarrow 0, \mbox{as} \ l \rightarrow \infty.$$
	Using estimate \eqref{eq 42}, one obtains
	$$\delta(\underline{w}^{(l)}) \rightarrow 0, \mbox{as} \ l\rightarrow \infty.$$

	To show that
$$\limsup_{n\rightarrow \infty} \|w_{n}^{(l)}\|_{L^{\infty}_{t}L^{6}_{x}} \longrightarrow 0, \ l\rightarrow \infty, $$ 
introduce a family of functions $\chi_{R}(t, x)= \chi^{1}_{R}(t) \cdot \chi^{2}_{R}(x) \in C_{0}^{\infty}(\mathbb{R} \times \mathbb{R}^{3})$ satisfying the following properties:
\begin{equation*}
	\left\{
	\begin{array}{ll}
		|\widetilde{\chi_{R}^{1}}| + |\widehat{\chi_{R}^{2}}| \leq 2, \quad \text{with}\quad \supp(\widehat{\chi_{R}^{2}}) \subset \Big\{\frac{1}{2Rh_{n}} \leq |\xi| \leq \frac{2R}{h_{n}}\Big\}; \\
		\widehat{\chi_{R}^{2}}(\xi) \equiv 1,\quad \mbox{ for } \Big\{\frac{1}{Rh_{n}} \leq |\xi| \leq \frac{R}{h_{n}}\Big\}; \\
		\widetilde{\chi_{R}^{1}} \Big(|\xi|^{2}\Big) = 1,\quad \mbox{ on } 	\supp(\widehat{\chi_{R}^{2}}); \\
		\supp(\chi^{1}_{R}) \subset [-T,0],	 
	\end{array}
	\right.
\end{equation*}
where \ $\widetilde{   }$ \ and \ $\widehat{   }$ \ denote de Fourier transform in time and space, respectively. One has 
\begin{equation} \label{eq 43}
	\|w_{n}^{(l)}\|_{L^{\infty}([0,T];L^{6}(\mathbb{R}^{3}))} \leq \| \chi_{R} \ast w_{n}^{(l)}\|_{L^{\infty}([0,T];L^{6}(\mathbb{R}^{3}))} + \|(\delta - \chi_{R}) \ast w_{n}^{(l)}\|_{L^{\infty}([0,T];L^{6}(\mathbb{R}^{3}))}
\end{equation}
where $\ast$ denotes the convolution in $(t, x)$ and $\delta$ denotes the Dirac distribution. Let us bound each term on the right-hand side of inequality \eqref{eq 43}. 

\vspace{0.1cm}
\noindent\textit{1. Bound for $\| \chi_{R} \ast w_{n}^{(l)}\|_{L^{\infty}([0,T];L^{6}(\mathbb{R}^{3}))} $}.
\vspace{0.1cm}

Note that
$$
	\| \chi_{R} \ast w_{n}^{(l)}\|_{L^{\infty}([0,T];L^{6}(\mathbb{R}^{3}))} \leq \| \chi_{R} \ast w_{n}^{(l)}\|^{\frac{1}{3}}_{L^{\infty}([0,T];L^{2}(\mathbb{R}^{3}))} \cdot \| \chi_{R} \ast w_{n}^{(l)}\|^{\frac{2}{3}}_{L^{\infty}([0,T] \times \mathbb{R}^{3})}.
$$
The function $\chi_{R} \ast w_{n}^{(l)}$ is a solution to the linear Schr\"odinger equation \eqref{eq 26} on $\mathbb{R}$ and, in particular, the $L^{2}$-conservation law gives
\begin{equation} \label{eq 45}
\begin{split}
	\| \chi_{R} \ast w_{n}^{(l)}\|^{2}_{L^{\infty}([0,T];L^{2}(\mathbb{R}^{3}))}  = & \| (\chi_{R} \ast w_{n}^{(l)})(0)\|^{2}_{L^{2}_{x}} = \frac{1}{(2 \pi)^{3} } \|\mathfrak{F}_{x \rightarrow \xi}((\chi_{R} \ast w_{n}^{(l)})(0))(\xi)\|^{2}_{L^{2}_{\xi}}.
	\end{split}
\end{equation}
On the other hand, we write
\begin{equation*}
	(\chi_{R} \ast w_{n}^{(l)})(0,x) = \int_{\mathbb{R}} \chi_{R}^{1}(-s) \int_{\mathbb{R}^{3}}  \chi_{R}^{2}(x-y) w_{n}^{(l)}(s,y) \ dyds.
\end{equation*}
By the Plancherel inversion formula, we get
\begin{eqnarray*}
	(\chi_{R} \ast w_{n}^{(l)})(0,x)  & = & \frac{1}{(2 \pi)^{3} } \int_{\mathbb{R}} \chi_{R}^{1}(-s) \int_{\mathbb{R}^{3}} \chi_{R}^{2}(x-y) \int_{\mathbb{R}^{3}} e^{iy\xi}\widehat{w_{n}^{(l)}(s)}(\xi) e^{-ix\xi}e^{ix\xi} \ d\xi dyds\\
	& = & \frac{1}{(2 \pi)^{3} } \int_{\mathbb{R}} \chi_{R}^{1}(-s) \int_{\mathbb{R}^{3}} e^{-i(x-y)\xi}\chi_{R}^{2}(x-y) \int_{\mathbb{R}^{3}} \widehat{w_{n}^{(l)}(s)}(\xi) e^{ix\xi} \ d\xi dyds\\
	& = & \frac{1}{(2 \pi)^{3} } \int_{\mathbb{R}} \chi_{R}^{1}(-s) \int_{\mathbb{R}^{3}} \widehat{\chi^{2}_{R}}(\xi)\widehat{w_{n}^{(l)}(s)}(\xi)  e^{ix\xi} \ d\xi ds.
\end{eqnarray*}
Since $\widehat{w_{n}^{(l)}(s)}(\xi) = e^{is|\xi|^{2}} \widehat{w_{n}^{(l)}(0)}(\xi)$, we obtain
\begin{eqnarray*}
	(\chi_{R} \ast w_{n}^{(l)})(0,x) & = &  \frac{1}{(2 \pi)^{3} } \int_{\mathbb{R}} \chi_{R}^{1}(-s) \int_{\mathbb{R}^{3}} \widehat{\chi^{2}_{R}}(\xi)\widehat{w_{n}^{(l)}(0)}(\xi) e^{is|\xi|^{2}} e^{ix\xi} \ d\xi ds \\
	& = & \frac{1}{(2 \pi)^{3} }  \int_{\mathbb{R}^{3}} \widetilde{\chi_{R}^{1}}\Big(|\xi|^{2}\Big) \widehat{\chi^{2}_{R}}(\xi)\widehat{w_{n}^{(l)}(0)}(\xi)  e^{ix\xi} \ d\xi \\
	& = & \mathfrak{F}^{-1}_{\xi \rightarrow x} \Big[\widetilde{\chi_{R}^{1}}\Big(|\xi|^{2}\Big) \widehat{\chi^{2}_{R}}(\xi)\widehat{w_{n}^{(l)}(0)}(\xi)\Big](x).
\end{eqnarray*}
Consequently
\begin{equation} \label{eq 46}
	\mathfrak{F}_{x \rightarrow \xi}((\chi_{R} \ast w_{n}^{(l)})(0))(\xi) = \widetilde{\chi_{R}^{1}}\Big(|\xi|^{2}\Big) \widehat{\chi^{2}_{R}}(\xi)\widehat{w_{n}^{(l)}(0)}(\xi).
\end{equation}
Using the properties of 
\eqref{eq 45} and \eqref{eq 46}, we get
\begin{equation*} 
\begin{split}
	\| \chi_{R} \ast w_{n}^{(l)}\|^{2}_{L^{\infty}([0,T];L^{2}(\mathbb{R}^{3}))}   = & \frac{1}{(2 \pi)^{3} } \Big\|\widetilde{\chi_{R}^{1}}\Big(|\xi|^{2}\Big) \widehat{\chi^{2}_{R}}(\xi)\widehat{w_{n}^{(l)}(0)}(\xi)\Big\|^{2}_{L^{2}_{\xi}}  \\
	\leq &  C\frac{1}{(2 \pi)^{3} } \int_{\frac{1}{2Rh_{n}} \leq |\xi| \leq \frac{2R}{h_{n}}} |\widehat{\chi^{2}_{R}}(\xi)\widehat{w_{n}^{(l)}(0)}(\xi)|^{2} \ d\xi  \\
	 \leq & C_{1}(R)h_{n}^{2}\|\xi \widehat{w_{n}^{(l)}(0)}\|^{2}_{L^{2}} \\ 
	 \leq & C_{1}(R) h_{n}^{2} \|\nabla w_{n}^{(l)}(0)\|^{2}_{L^{2}_{x}},
	 \end{split}
\end{equation*}
where $C_{1}$ is an $R$-dependent constant. Now, observe that
$$\limsup_{n\rightarrow \infty} \| \chi_{R} \ast w_{n}^{(l)}\|_{L^{\infty}([0,T] \times \mathbb{R}^{3})} = \sup_{(t_{n},x_{n})} \limsup_{n \rightarrow \infty}\Big|(\chi_{R}\ast w_{n}^{(l)})(t_{n},x_{n})\Big|.$$
Let $\varphi \in \mathcal{V}(w_{n}^{(l)})$ such that
$\sqrt{h_{n}}w_{n}^{(l)}(t_{n} + h_{n}^{2}s, x_{n} + h_{n}y) \rightharpoonup \varphi(s,y)$ and $\widetilde{p}_{n}$ be the rescaled function $\widetilde{p_{n}}(t,x) = \frac{1}{\sqrt{h_{n}}} \varphi \Big(\frac{t}{h_{n}^{2}},\frac{x}{h_{n}}\Big)$. We have that $\widetilde{p}_{n}$ satisfies the linear Schr\"{o}dinger equation  and 
$$w_{n}^{(l)}(t_{n} + t, x_{n} + x) \rightharpoonup \widetilde{p}_{n}(t,x).$$
Hence,
$$(\chi_{R}\ast w_{n}^{(l)})(t_{n}+t,x_{n} + x) \rightharpoonup(\chi_{R} \ast \tilde{p}_{n})(t,x)$$
and
$$(\chi_{R}\ast w_{n}^{(l)})(t_{n},x_{n}) \rightharpoonup(\chi_{R} \ast \tilde{p}_{n})(0,0).$$
Thus, 
\begin{eqnarray*}
	\limsup_{n\rightarrow \infty} \| \chi_{R} \ast w_{n}^{(l)}\|_{L^{\infty}([0,T] \times \mathbb{R}^{3})} 
	& \leq & \sup \Bigg\{ \Bigg|\int_{\mathbb{R}} \int_{\mathbb{R}^{3}} \chi_{R}(-t,-x)\widetilde{p}_{n}(t,x) \ dxdt\Bigg| \Bigg\}. \\ 
\end{eqnarray*}
Therefore, by H\"{o}lder's inequality, it follows that
\begin{equation*}
	\limsup_{n\rightarrow \infty} \| \chi_{R} \ast w_{n}^{(l)}\|_{L^{\infty}([0,T] \times \mathbb{R}^{3})}  \leq C_{2}(R) \sup \Big\{\|\widetilde{p}_{n}\|_{L^{\infty}_{t}L^{6}_{x}}\Big\},
\end{equation*}
where $C_{2}(R) = \|\chi_{R}\|_{L^{1}([0,T];L^{{\frac{6}{5}}}(\mathbb{R}^{3}))}$. Since
$$\|\widetilde{p}_{n}(t)\|_{L^{6}_{x}} \leq \|\widetilde{p}_{n}(t)\|_{\dot{H}^{1}_{x}}  = \|\widetilde{p}_{n}(0)\|_{\dot{H}^{1}_{x}} =  \|\varphi(0)\|_{\dot{H}^{1}_{x}}  \leq C \delta(w_{n}^{(l)}),$$
it follows that
\begin{equation*}
	\| \chi_{R} \ast w_{n}^{(l)}\|_{L^{\infty}([0,T] \times \mathbb{R}^{3})}   \leq C_{2}(R) \delta(w_{n}^{(l)})
\end{equation*}
for every $l \geq 1$.
Putting  these estimates together, we conclude that
\begin{equation} \label{eq 48}
\begin{split}
	\| \chi_{R} \ast w_{n}^{(l)}\|_{L^{\infty}([0,T];L^{6}(\mathbb{R}^{3}))} \leq &  C_{1}(R) h_{n}^{\frac{1}{3}}\|\nabla w_{n}^{(l)}\|_{L^{2}}^{\frac{1}{3}} \cdot C_{2}(R) \delta(w_{n}^{(l)})^{\frac{2}{3}}\leq  C(R) h_{n}^{\frac{1}{3}}\delta(w_{n}^{(l)})^{\frac{2}{3}},
	\end{split}
\end{equation}
which is the desired bound.

\vspace{0.2cm}

\noindent\textit{2. Bound for $\|(\delta - \chi_{R}) \ast w_{n}^{(l)}\|_{L^{\infty}([0,T];L^{6}(\mathbb{R}^{3}))}$}.
\vspace{0.1cm}

The function $(\delta - \chi_{R}) \ast w_{n}^{(l)}$ is a solution to the linear Schr\"{o}dinger equation in $\mathbb{R}$. Therefore, 
$$\|(\delta - \chi_{R}) \ast w_{n}^{(l)}\|^{2}_{L^{\infty}([0,T];L^{6}(\mathbb{R}^{3}))} \leq C \|\nabla (\delta - \chi_{R}) \ast w_{n}^{(l)}(t)\|^{2}_{L^{2}} \leq C \|\nabla (\delta - \chi_{R}) \ast w_{n}^{(l)}(0)\|^{2}_{L^{2}}.$$
By Plancherel's theorem and identity \eqref{eq 46}, one has
$$\|\nabla (\delta - \chi_{R}) \ast w_{n}^{(l)}(0)\|_{L^{2}} = \frac{1}{(2\pi)^{3}} \int_{\mathbb{R}^{3}} |\xi|^{2} \Big|\widehat{w_{n}^{(l)}(0)}(\xi)\Big[1-\widetilde{\chi_{R}^{1}}\Big(|\xi|^{2}\Big) \widehat{\chi^{2}_{R}}(\xi)\Big]\Big|^{2} \ d\xi .$$
Observe that $$
\Big[1- \widetilde{\chi_{R}^{1}}\Big(|\xi|^{2}\Big) \widehat{\chi^{2}_{R}}(\xi)\Big]=0, \quad \text{for} \frac{1}{h_{n}R} \leq |\xi| \leq \frac{R}{h_{n}},
$$ and, consequently,
\begin{equation} \label{eq 49}
	\limsup_{n \rightarrow \infty} \|(\delta - \chi_{R}) \ast w_{n}^{(l)}\|^{2}_{L^{\infty}([0,T];L^{6}(\mathbb{R}^{3}))} \leq C \limsup_{n\rightarrow \infty} \int_{\{|\xi| \leq \frac{1}{h_{n}R}\} \cup \{|\xi| \geq \frac{R}{h_{n}}\}} |\xi|^{2} |\widehat{w_{n}^{(l)}(0)}|^{2} \ d\xi , 
\end{equation}
which is the desired bound for the second term on the right-hand side of inequality \eqref{eq 43}.

With these bounds in hand, let us analyze \eqref{eq 43}. From estimates \eqref{eq 48} and \eqref{eq 49}, one has
\begin{equation*}
	\limsup_{n \rightarrow \infty}\|w_{n}^{(l)}\|_{L^{\infty}([0,T]; L^{6}(\mathbb{R}^{3}))} \leq C(R) \limsup_{n \rightarrow \infty} \Bigg[h_{n}^{\frac{1}{3}} \delta(w_{n}^{(l)})^{\frac{2}{3}} +  \int_{\{|\xi| \leq \frac{1}{h_{n}R}\} \cup \{|\xi| \geq \frac{R}{h_{n}}\}} |\xi|^{2} |\widehat{w_{n}^{(l)}(0)}|^{2} \ d\xi\Bigg].
\end{equation*}
So, taking  $l,R\to\infty$, using that $\delta(w_{h}^{(l)})\longrightarrow 0$ and  $w_{n}^{(l)}$ is (strictly) $h_{n}$-oscillatory (Remark \ref{remark2}), it follows that 
$$\limsup_{n \rightarrow \infty}\|w_{n}^{(l)}\|_{L^{\infty}([0,T]; L^{6}(\mathbb{R}^{3}))}  \longrightarrow 0  \mbox{ as }\l\rightarrow \infty.$$
Therefore, by interpolation, one gets 
$$\limsup_{n \rightarrow \infty}\|w_{n}^{(l)}\|_{L^{10}([0,T]; L^{10}(\mathbb{R}^{3}))}  \longrightarrow 0 \mbox{ as } l\rightarrow \infty,$$
since $\|w_{n}^{(l)}\|_{L^{7}_{t}L^{14}_{x}} \leq C \|w_{n}(0)\|_{\dot{H}^{1}}.$
This completes the proof of the first part of Proposition \ref{prop3.11}. It remains only to show the orthogonality of cores. We show it by contradiction. 
To this end, assume that the index
$$j_{K} = \max \big\{j \in \{1,...,K\}; (t_{n}^{(j)},x_{n}^{(j)}) \not \perp_{h_{n}}(t_{n}^{(K+1)},x_{n}^{(K+1)}) \big\}$$
exists. 
The following are consequences of the construction at the beginning of the demonstration of Proposition \ref{prop3.11}:
\begin{equation} \label{eq 54}
	D_{h_{n}}^{(l+1)}w_{n}^{(l)} \rightharpoonup \varphi ^{(l+1)} \ \mbox{with } \varphi^{(l+1)} \neq 0 \mbox{ if } l\leq K,
\end{equation}
\begin{equation} \label{eq 55}
	w_{n}^{(l)} = p_{n}^{(l+1)} + w_{n}^{(l+1)},
\end{equation}
and
\begin{equation} \label{eq 56}
	w_{n}^{(j_{K})} = \sum_{j=j_{K} +1}^{K+1} p_{n}^{(j)} + w_{n}^{(K+1)}.
\end{equation}
Moreover, the definition of $p_{n}^{(l)}$ and Lemma \ref{lemma3.2} implies 
$D_{h_{n}}^{(l)}p_{n}^{(l)} \rightharpoonup \varphi^{(l)}.$
Then, we get, from \eqref{eq 54} and \eqref{eq 55}, that  
$D_{h_{n}}^{(l+1)}w_{n}^{(l+1)}\rightharpoonup 0.$ Applying this to $l + 1 = j_{K}$ gives us  $D_{h_{n}}^{(K+1)}w_{n}^{(j_{K})}\rightharpoonup 0,$
due to the first part of Lemma \ref{lemma3.14} and the definition of $j_{K}$, since $(t_{n}^{(j_{K})},x_{n}^{(j_{K})}) \not \perp_{h_{n}}(t_{n}^{(K+1)},x_{n}^{(K+1)})$.

The definition of $j_{K}$ and the second part of Lemma \ref{lemma3.14} give
$D_{h_{n}}^{(K+1)}p_{n}^{(l)}\rightharpoonup 0$ for $j_{K} + 1 \leq l \leq K.$
``Applying" $D_{h_{n}}^{(K+1)}$ to equality \eqref{eq 56}, one gets 
\begin{eqnarray*}
	D_{h_{n}}^{(K+1)} w_{n}^{(j_{K})} & = &   \sum_{j=j_{K} +1}^{K+1} D_{h_{n}}^{(K+1)} p_{n}^{(j)} + D_{h_{n}}^{(K+1)} w_{n}^{(K+1)} \\
	& = & \sum_{j=j_{K} +1}^{K} D_{h_{n}}^{(K+1)} p_{n}^{(j)} + D_{h_{n}}^{(K+1)} p_{n}^{(K+1)} +  D_{h_{n}}^{(K+1)} w_{n}^{(K+1)} .
\end{eqnarray*}
Therefore, 
$$D_{h_{n}}^{(K+1)}w_{n}^{(j_{K})}\rightharpoonup \varphi^{K+1} \neq 0, $$
which is a contradiction since we have already proven that
$$D_{h_{n}}^{(K+1)}w_{n}^{(j_{K})}\rightharpoonup 0.$$
This completes the proof of the Proposition \ref{prop3.11}.
\end{proof}

	\begin{remark}\label{remark2}
	Observe that $w_{n}^{(l)}$ is (strictly) $h_{n}$-oscillatory. 
	
	Indeed, being $w_{n}^{(1)} = \widetilde{v}_{n} - \widetilde{p}_{n}^{(1)}$ for $l=1$, we apply the operator $\sigma_{R}(D)$ to equation \eqref{eq 43}, where $\sigma_{R} = \textbf{1}_{\{h_{n}|\xi| \leq \frac{1}{R}\} \cup \{h_{n}|\xi| \geq R\}} $, $R>0$. We get
	$$\|\nabla \sigma_{R}(D)\widetilde{v}_{n}\|_{L^{2}}^{2}= \|\nabla \sigma_{R}(D)\widetilde{p}_{n}^{(1)}\|_{L^{2}}^{2} + \|\nabla \sigma_{R}(D)w_{n}^{(1)}\|_{L^{2}}^{2} + o(1).$$
	Iterating, we obtain
	$$\| \nabla \sigma_{R}(D)\widetilde{v}_{n}\|_{L^{2}}^{2}=  \sum_{j=1}^{l} \|\nabla \sigma_{R}(D)\widetilde{p}_{n}^{(j)}\|_{L^{2}}^{2} + \|\nabla \sigma_{R}(D)w_{n}^{(l)}\|_{L^{2}}^{2} + o(1), $$
	which means 
	$$\limsup_{n \rightarrow \infty} \int_{\{h_{n}|\xi| \leq \frac{1}{R}\} \cup \{h_{n}|\xi| \geq R\}} |\xi|^{2} |\hat{w}_{n}^{(l)}(.,\xi)|^{2} \ d\xi \leq \limsup_{n \rightarrow \infty} \int_{\{h_{n}|\xi| \leq \frac{1}{R}\} \cup \{h_{n}|\xi| \geq R\}} |\xi|^{2} |\hat{v}_{n}(.,\xi)|^{2} \ d\xi.$$
		Since $\widetilde{v}_{n}$ is a (strictly) $h_{n}$-oscillatory sequence, so it is $w^{(l)}_{n}$.
\end{remark}

Before presenting the proof of Theorem \ref{lineardecomp}, let us present a result from \cite[Lemma 2.7 and Remark 2.8]{keraani}, which will be used.
 
	\begin{lemma}
		Let $(\underline{h}^{(j)}, \underline{x}^{(j)},\underline{t}^{(j)})$ be a family of pairwise orthogonal scales-cores and $(V^{(j)})$ a family of functions in $L^{10}(\mathbb{R},L^{10}(\mathbb{R}^{3}))$. For every $l \geq 1$, we have
\begin{equation}\label{lemma3.15}
\Bigg\| \sum_{j=1}^{l} \frac{1}{\sqrt{h_{n}^{(j)}}} V^{(j)} \Big(\frac{.-t_{n}^{(j)}}{{h_{n}^{(j)}}^{2}},\frac{.-x_{n}^{(j)}}{h_{n}^{(j)}}\Big) \Bigg\|^{10}_{L^{10}_{t}L^{10}_{x}} \longrightarrow \sum_{j=1}^{l}\|V^{(j)}\|_{L^{10}_{t}L^{10}_{x}}^{10}, \mbox{ as }  n\rightarrow \infty.
\end{equation}
		Additionally, 
\begin{equation}\label{obs7}
\Bigg\| \nabla \Big(\sum_{j=1}^{l} \frac{1}{\sqrt{h_{n}^{(j)}}} V^{(j)} \Big(\frac{.-t_{n}^{(j)}}{{h_{n}^{(j)}}^{2}},\frac{.-x_{n}^{(j)}}{h_{n}^{(j)}}\Big) \Big) \Bigg\|^{\frac{10}{3}}_{L^{\frac{10}{3}}_{t}L^{\frac{10}{3}}_{x}} \longrightarrow \sum_{j=1}^{l}\|\nabla V^{(j)}\|_{L^{\frac{10}{3}}_{t}L^{\frac{10}{3}}_{x}}^{\frac{10}{3}}, \mbox{ as }  n\rightarrow \infty.
\end{equation}
	\end{lemma}

Now, we have all the tools to prove Theorem \ref{lineardecomp}.

\noindent\textit{\bf{Step 4. Proof of Theorem \ref{lineardecomp}}.} Denote by $v_{n}^{(j)}$ (and the rest ($\rho_{n}^{(l)}$)) the $h_{n}^{(j)}$-oscillatory component obtained by decomposition \eqref{eq 31} and $p_{n}^{(j,\alpha)}$ the concentrating solutions obtained from decomposition \eqref{eq 38} (and the rest $w_{n}^{(j,A_{j})})$. Summing everything up, one has
	$$v_{n}(t,x) = \sum_{j=1}^{l} \Big(\sum_{\alpha =1}^{A_{j}} p_{n}^{(j,\alpha)}(t,x) + w_{n}^{(j,A_{j})}(t,x)\Big) + \rho_{n}^{(l)}(t,x).$$
	Rewrite this equation as 
	$$v_{n}(t,x) = \sum_{j=1}^{l} \Big(\sum_{\alpha =1}^{A_{j}} p_{n}^{(j,\alpha)}(t,x) \Big) + w_{n}^{(l,A_{1},...,A_{l})}(t,x),$$
	where
	$$w_{n}^{(l,A_{1},...,A_{l})}(t,x) = \sum_{j=1}^{l} w_{n}^{(j,A_{j})}(t,x) + \rho_{n}^{(l)}(t,x)$$
	for $l$ and $A_{j}$ fixed, $1 \leq j \leq l$. We enumerate this pairs by the bijection $\sigma: \mathbb{N}^{2}\rightarrow \mathbb{N}$ defined by
	$$\sigma(j,\alpha) < \sigma(k,\beta) \ \ \mbox{if } j+ \alpha < k+ \beta \ \mbox{or } j+\alpha = k + \beta \ \mbox{and } j<k.$$ 	The almost orthogonality identity \eqref{eq 29} is satisfied. Indeed, combining \eqref{eq 33} and \eqref{eq 40}, we obtain
	\begin{eqnarray*}
		\|\nabla v_{n}\|^{2}_{L^{2}} & = & \sum_{j=1}^{l} \|\nabla v_{n}^{(j)}\|^{2}_{L^{2}} + \|\nabla \rho_{n}^{(l)}\|^{2}_{L^{2}} + o(1)\\
		& = & \sum_{j=1}^{l} \Big( \sum_{\alpha =1}^{A_{j}} \| \nabla p_{n}^{(j,\alpha)}\|^{2}_{L^{2}}  + \|\nabla w_{n}^{(j,A_{j})}\|^{2}_{L^{2}} \Big) +  \|\nabla \rho_{n}^{(l)}\|^{2}_{L^{2}} + o(1)\\
		& = & \sum_{j=1}^{l} \Big( \sum_{\alpha =1}^{A_{j}} \| \nabla p_{n}^{(j,\alpha)}\|^{2}_{L^{2}}  \Big) + \sum_{j=1}^{l} \|\nabla w_{n}^{(j,A_{j})}\|^{2}_{L^{2}}  +  \|\nabla \rho_{n}^{(l)}\|^{2}_{L^{2}} + o(1),
	\end{eqnarray*}
	but
	\begin{eqnarray*}
		\|\nabla w_{n}^{(l,A_{1},...,A_{l})}\|^{2}_{L^{2}} & = & \Bigg\| \nabla \Bigg( \sum_{j=1}^{l} w_{n}^{(j,A_{j})} + \rho_{n}^{(l)}\Bigg) \Bigg\|^{2}_{L^{2}} \\
		& = &  \sum_{j=1}^{l}  \|\nabla w_{n}^{(j,A_{j})}\|^{2}_{L^{2}}  + \|\nabla \rho_{n}^{(l)}\|^{2}_{L^{2}} 
	\end{eqnarray*}
	since $w_{n}^{(j,A_{j})}$ is $h_{n}^{j}$-oscillatory and $\rho_{n}^{(l)}$ is $h_{n}^{j}$-singular for all $1\leq j\leq l$.
	Therefore,
	\begin{equation} \label{eq 57}
		\|\nabla v_{n}\|^{2}_{L^{2}} =  \sum_{j=1}^{l} \sum_{\alpha =1}^{A_{j}} \| \nabla p_{n}^{(j,\alpha)}\|^{2}_{L^{2}} + 	\|\nabla w_{n}^{(l,A_{1},...,A_{l})}\|^{2}_{L^{2}} + o(1) \mbox{ as } n\rightarrow \infty.
	\end{equation}
	The last point that remains to be checked is the convergence of the remainder $w_{n}^{(l, A_{1}, ..., A_{l})}$ to zero in the Strichartz norm. To this end, let $\varepsilon > 0$ be a small arbitrary number. To get the result, it suffices to prove that for $l_{0}$ large enough, 
$$
		\|w_{n}^{(l,A_{1},...,A_{l})}\|_{L^{\infty}_{t}L^{6}_{x}} \leq \varepsilon   
$$
	for all $(l,A_{1},...,A_{l})$ satisfying $l \geq l_{0}$ and $\sigma(j,A_{j}) \geq \sigma(l_{0},1)$. To prove this, 
	first choose $l_{0}$ such that, for every $l \geq l_{0}$,
$$
		\limsup_{n \rightarrow \infty} \|\rho_{n}^{(l)}\|_{L^{\infty}_{t}L^{6}_{x}} \leq \varepsilon,
$$
Note that the existence of such $l_0$ is ensured by \eqref{eq 32}.
	Moreover, by \eqref{eq 39}, for every $l \geq l_{0}$, there exists $B_{l}$, such that $A_{j} \geq B_{l}$ for every $j \in \{1,..., l\}$ and 
$$
		\limsup_{n \rightarrow \infty} \|w_{n}^{(j,A_{j})}\|_{L^{\infty}_{t}L^{6}_{x}} \leq \frac{\varepsilon}{l}.
$$
	Moreover, the expression 
	\eqref{eq 57} implies that the series with general term $\sum_{(j,\alpha)} \limsup_{n \rightarrow \infty} \|\nabla p_{n}^{(j,\alpha)}(0)\|^{2}_{L^{2}}$ is convergent. In particular, we may also assume, increasing $l_0$ if necessary,  that $l_{0}$ is such that
	\begin{equation} \label{eq 61}
		\sum_{\sigma(j,\alpha) > \sigma(l_{0},1)} \limsup_{n \rightarrow \infty} \|\nabla p_{n}^{(j,\alpha)}(0)\|^{2}_{L^{2}} \leq \varepsilon.
	\end{equation}
	Now, rewrite the remainder $w_{n}^{(l, A_{1}, ..., A_{l})}$ as 
	\begin{equation*}
		w_{n}^{(l,A_{1},...,A_{l})} = \rho_{n}^{(l)} + \sum_{1\leq j\leq l} w_{n}^{(j, \max(A_{j},B_{l}))} + S_{n}^{(l,A_{1},...,A_{l})},
	\end{equation*}
	where
	\begin{equation*}
		S_{n}^{(l,A_{1},...,A_{l})} = \sum_{1\leq j\leq l, A_{j} < B_{l}} w_{n}^{(j,A_{j})} - w_{n}^{(j,B_{l})}.
	\end{equation*}
	One has
$$
		w_{n}^{(j,A_{j})} - w_{n}^{(j,B_{l})}  =  \sum_{\alpha=1}^{B_{l}} p_{n}^{(j,\alpha)} - \sum_{\alpha=1}^{A_{j}} p_{n}^{(j,\alpha)} =   \sum_{A_{j} < \alpha \leq B_{l}} p_{n}^{(j,\alpha)}.
$$
	Hence,
	\begin{equation*}
		S_{n}^{(l,A_{1},...,A_{l})} = \sum_{1\leq j\leq l, A_{j} < B_{l}} \sum_{A_{j} < \alpha \leq B_{l}} p_{n}^{(j,\alpha)}.
	\end{equation*}
	Therefore, 
	\begin{eqnarray*}
		\limsup_{n \rightarrow \infty} \|w_{n}^{(l,A_{1},...,A_{l})} \|_{L^{\infty}_{t}L^{6}_{x}} & \leq & \limsup_{n \rightarrow \infty} \|\rho_{n}^{(l)}\|_{L^{\infty}_{t}L^{6}_{x}} + \limsup_{n \rightarrow \infty} \sum_{j=1}^{l} \|w_{n}^{(j, \max(A_{j},B_{l}))}\|_{L^{\infty}_{t}L^{6}_{x}} \\
		& & \mbox{  } + \limsup_{n \rightarrow \infty} \|S_{n}^{(l,A_{1},...,A_{l})}\|_{L^{\infty}_{t}L^{6}_{x}} \\
		& \leq & {2 \varepsilon} +  \limsup_{n \rightarrow \infty} \|S_{n}^{(l,A_{1},...,A_{l})}\|_{L^{\infty}_{t}L^{6}_{x}}.
	\end{eqnarray*}
	Since $S_{n}^{(l,A_{1},...,A_{l})}$ is a solution of the linear Schr\"{o}dinger equation, we have
	\begin{eqnarray*}
		\|S_{n}^{(l,A_{1},...,A_{l})}\|_{L^{\infty}_{t}L^{6}_{x}}  \leq  C \|\nabla S_{n}^{(l,A_{1},...,A_{l})}\|_{L^{2}} 
		& \leq& C \|\nabla S_{n}^{(l,A_{1},...,A_{l})}(0)\|_{L^{2}}\\
		&\leq&  C \sum_{1\leq j\leq l, A_{j} < B_{l}} \sum_{A_{j} < \alpha \leq B_{l}} \|\nabla p_{n}^{(j,\alpha)}(0)\|_{L^{2}} \\
		& \leq & C \varepsilon,
	\end{eqnarray*}
	because the sum is restricted to some $\sigma(j,\alpha)$ satisfying $\sigma(j,\alpha) > \sigma(j, \alpha_{j}) > \sigma(l_{0},1)$ and is indeed smaller than $C\varepsilon$ due to inequality \eqref{eq 61}.
	Therefore,
	$\limsup_{n \rightarrow \infty} \|w_{n}^{(l,A_{1},...,A_{l})} \|_{L^{\infty}_{t}L^{6}_{x}}$ is smaller than $(2+C) \varepsilon$ for all $(l,A_{1},...,A_{l})$ satisfying $l \geq l_{0}$ and $\sigma(j,A_{j}) \geq \sigma(l_{0},1)$. Through the same procedure, we get the same estimates
	for the $L^{10}(L^{10})$ norm, that is, 
	\begin{eqnarray*}
		\limsup_{n \rightarrow \infty} \|w_{n}^{(l,A_{1},...,A_{l})} \|_{L^{10}_{t}L^{10}_{x}} & \leq & \limsup_{n \rightarrow \infty} \|\rho_{n}^{(l)}\|_{L^{10}_{t}L^{10}_{x}} + \limsup_{n \rightarrow \infty} \sum_{j=1}^{l} \|w_{n}^{(j, \max(A_{j},B_{l}))}\|_{L^{10}_{t}L^{10}_{x}} \\
		& & \mbox{  } + \limsup_{n \rightarrow \infty} \|S_{n}^{(l,A_{1},...,A_{l})}\|_{L^{10}_{t}L^{10}_{x}} \\
		& \leq & {2 \varepsilon}+  \limsup_{n \rightarrow \infty} \|S_{n}^{(l,A_{1},...,A_{l})}\|_{L^{10}_{t}L^{10}_{x}}.
	\end{eqnarray*}
Moreover, 
	$$\limsup_{n \rightarrow \infty} \|S_{n}^{(l,A_{1},...,A_{l})}\|_{L^{10}_{t}L^{10}_{x}}^{10} = \limsup_{n \rightarrow \infty} \Big\| \sum_{(j,\alpha)} p_{n}^{(j,\alpha)} \Big\|_{L^{10}_{t}L^{10}_{x}}^{10}$$
	and rescaling,
	$$p_{n}^{(j,\alpha)}(t,x) = \frac{1}{\sqrt{h_{n}}} \psi^{(j,\alpha)} \Big(\frac{t-t_{n}^{(j,\alpha)}}{h_{n}^{2}}, \frac{x-x_{n}^{(j,\alpha)}}{h_{n}}\Big),$$
	where $\psi^{(j,\alpha)} \in L^{\infty}(\mathbb{R};\dot{H}^{1}(\mathbb{R}^{3}))$. 
	So, by convergente \eqref{lemma3.15},
	$$\limsup_{n \rightarrow \infty} \Big\| \sum_{(j,\alpha)} p_{h}^{(j,\alpha)} \Big\|_{L^{10}_{t}L^{10}_{x}}^{10} = \sum_{(j,\alpha)} \|\psi^{(j,\alpha)}\|^{10}_{L^{10}_{t}L^{10}_{x}}. $$
So, through Strichartz estimates and Lemma \ref{lemma3.3}, one gets
	\begin{eqnarray}\label{eq 62}
		\sum_{(j,\alpha)} 	\|\psi^{(j,\alpha)}\|_{L^{10}_{t}L^{10}_{x}} ^{10}  =  	\sum_{(j,\alpha)} 	\|p_{n}^{(j,\alpha)}\|_{L^{10}_{t}L^{10}_{x}} ^{10}
	\leq  C\sum_{(j,\alpha)}  \Big( \|\nabla p_{n}^{(j,\alpha)}(0)\|_{L^{2}}^{2}\Big)^{5}.
	\end{eqnarray}
	On the other hand, by \eqref{eq 57} one has that $\sum_{(j,\alpha)} \|\nabla p_{n}^{(j,\alpha)}(0)\|_{L^{2}}^{2}$ is convergent, and so the right-hand side of \eqref{eq 62} is finite.
	Thus
	$$\Big(\sum_{\sigma(j,\alpha) > \sigma(l_{0},1)} 	\|\psi^{(j,\alpha)}\|_{L^{10}_{t}L^{10}_{x}} ^{10}  \Big)^\frac{1}{10} \leq {\varepsilon}.$$
	Hence,
	\begin{eqnarray*}
		\limsup_{n \rightarrow \infty} \|w_{n}^{(l,A_{1},...,A_{l})} \|_{L^{10}_{t}L^{10}_{x}} & \leq & {2 \varepsilon} +  \limsup_{n \rightarrow \infty} \|S_{n}^{(l,A_{1},...,A_{l})}\|_{L^{10}_{t}L^{10}_{x}} \\
		& \leq & {2 \varepsilon} +  \Big(\sum_{(j,\alpha)}	\|\psi^{(j,\alpha)}\|_{L^{10}_{t}L^{10}_{x}} ^{10}  \Big)^\frac{1}{10}  =  3\varepsilon.
	\end{eqnarray*}
	Since $\varepsilon$ is an arbitrary small number, we conclude that
$$
		\limsup_{n \rightarrow \infty}\|w_{n}^{(l,A_{1},...,A_{l})}\|_{L^{10}_{t}L^{10}_{x}} \longrightarrow 0,
$$
which proves Theorem \ref{lineardecomp}.
\end{proof}

To finish this section, we present the next result, which is a consequence of the construction carried out in the proof of Proposition \ref{prop3.11}. 
\begin{lemma} \label{lemma3.16} Consider the notations and the assumptions of Theorem \ref{lineardecomp}. 	For any $l\in \mathbb{N}$ and $1\leq j \leq l$, we have 
	$D_{h_{n}}^{(j)} w_{n}^{(l)}\rightharpoonup 0.$
\end{lemma}
\begin{proof}
	Assuming that $D_{h_{n}}^{(j)}w_{n}^{(l)} \rightharpoonup \varphi$, we can directly use the decomposition from Theorem \ref{lineardecomp} to write 
	$$w_{n}^{(l)} = \sum_{i=l+1}^{L} p_{n}^{(i)} + w_{n}^{(L)},$$
for $L > l$. 
In case of scale orthogonality of $h_{n}^{(j)}$ and $h_{n}^{(i)},$ for $l+1 \leq i\leq L$, we have $D_{h_{n}}^{(j)} p_{n}^{(i)} \rightharpoonup 0.$ Indeed, by hypothesis, $p_{n}^{(i)}$ is a concentrating solution and so 
	
	$$p_{n}^{(i)}(t,x) = \frac{1}{\sqrt{h_{n}^{(i)}}} \varphi^{(i)}\Big(\frac{t-t_{n}^{(i)} }{{h_{n}^{(i)}}^{2}}, \frac{x-x_{n}^{(i)} }{{h_{n}^{(i)}}}\Big),$$
	which means that
	$$\sqrt{h_{n}^{(j)}}p_{n}^{(i)}(t_{n}^{(j)}+ (h_{n}^{(j)})^{2}s ,x_{n}^{(j)} + {h_{n}^{(j)}}y) = \dfrac{\sqrt{h_{n}^{(j)}}}{\sqrt{h_{n}^{(i)}}} \varphi^{(i)}\Big(\frac{t_{n}^{(j)}-t_{n}^{(i)}}{{h_{n}^{(i)}}^{2}}+ s\Big(\frac{h_{n}^{(j)}}{h_{n}^{(i)}}\Big)^2,\frac{x_{n}^{(j)}-x_{n}^{(i)}}{h_{n}^{(i)}} + y\frac{h_{n}^{(j)}}{h_{n}^{(i)}}\Big).$$
	Without loss of generality, we may assume $\varphi^{(i)}$ to be continuous and compactly supported. Thus, for  $\psi$ to be a compactly supported function, one has
	\begin{equation*}
		\int_{\mathbb{R}^{3}} \nabla \sqrt{h_{n}^{(j)}}p_{n}^{(i)}(t_{n}^{(j)}+ (h_{n}^{(j)})^{2}s ,x_{n}^{(j)} + {h_{n}^{(j)}}y)  \cdot \nabla \psi(y) \ dy \end{equation*}
	\begin{equation*}
		= 	\dfrac{\sqrt{h_{n}^{(j)}}}{\sqrt{h_{n}^{(i)}}}\int_{\mathbb{R}^{3}}  \nabla \varphi^{(i)}\Big(\frac{t_{n}^{(j)}-t_{n}^{(i)}}{{h_{n}^{(i)}}^{2}}+ s\Big(\frac{h_{n}^{(j)}}{h_{n}^{(i)}}\Big)^2,\frac{x_{n}^{(j)}-x_{n}^{(i)}}{h_{n}^{(i)}} + y\frac{h_{n}^{(j)}}{h_{n}^{(i)}}\Big) \cdot \nabla \psi(y) \ dy
	\end{equation*}
	and the orthogonality of $h_{n}^{(j)}$ and $h_{n}^{(i)}$ means that 
	$$\frac{h_{n}^{(j)}}{h_{n}^{(i)}} + \frac{h_{n}^{(i)}}{h_{n}^{(j)}} \longrightarrow +\infty.$$
	If $\frac{h_{n}^{(j)}}{h_{n}^{(i)}} \longrightarrow +\infty$, we have 
	\begin{equation*}
		\begin{split}
			\dfrac{\sqrt{h_{n}^{(j)}}}{\sqrt{h_{n}^{(i)}}}&\int_{\mathbb{R}^{3}}  \nabla \varphi^{(i)}\Big(\frac{t_{n}^{(j)}-t_{n}^{(i)}}{{h_{n}^{(i)}}^{2}}+ s\Big(\frac{h_{n}^{(j)}}{h_{n}^{(i)}}\Big)^2,\frac{x_{n}^{(j)}-x_{n}^{(i)}}{h_{n}^{(i)}} + y\frac{h_{n}^{(j)}}{h_{n}^{(i)}}\Big) \cdot \nabla \psi(y) \ dy = O\Bigg(\Big(\frac{h_{n}^{(j)}}{h_{n}^{(i)}}\Big)^{\frac{1}{2}}\Bigg) \rightarrow 0,
		\end{split}
	\end{equation*}
	as done in \cite[Lemma 2.7]{keraani}.  If $\frac{h_{n}^{(i)}}{h_{n}^{(j)}} \longrightarrow +\infty$, we make the change of variables  $$\frac{x_{n}^{(j)}-x_{n}^{(i)}}{h_{n}^{(i)}} + y\frac{h_{n}^{(j)}}{h_{n}^{(i)}} = x,$$ to get  
	\begin{equation*}
		\begin{split}
			\dfrac{\sqrt{h_{n}^{(j)}}}{\sqrt{h_{n}^{(i)}}}&\int_{\mathbb{R}^{3}}  \nabla \varphi^{(i)}\Big(\frac{t_{n}^{(j)}-t_{n}^{(i)}}{{h_{n}^{(i)}}^{2}}+ s\Big(\frac{h_{n}^{(j)}}{h_{n}^{(i)}}\Big)^2,x\Big) \cdot \nabla \psi\Big(\frac{h_{n}^{(i)}}{h_{n}^{(j)}}x - \frac{x_{n}^{(j)}-x_{n}^{(i)}}{h_{n}^{(i)}}\Big) \frac{h_{n}^{(i)}}{h_{n}^{(j)}} \ dx.\\
			&= O\Bigg(\Big(\frac{h_{n}^{(i)}}{h_{n}^{(j)}}\Big)^{\frac{1}{2}}\Bigg) \rightarrow 0,
		\end{split}
	\end{equation*}
	since $\psi$ is assumed to be compactly supported, which gives the desired result $D_{h_{n}}^{(j)} p_{n}^{(i)} \rightharpoonup 0 $.
	
	Now, in case $h_{n}^{(j)} = h_{n}^{(i)}$ and $(\underline{x}^{(j)},\underline{t}^{(j)})\perp_{h_{n}} (\underline{x}^{(i)},\underline{t}^{(i)})$, the second part of Lemma \ref{lemma3.14} gives the same convergence. Therefore, in both cases one has $$D_{h_{n}}^{(j)}w_{n}^{(L)}\rightharpoonup \varphi.$$
	Since, by Theorem \ref{lineardecomp},
	$\limsup_{n \rightarrow \infty}\|w_{n}^{(L)}\|_{L^{\infty}_{t}L^{6}_{x}}\rightarrow 0$, we have $\varphi=0$, proving the lemma. 
	\color{black}
\end{proof}

	
\section{Nonlinear profile decomposition}\label{sec3}
  In this section, we establish a decomposition into profiles similar to the one carried out in the previous section, but, this time, for a sequence of solutions to the nonlinear equation \eqref{eq 15}.

\begin{theorem} \label{nonlineardec}
	Let $u_{n}$ be the sequence of solutions to nonlinear Schr\"{o}dinger equation \eqref{eq 15} with initial data $\varphi_{n}$ bounded in $\dot{H}^{1}(\mathbb{R}^{3})$ and satisfying $\limsup_{n \rightarrow \infty}\|\varphi_{n}\|_{\dot{H}^{1}} < \lambda_{0}$. Let $p_{n}^{(j)}$ be the linear concentrating solutions given by Theorem \ref{lineardecomp} and $q_{n}^{(j)}$ the associated nonlinear concentrating solutions. Then, up to a subsequence, we have
$$
		u_{n}(t,x) = \sum_{j=1}^{l} q_{n}^{(j)}(t,x) + w_{n}^{(l)}(t,x) + r_{n}^{(l)}(t,x),
$$
and
	\begin{equation} \label{eq 68}
		\begin{split}
			\limsup_{n \rightarrow \infty} &\left(\|\nabla r_{n}^{(l)}\|_{L^{\frac{10}{3}}([0,T];L^{\frac{10}{3}}(\mathbb{R}^{3}))} \right.\\ &\left.+ \|r_{n}^{(l)}\|_{L^{10}([0,T];L^{10}(\mathbb{R}^{3}))} + \|r_{n}^{(l)}\|_{L^{\infty}([0,T];\dot{H}^{1}(\mathbb{R}^{3}))} \right) \longrightarrow 0 \mbox{ as } l\rightarrow \infty.
		\end{split}
	\end{equation}
\end{theorem}
The following notations will be used in this section: $\beta(z) = |z|^{4}z,$ $W_{n}^{(l)} = \sum_{j=1}^{l} q_{n}^{(j)},$ and
$$f_{n}^{(l)} = \sum_{j=1}^{l} \beta(q_{n}^{(j)}) - \beta \Big( \sum_{j=1}^{l} q_{n}^{(j)} + w_{n}^{(l)} + r_{n}^{(l)}\Big).$$
Before proving the decomposition result, we first show that nonlinear concentration solutions behave similarly to linear concentration solutions, at least in a specific type of interval.
			
\subsection{Behavior of nonlinear concentrating solutions} To give the behavior of nonlinear concentration solutions, we will use the following two lemmas from Keraani in \cite{keraani}.

\begin{lemma}\label{prop4.4}
		Let $I=[a,b]$. The solution $v \in C([a,b];\dot{H}^{1}(\mathbb{R}^{3}))$ of the equation
		$$i\partial_{t} v + \Delta v = f,\quad I\times\mathbb{R}^3,$$
		with $\nabla f \in L^{\frac{10}{7}}(I\times \mathbb{R}^{3})$, satisfies
		$$|||v|||_{I} + \sup_{t \in I} \|\nabla v(t)\|_{L^{2}} \leq C \Big(\|\nabla v(a)\|_{L^{2}} + \|\nabla f\|_{L^{\frac{10}{7}}(I\times \mathbb{R}^{3})}\Big).$$
	\end{lemma}
			\begin{lemma}\label{lemma4.7}
				Let $M=M(t)$ be a positive continuous function on $[0,T]$ such that $M(0)=0$ and, for every $t \in [0, T]$, 
				$$M(t) \leq c\Big(a + \sum_{\alpha =2}^{5} M^{\alpha}(t)\Big),$$
				with $0 < a < a_{0} = a_{0}(c)$. One has
				$$M(t) \leq 2ca, \quad \forall t \in [0, T].$$
			\end{lemma}

Now, we examine the behavior of the nonlinear concentrating solutions. As already seen, 
the evolution problem \eqref{eq 15} admits a complete scattering theory concerning linear problems
in the ball $\|u_{0}\|_{\dot{H}^{1}(\mathbb{R}^{3})} < \lambda_{0}$. The next theorem is a consequence of this scattering property.
\begin{theorem} \label{teo4.1}
	Let $u_{n}$ be a nonlinear concentrating solution. There exist two linear concentrating solutions $[\varphi_{i},\underline{h},\underline{x},\underline{t}]$, $i=1,2$, such that for all interval $[-T,T ]$ containing $t_{\infty}$, one has
	\begin{equation} \label{eq 64}
		\limsup_{n \rightarrow \infty} \Big(\|u_{n} - [\varphi_{1},\underline{h},\underline{x},\underline{t}]\|_{L^{10}(I_{n}^{1,\Lambda} \times \mathbb{R}^{3})} + \|u_{n} - [\varphi_{1},\underline{h},\underline{x},\underline{t}]\|_{L^{\infty}(I_{n}^{1,\Lambda};\dot{H}^{1}(\mathbb{R}^{3}))} \Big) \longrightarrow 0, 
	\end{equation}
	and 
	\begin{equation}\label{eq 65}
		\limsup_{n \rightarrow \infty} \Big(\|u_{n} - [\varphi_{2},\underline{h},\underline{x},\underline{t}]\|_{L^{10}(I_{n}^{3,\Lambda} \times \mathbb{R}^{3})} + \|u_{n} - [\varphi_{2},\underline{h},\underline{x},\underline{t}]\|_{L^{\infty}(I_{n}^{3,\Lambda};\dot{H}^{1}(\mathbb{R}^{3}))} \Big) \longrightarrow 0,
	\end{equation}
as $\Lambda \rightarrow \infty$. Here, $I_{n}^{1,\Lambda} =[-T, t_{n} - \Lambda h_{n}^{2}]$ and $I_{n}^{3,\Lambda} =(t_{n} + \Lambda h_{n}^{2}, T]$.
\end{theorem}
\begin{proof}
	We consider the case $\frac{t_{n}}{h_{n}}\rightarrow \infty$. The other cases are followed analogously. 
	
	Let us show \eqref{eq 64}. 
For the sake of simplicity, we take $I_{n}^{1,\Lambda} =[0, t_{\infty} - \Lambda h_{n}^{2}]$. We know that $u_{n}$ is a solution to 
	\begin{equation*}
		\left\{
		\begin{array}{ll}
			i\partial_{t} u_{n} + \Delta u_{n} -|u_{n}|^{4}u_{n}  =0,& \ \mbox{ on } [0,T] \times \mathbb{R}^{3},\\
			u_{n}(0)= \varphi \ \in \dot{H}^{1}(\mathbb{R}^{3}).
		\end{array}
		\right.
	\end{equation*}
	Since  $u_{n}(t,x)$ is a nonlinear concentrating solution, one has 
	$$u_{n}(t,x)  = \frac{1}{\sqrt{h_{n}}}u\Big(\dfrac{t-t_{n}}{h_{n}^{2}},\dfrac{x-x_{n}}{h_{n}}\Big),$$ where $u$ satisfies			
	\begin{equation*}
		i\partial_{s} u+ \Delta u-|u|^{4}u = 0  \ \mbox{ on } [0,T] \times \mathbb{R}^{3}.
	\end{equation*}
	By the scattering theory of Proposition \ref{prop2.5}, there exists $v$, solution of the system 
	\begin{equation*}
		\left\{
		\begin{array}{ll}
			i\partial_{s} v+ \Delta v =0, & \ \mbox{ on } [0,T] \times \mathbb{R}^{3},\\
			v(0)= \varphi^{1},
		\end{array}
		\right.
	\end{equation*}
	such that 
	$$\|\nabla u(s) - \nabla v(s)\|_{L^{2}} \longrightarrow 0, \ \ \mbox{ as } s\rightarrow - \infty.$$
Let  $$v_{n}(t,x)= \frac{1}{\sqrt{h_{n}}}v\Big(\frac{t-t_{n}}{h_{n}^{2}},\frac{x-x_{n}}{h_{n}}\Big) $$ satisfying 
	\begin{equation*}
		\left\{
		\begin{array}{lr}
			i\partial_{t}v_{n} + \Delta v_{n}   =0 \ \mbox{ on } [0,T] \times \mathbb{R}^{3} , \\
			{v}_{n}(t_{n})= \frac{1}{\sqrt{h_{n}}}\varphi^{1}. 
		\end{array}
		\right.
	\end{equation*}
We should prove that 
	\begin{equation} \label{v3}	\limsup_{n \rightarrow \infty} \Big(\|u_{n} - v_{n}\|_{L^{10}(I_{n}^{1,\Lambda} \times \mathbb{R}^{3})} + \|u_{n} - v_{n}\|_{L^{\infty}(I_{n}^{1,\Lambda};\dot{H}^{1}(\mathbb{R}^{3}))} \Big) \longrightarrow 0 
		\end{equation}
when $\Lambda \rightarrow \infty$. To this end, define $w_{n} := u_{n} - v_{n}$. Thus, $w_{n}$ satisfies the system
	\begin{equation*}
		\left\{
		\begin{array}{ll}
			i\partial_{t} w_{n} + \Delta w_{n} = |w_{n}+v_{n}|^{4}(w_{n}+v_{n}), \\
			w_{n}(0)= u_{n}(0)-v_{n}(0). 
		\end{array}
		\right.
	\end{equation*}
	Using Lemma \ref{prop4.4},  and denoting $$|||.|||_{I_{n}^{1,\Lambda}}:= \|.\|_{L^{10}(I_{n}^{1,\Lambda} \times \mathbb{R}^{3})} + \|\nabla .\|_{L^{\frac{10}{3}}(I_{n}^{1,\Lambda} \times \mathbb{R}^{3})},$$ one has
	\begin{eqnarray*}
		|||w_{n}|||_{I_{n}^{1,\Lambda}} + \| \nabla w_{n}\|_{L^{\infty}(I_{n}^{1,\Lambda};L^{2}(\mathbb{R}^{3}))} & \leq & c\Big(\|\nabla w_{n}(0)\|_{L^{2}} +  \|\nabla(w_{n}+v_{n})^{4}(w_{n} + v_{n})\|_{L^{\frac{10}{7}}(I_{n}^{1,\Lambda} \times \mathbb{R}^{3})}\Big).
	\end{eqnarray*}
	On the other hand, one has
$$
		\|\nabla w_{n}(0)\|_{L^{2}} 
		 =  \Big\|\nabla u\Big(-\frac{t_{n}}{h_{n}^{2}}\Big)- \nabla v\Big(-\frac{t_{n}}{h_{n}^{2}}\Big)\Big\|_{L^{2}} \mbox{ } \longrightarrow 0 \ \ \mbox{ as } n\rightarrow \infty.
		 $$
	Therefore
	\begin{eqnarray*}
		|||w_{n}|||_{I_{n}^{1,\Lambda}} + \| \nabla w_{n}\|_{L^{\infty}(I_{n}^{1,\Lambda};L^{2}(\mathbb{R}^{3}))} & \leq & 
		c\Big( \|\nabla w_{n}(0)\|_{L^{2}} \\ &    & \mbox{    } +\left. \| w_{n} + v_{n}\|^{4}_{L^{10}(I_{n}^{1,\Lambda} \times \mathbb{R}^{3})} \|\nabla (w_{n}+v_{n})\|_{L^{\frac{10}{3}}(I_{n}^{1,\Lambda} \times \mathbb{R}^{3})} \right)\\
		& \leq & c\Big( \|\nabla w_{n}(0)\|_{L^{2}} + \|w_{n}\|_{L^{10}(I_{n}^{1,\Lambda} \times \mathbb{R}^{3})}^{4}\|\nabla w_{n}\|_{L^{\frac{10}{3}}(I_{n}^{1,\Lambda} \times \mathbb{R}^{3})} \\
		&    & \mbox{    } + \|w_{n}\|_{L^{10}(I_{n}^{1,\Lambda} \times \mathbb{R}^{3})}^{4}\|\nabla v_{n}\|_{L^{\frac{10}{3}}(I_{n}^{1,\Lambda} \times \mathbb{R}^{3})} \\
		&    & \mbox{   } + \|v_{n}\|_{L^{10}(I_{n}^{1,\Lambda} \times \mathbb{R}^{3})}^{4}\|\nabla w_{n}\|_{L^{\frac{10}{3}}(I_{n}^{1,\Lambda} \times \mathbb{R}^{3})} \\
		&     & \mbox{   } + \|v_{n}\|_{L^{10}(I_{n}^{1,\Lambda} \times \mathbb{R}^{3})}^{4} \|\nabla v_{n}\|_{L^{\frac{10}{3}}(I_{n}^{1,\Lambda} \times \mathbb{R}^{3})} \Big).		
	\end{eqnarray*}
Using Lemma \ref{lemma3.10}, one gets
	$$\limsup_{n \rightarrow \infty}\|v_{n}\|_{L^{10}(I_{n}^{1,\Lambda} \times \mathbb{R}^{3})} \longrightarrow 0 \mbox{ as }\Lambda \rightarrow \infty.$$
	Hence, 
	$$\limsup_{n \rightarrow \infty}\Big(\|\nabla w_{n}(0)\|_{L^{2}}+ \|v_{n}\|_{L^{10}(I_{n}^{1,\Lambda} \times \mathbb{R}^{3})}\Big) \longrightarrow 0 \mbox{ as } \Lambda \rightarrow \infty.$$
	Given $\delta > 0$, there exists $\Lambda_{0}$ and and integer $n_{0}(\Lambda)$ such that for all  $\Lambda \geq \Lambda_{0}$ and for any integer $n \geq n_{0}(\Lambda)$, one has $\|\nabla w_{n}(0)\|_{L^{2}}+ \|v_{n}\|_{L^{10}(I_{n}^{1,\Lambda} \times \mathbb{R}^{3})} < \delta$. Therefore, choosing $\delta$ such that $\delta^{4} < \frac{1}{2c}$, one has
	\begin{eqnarray*}
		|||w_{n}|||_{I_{n}^{1,\Lambda}} + \| \nabla w_{n}\|_{L^{\infty}(I_{n}^{1,\Lambda};L^{2}(\mathbb{R}^{3}))} 
		& \leq & C\left( \|\nabla w_{n}(0)\|_{L^{2}} + |||w_{n}|||_{I_{n}^{1,\Lambda} }^{5} \right.\\ && \left.+ |||w_{n}|||_{I_{n}^{1,\Lambda} }^{4}  + \|v_{n}\|_{L^{10}(I_{n}^{1,\Lambda} \times \mathbb{R}^{3})}^{4}  \right).
	\end{eqnarray*}
Define $M_{n}(t)$ by 
	$$M_{n}(t) := 	\|w_{n}\|_{L^{10}([0,t] \times \mathbb{R}^{3})} + \|\nabla w_{n}\|_{L^{\frac{10}{3}}([0,t] \times \mathbb{R}^{3})}+ \| \nabla w_{n}\|_{L^{\infty}([0,t];L^{2}(\mathbb{R}^{3}))}.$$
	Then
	\begin{eqnarray*}
		M_{n}(t) 
		& \leq & C\Big( \|\nabla w_{n}(0)\|_{L^{2}} + \|v_{n}\|_{L^{10}(I_{n}^{1,\Lambda} \times \mathbb{R}^{3})}^{4} + \sum_{\alpha =2}^{5}M_{n}(t)^{\alpha}    \Big),
	\end{eqnarray*}
	for all $t \in I_{n}^{1,\Lambda}$.
Lemma \ref{lemma4.7} implies that  
	$$M_{n}(t) \le2C\Big(\|\nabla w_{n}(0)\|_{L^{2}} + \|v_{n}\|_{L^{10}(I_{n}^{1,\Lambda} \times \mathbb{R}^{3})}^{4}\Big) \longrightarrow 0, \ \ \ \forall t \in I_{n}^{1,\Lambda},$$
	for any $\Lambda \geq \Lambda_{0}$ as $n\rightarrow \infty$. Hence,
	$$
	\limsup_{n \rightarrow \infty}\Big(\|u_{n} - v_{n}\|_{L^{10}(I_{n}^{1,\Lambda} \times \mathbb{R}^{3})} + \| \nabla (u_{n}-v_{n})\|_{L^{\infty}(I_{n}^{1,\Lambda};L^{2}(\mathbb{R}^{3}))}\Big) \longrightarrow0 \mbox{ as }\Lambda \rightarrow \infty,$$
showing the convergence \eqref{v3}.	 The proof of \eqref{eq 65} is analogous.  
\end{proof}
\subsection{Auxiliary results}  We state and prove some technical results which will be used in the proof of Theorem 
ref{nonlineardec}. 
		\begin{lemma} \label{lemma4.6}
			There exists $\delta_{0} > 0$ such that if v is a solution of the linear Schr\"{o}dinger equation satisfying
			$|||v|||_{\mathbb{R}} \leq \delta_{0}$ and
			 $u$ is a solution  of the nonlinear Schr\"{o}dinger equation satisfying $v(T,x) = u(T,x)$ for some $T\in
			[-\infty,+\infty]$, then
			$|||u|||_{\mathbb{R}} \leq 2 |||v|||_{\mathbb{R}}.$
		\end{lemma}
		\begin{proof} Suppose  $\|\nabla(u-v)(x, -\infty)\|_{L^{2}}=0$ (the other cases can be analogously handled). Let $(T_{n})$ be a sequence of numbers converging to $+ \infty$ as $n \rightarrow +\infty$. Set $J_{n} = [-T_{n}, T_{n}].$
			The difference $w=u - v$ satisfies
			\begin{equation*}
				\left\{
				\begin{array}{ll}
					i\partial_{t} w + \Delta w = |w+v|^{4}(w+v),\\ 
					w (-T_{n})= (u-v)(-T_{n}).
				\end{array}
				\right.
			\end{equation*}
			From Lemma \ref{prop4.4}, it follows that
			\begin{eqnarray*}
				|||w|||_{J_{n}} &\leq &  C\Big( ||\nabla (u-v)(-T_{n})||_{L^{2}} + \|\nabla(w+v)^{5}\|_{L^{\frac{10}{7}}(J_{n}\times \mathbb{R}^{3})}\Big) \\
				& \leq & C\Big(|\nabla (u-v)(-T_{n})||_{L^{2}} + \|w\|^{4}_{L^{10}(J_{n}\times \mathbb{R}^{3})}\|\nabla w\|_{L^{\frac{10}{3}}(J_{n}\times \mathbb{R}^{3})} \\
				&    & \mbox{   } + \|w\|^{4}_{L^{10}(J_{n}\times \mathbb{R}^{3})}\|\nabla v\|_{L^{\frac{10}{3}}(J_{n}\times \mathbb{R}^{3})} + \|v\|^{4}_{L^{10}(J_{n}\times \mathbb{R}^{3})}\|\nabla w\|_{L^{\frac{10}{3}}(J_{n}\times \mathbb{R}^{3})} \\
				&    & \mbox{    }+ \|v\|^{4}_{L^{10}(J_{n}\times \mathbb{R}^{3})}\|\nabla v\|_{L^{\frac{10}{3}}(J_{n}\times \mathbb{R}^{3})}\Big).
			\end{eqnarray*}
Now, let $\delta_{0}>0$ such that  $\delta_{0}^{4} < \frac{1}{2C}$, $\delta_{0}^{5} < \frac{a_{0}}{2}$ and $\delta_{0}<1$, where $a_{0}$ is the constant from Lemma \ref{lemma4.7}. Therefore,
			\begin{eqnarray*}
				\|w\|_{L^{10}(J_{n}\times \mathbb{R}^{3})} + \|\nabla w\|_{L^{\frac{10}{3}}(J_{n}\times \mathbb{R}^{3})} & \leq &  2C\Big(\|\nabla (u-v)(-T_{n})||_{L^{2}} +|||w|||^{5}_{J_{n}} \\
				&   & \mbox{   } + \|w\|^{4}_{L^{10}(J_{n}\times \mathbb{R}^{3})}\|\nabla v\|_{L^{\frac{10}{3}}(J_{n}\times \mathbb{R}^{3})}  + |||v|||^{5}_{J_{n}}\Big)\\
				& \leq &  2C\Big(\|\nabla (u-v)(-T_{n})||_{L^{2}} +|||w|||^{5}_{J_{n}} \\
				&   & \mbox{   } + |||w|||^{4}_{L^{10}(J_{n}\times \mathbb{R}^{3})} + |||v|||^{5}_{J_{n}}\Big).
			\end{eqnarray*}
			Using the fact that $\|\nabla(u-v)(x, -T_{n})\|_{L^{2}} \rightarrow 0$ as $n$ goes to infinity, we get, for $n$ large,
			$$\|\nabla(u-v)(x, -T_{n})\|_{L^{2}}  + |||v|||_{J_{n}}^{5} \leq a_{0}.$$
			Thus, for $n$ large, the function $M: s \mapsto |||w|||_{[-T_{n},s]}$ satisfies the conditions of
			Lemma \ref{lemma4.7} on $[-T_{n} , T_{n}]$, so that
			$$M(T_{n}) = |||w|||_{J_{n}} \leq 4C(\|\nabla(u-v)(x, -T_{n})\|_{L^{2}}  + |||v|||_{J_{n}}^{5})$$
			for large $n$. Taking $n \rightarrow \infty$, we obtain
			$$|||w|||_{\mathbb{R}} \leq 4C |||v|||_{\mathbb{R}}^{5}.$$ 
			Hence
$$
				|||u|||_{\mathbb{R}}\leq |||w|||_{\mathbb{R}} + |||v|||_{\mathbb{R}}  \leq  (4C |||v|||^{4}_{\mathbb{R}} +1)|||v|||_{\mathbb{R}}.
$$
			Since $2C\delta_{0}^{4}  < 1$, this proves the lemma. 
		\end{proof}
	\begin{proposition} \label{prop4.5}
		There exists $C > 0$ such that
		\begin{equation} \label{eq 74}
			\limsup_{n \rightarrow \infty} |||W_{n}^{(l)} + w_{n}^{(l)}|||_{I} \leq C 
		\end{equation}
		for all $l \geq 1$. 
	\end{proposition}
		\begin{proof}
		First of all, observe that, using \eqref{eq 29},
		$$\limsup_{n \rightarrow \infty} |||w_{n}^{(l)}|||_{I} \leq C \limsup_{n \rightarrow \infty} \|\nabla w_{n}^{(l)}(0)\|_{L^{2}}  \leq C \limsup_{n \rightarrow \infty} \|\nabla v_{n}(0)\|_{L^{2}} \leq C$$
		for all $l \geq 1$. Thereby, to obtain \eqref{eq 74}, it suffices to prove that 
		$$\limsup_{n \rightarrow \infty}|||W_{n}^{(l)}|||_{I} \leq C,$$
		for all $l \geq 1$. Now, rescaling $p_{n}^{(j)}$ and $q_{n}^{(j)}$ by
		$$p_{n}^{(j)}(t,x) = \frac{1}{\sqrt{h_{n}^{(j)}}} \varphi^{(j)}\Big(\frac{t-t_{n}^{(j)}}{(h_{n}^{(j)})^{2}},\frac{x-x_{n}^{(j)}}{h_{n}^{(j)}}\Big)$$
		and
		$$q_{n}^{(j)}(t,x) = \frac{1}{\sqrt{h_{n}^{(j)}}} \psi^{(j)}\Big(\frac{t-t_{n}^{(j)}}{(h_{n}^{(j)})^{2}},\frac{x-x_{n}^{(j)}}{h_{n}^{(j)}}\Big),$$
respectively, with $\varphi$ and $\psi$ belonging to $L^{\infty}(\mathbb{R};\dot{H}^{1}(\mathbb{R}^{3}))$, \eqref{lemma3.15} and \eqref{obs7} ensure that 
		$$\|W_{n}^{(l)}\|^{10}_{L^{10}(I \times \mathbb{R}^{3})} \rightarrow \sum_{j=1}^{l} \|\psi^{(j)}\|_{L^{10}( \mathbb{R} \times \mathbb{R}^{3})}^{10}, \ \ \|\nabla W_{n}^{(l)}\|^{\frac{10}{3}}_{L^{\frac{10}{3}}(I \times \mathbb{R}^{3})} \rightarrow \sum_{j=1}^{l} \|\nabla\psi^{(j)}\|_{L^{\frac{10}{3}}( \mathbb{R} \times \mathbb{R}^{3})}^{\frac{10}{3}},$$
		as $n \rightarrow \infty$, for every $l$.
		Let us prove that the series $\sum_{j \geq 1} |||\psi^{(j)}|||^\frac{10}{3}_{ \mathbb{R}}$ are convergent. 
		To this end, first note that \eqref{eq 29} and Lemma \ref{lemma3.3} imply
		\begin{equation} \label{eq 75} 
			\sum_{j\geq 1} |||\varphi^{(j)}|||^{\frac{10}{3}}_{ \mathbb{R}}  = 	\sum_{j\geq 1} |||p_{n}^{(j)}|||^{\frac{10}{3}}_{I}\leq  C \sum_{j \geq 1} \|\nabla p_{n}^{(j)}(0)\|_{L^{2}_{x}}^{\frac{10}{3}} \leq C ,
		\end{equation}
		where we have used that the series $ \sum_{j \geq 1} \|\nabla p_{n}^{(j)}(0)\|_{L^{2}_{x}}^{\frac{10}{3}}$ is
		convergent. Thus, if
\begin{equation} \label{desigualdade1}
			|||\psi^{(j)}|||_{\mathbb{R}} \leq C |||\varphi^{(j)}|||_{\mathbb{R}},
\end{equation}
		for large enough $j$,  then $\sum_{j \geq 1} |||\psi^{(j)}|||^\frac{10}{3}_{ \mathbb{R}}$ is convergent. But from \eqref{eq 75}, one has that $|||\varphi^{(j)}|||_{\mathbb{R}} \leq \delta_{0}$,  for large enough $j$ large enough, since $|||\varphi^{(j)}|||_{\mathbb{R}} $ is the general term of a convergent series, where $\delta_{0}$ is as in Lemma \ref{lemma4.6}. Moreover,  $$\|\nabla (\psi^{(j)}-\varphi^{(j)})(-t_{n}^{(j)}/(h_{n}^{(j)})^{2})\|_{L^{2}_{x}}= \|\nabla (q_{n}^{(j)}-p_{n}^{(j)})(0)\|_{L^{2}_{x}} = 0.$$ Consequently, $\psi^{(j)}$ and $\varphi^{(j)}$ satisfy the conditions of Lemma \ref{lemma4.6} for large $j$, and, therefore, inequality \eqref{desigualdade1} holds. This finishes the proof. 
\end{proof} 
	\begin{proposition} \label{prop37}
		For every $\varepsilon > 0$, there exists an n-dependent finite partition 
		\begin{equation}\label{partition}
			[0,T] = \bigcup_{i=1}^{p} I_{n}^{i}
		\end{equation}
		such that
		\begin{equation}\label{eq 78}
			\limsup_{n \rightarrow \infty} \|W_{n}^{(l)} + w_{n}^{(l)}\|_{L^{10}(I_{n}^{i} \times \mathbb{R}^{3})} \leq \varepsilon,
		\end{equation}
		for all $1 \leq i \leq p$, $ l \geq 1$.
	\end{proposition}
	\begin{proof}
	Since $$ \limsup_{n \rightarrow \infty} \|w_{n}^{(l)}\|_{L^{10}(I \times \mathbb{R}^{3})} \longrightarrow 0 \mbox{ as } l\rightarrow \infty,$$
	given $\varepsilon >0$, there exists $l_{1} \geq 1$ such that
	\begin{equation} \label{eq 79}
		\limsup_{n \rightarrow \infty} \|w_{n}^{(l)}\|_{L^{10}(I \times \mathbb{R}^{3})} \leq \frac{\varepsilon}{2}
	\end{equation}
	if $l \geq l_{1}$. Moreover, by \eqref{lemma3.15}, there exists $l_2\geq 1$ such that 
	\begin{equation*} 
		\limsup_{n \rightarrow \infty} \|W_{n}^{(l)}\|_{L^{10}(I \times \mathbb{R}^{3})} \leq 	\limsup_{n \rightarrow \infty} \|W_{n}^{(l_{2})}\|_{L^{10}(I \times \mathbb{R}^{3})} + \Big( \sum_{j=l_{2}+1}^{l} \|\psi^{(j)}\|^{10}_{L^{10}(\mathbb{R}\times \mathbb{R}^{3})}\Big)^{\frac{1}{10}}
	\end{equation*} for all $l \geq l_{2} \geq 1$,
	Since the series $\sum_{j\geq 1} \|\psi^{(j)}\|^{10}_{L^{10}(\mathbb{R} \times \mathbb{R}^{3})}$ is convergent, we may choose $l_{2}$ such that
	\begin{equation}\label{eq 81}
		\Big(\sum_{j\geq l_{2}} \|\psi^{(j)}\|^{10}_{L^{10}(\mathbb{R} \times \mathbb{R}^{3})}\Big)^{\frac{1}{10}} \leq \frac{\varepsilon}{4}.
	\end{equation}
	Putting together \eqref{eq 79} and \eqref{eq 81}, it follows that
	\begin{equation*} 
	\begin{split}
		\limsup_{n \rightarrow \infty} \|W_{n}^{(l)} + w_{n}^{(l)}\|_{L^{10}(I \times \mathbb{R}^{3})}  \leq & 
		\limsup_{n \rightarrow \infty} \|W_{n}^{(l)}\|_{L^{10}(I \times \mathbb{R}^{3})} + \limsup_{n \rightarrow \infty} \|w_{n}^{(l)}\|_{L^{10}(I\times \mathbb{R}^{3})} \\
		 \leq & \limsup_{n \rightarrow \infty} \|W_{n}^{(l_{2})}\|_{L^{10}(I\times \mathbb{R}^{3})} + \Big( \sum_{j=l_{2}+1}^{l} \|\psi^{(j)}\|^{10}_{L^{10}(\mathbb{R}\times \mathbb{R}^{3})}\Big)^{\frac{1}{10}} + \frac{\varepsilon}{2}\\
		 \leq &  \limsup_{n \rightarrow \infty} \|W_{n}^{(l_{3})}\|_{L^{10}(I \times \mathbb{R}^{3})}  + \frac{3\varepsilon}{4},
		\end{split}
	\end{equation*}
	for every $l \geq l_{3} = \sup(l_{1} , l_{2})$. Considering the natural number $l_3$, we must construct  $l_3$ partial finite portion of $I$ for every $1  \leq j \leq  l_{3}$, and the global decomposition is obtained by intersecting all the partial ones. Note that the partition \eqref{partition} is needed for $n$ large, then in the next construction, we take $n$ large enough. 
	
	For $j=1$, we split the interval $[0,T] = I_{n}^{1,\Lambda} \cup I_{n}^{2,\Lambda} \cup I_{n}^{3,\Lambda}$ according to Theorem \ref{teo4.1}.
	\begin{itemize}
		\item[i.] For $(I_{n}^{1,\Lambda})$:
		Using Theorem \ref{teo4.1} and Lemma \ref{lemma3.10}, there exists a linear concentrating solution $p_{n}^{(1)}$ such that 
		$$\limsup_{n \rightarrow \infty} \|q_{n}^{(1)}\|_{L^{10}(I_{n}^{1,\Lambda}\times \mathbb{R}^{3})} \leq \|q_{n}^{(1)} - p_{n}^{(1)}\|_{L^{10}(I_{n}^{1,\Lambda}\times \mathbb{R}^{3})} + \|p_{n}^{(1)}\|_{L^{10}(I_{n}^{1,\Lambda}\times \mathbb{R}^{3})} \leq \frac{\varepsilon}{4l_{3}}.$$
		\item[ii.] For $(I_{n}^{3,\Lambda})$:
		 Analogously,  
		$$\limsup_{n \rightarrow \infty} \|q_{n}^{(1)}\|_{L^{10}(I_{n}^{3,\Lambda}\times \mathbb{R}^{3})} \leq \|q_{n}^{(1)} - p_{n}^{(1)}\|_{L^{10}(I_{n}^{3,\Lambda}\times \mathbb{R}^{3})} + \|p_{n}^{(1)}\|_{L^{10}(I_{n}^{3,\Lambda}\times \mathbb{R}^{3})} \leq \frac{\varepsilon}{4l_{3}}.$$
		\item[iii.] For $(I_{n}^{2,\Lambda})$:
		We have $I_{n}^{2,\Lambda} = [t_{n}^{(1)} - (h_{n}^{(1)})^{2}\Lambda, t_{n}^{(1)} + (h_{n}^{(1)})^{2}\Lambda ].$
		Therefore,
		$$\|q_{n}^{(1)}\|_{L^{10}(I_{n}^{2,\Lambda}\times \mathbb{R}^{3})} = \|\psi^{(1)}\|_{L^{10}([-\Lambda,\Lambda]\times \mathbb{R}^{3})}.$$
			\end{itemize}
		Once $\Lambda$ is fixed,
		divide $[-\Lambda,\Lambda]$ in a finite number of intervals $I^{(i),\Lambda}$ such that $$\|\psi^{(1)}\|_{L^{10}(I^{(i),\Lambda} \times \mathbb{R}^{3})} \leq \frac{\varepsilon}{4l_{3}}.$$ Therefore
		$$ \|q_{n}^{(1)}\|_{L^{10}(I_{n}^{(i),\Lambda}\times \mathbb{R}^{3})} = \|\psi^{(1)}\|_{L^{10}(I^{(i),\Lambda}\times \mathbb{R}^{3})} \leq  \frac{\varepsilon}{4l_{3}}.$$
This gives the decomposition for $j=1$. Analogously, we construct a partial decomposition for every $j=2, ..., l_{3}$. Finally, the global decomposition is obtained by intersecting all the partial ones. Hence,
		\begin{eqnarray*}
			\limsup_{n \rightarrow \infty} \|W_{n}^{(l)} + w_{n}^{(l)}\|_{L^{10}(I_{n}^{i} \times \mathbb{R}^{3})} & \leq &  \limsup_{n \rightarrow \infty} \|W_{n}^{(l_{3})}\|_{L^{10}(I_{n}^{i}  \times \mathbb{R}^{3})}  + \frac{3\varepsilon}{4} \\
			& \leq & \limsup_{n \rightarrow \infty} \sum_{j=1}^{l_{3}} \|q_{n}^{(j)}\|_{L^{10}(I_{n}^{i}  \times \mathbb{R}^{3})} + \frac{3\varepsilon}{4} \\
			& \leq &  \sum_{j=1}^{l_{3}}  \frac{\varepsilon}{4l_{3}} + \frac{3\varepsilon}{4} = \varepsilon,
		\end{eqnarray*}
proving \eqref{eq 78}.
\end{proof} 
	\begin{lemma} \label{lemma4.10}
		Let $\mathcal{B}$ be a compact set of $\mathbb{R} \times \mathbb{R}^{3}$. For every $\varepsilon > 0$, there exists a constant $C(\varepsilon)$ such that
		\begin{equation} \label{eq 90}
			\|\nabla v\|_{L^{2}(\mathcal{B})} \leq C(\varepsilon) \|v\|_{L^{10}(\mathbb{R} \times \mathbb{R}^{3})}  + \varepsilon \|\nabla v(0)\|_{L^{2}(\mathbb{R}^{3})},
		\end{equation}
		for all solutions $v$ of the linear Schr\"{o}dinger equation.
	\end{lemma}
	\begin{proof}
		We argue by contradiction. Suppose that \eqref{eq 90} does not hold. Then, there exist an $\varepsilon > 0$ and a sequence $(v_{m})$ of solutions of the linear Schr\"{o}dinger equation such that
		$$\|\nabla v_{m}\|_{L^{2}(\mathcal{B})}  > m \|v_{m}\|_{L^{10}(\mathbb{R} \times \mathbb{R}^{3})}  + \varepsilon \|\nabla v_{m}(0)\|_{L^{2}(\mathbb{R}^{3})}.$$
		Define $\tilde{v}_{m} := v_{m}/\|\nabla v_{m}\|_{L^{2}(\mathcal{B})}$. One has 
		$$1> m\|\tilde{v}_{m}\|_{L^{10}(\mathbb{R} \times \mathbb{R}^{3})}  + \varepsilon \|\nabla \tilde{v}_{m}(0)\|_{L^{2}(\mathbb{R}^{3})}.$$
Note that $ \|\nabla \tilde{v}_{m}(0)\|_{L^{2}(\mathbb{R}^{3})}$ is bounded and
$$
		m\|\tilde{v}_{m}\|_{L^{10}(\mathbb{R} \times \mathbb{R}^{3})}  <  -\varepsilon \|\nabla \tilde{v}_{m}(0)\|_{L^{2}(\mathbb{R}^{3})} +  1,
$$
thus	
		\begin{equation}\label{eq 91}
			\|\tilde{v}_{m}\|_{L^{10}(\mathbb{R} \times \mathbb{R}^{3})}  \longrightarrow 0, \ m\rightarrow \infty.
		\end{equation}
By Strichartz estimates, $\|\nabla \tilde{v}_{m}\|_{L^{\frac{10}{3}}(\mathbb{R} \times \mathbb{R}^{3})} \leq C  \|\nabla \tilde{v}_{m}(0)\|_{L^{2}(\mathbb{R}^{3})}.$ So, we conclude that  $\|\nabla \tilde{v}_{m}\|_{L^{\frac{10}{3}}(\mathbb{R} \times \mathbb{R}^{3})}$ is also bounded. In view of \eqref{eq 91} and \cite[Lemma 3.23]{merle}, there exists a subsequence of $(\tilde{v}_{m})$, which we keep denoted by the same index, such that
$$
			\nabla \tilde{v}_{m} \rightharpoonup 0 \ \ \mbox{ weakly in } \  L^{\frac{10}{3}}(\mathbb{R} \times \mathbb{R}^{3}).
$$
Setting $\psi_{m} = \nabla \tilde{v}_{m}(0,\cdot)$, we get
$$
			\|e^{it\Delta} \psi_{m}\|_{L^{2}(\mathcal{B})} = 	\|e^{it\Delta} \nabla \tilde{v}_{m}(0)\|_{L^{2}(\mathcal{B})} =\dfrac{\|\nabla v_{m}(0)\|_{L^{2}(\mathcal{B})}}{\|\nabla v_{m}(t)\|_{L^{2}(\mathcal{B})}} = 1.
$$
But, up to a subsequence, $\psi_{m} \rightharpoonup 0 \ \mbox{in} \ L^{2}( \mathbb{R}^{3})$, which is a contradiction. Therefore, \eqref{eq 90} holds.
	\end{proof}
	
The previous lemma gives the following proposition, which guarantees the smallness, for large $n$ and $l$, of 
\begin{equation*} 
	\delta_{n}^{(l)} = \Big\|\nabla\Big[\beta(W_{n}^{(l)} + w_{n}^{(l)}) - \beta(W_{n}^{(l)})\Big] \Big\|_{L^{\frac{10}{7}}(I\times \mathbb{R}^{3})} +  \Big\| \nabla \Big( \sum_{j=1}^{l} \beta(q_{n}^{(j)}) - \beta(W_{n}^{(l)})\Big) \Big\|_{L^{\frac{10}{7}}(I\times \mathbb{R}^{3})}.
\end{equation*}
\begin{proposition} \label{prop4.9}We have that
	\begin{equation}\label{eq 83}
		\limsup_{n \rightarrow \infty} \delta_{n}^{(l)} \longrightarrow 0 \quad \text{as}\quad l\rightarrow \infty.
	\end{equation}
\end{proposition}
\begin{proof}
	We split the proof into two parts. The first one is devoted to proving that for every $l \geq 1$, one has
	\begin{equation}\label{eq 84}
		\Big\|\nabla \Big(\sum_{j=1}^{l} \beta(q_{n}^{(j)}) - \beta(W_{n}^{(j)})\Big) \Big\|_{L^{\frac{10}{7}}(I \times \mathbb{R}^{3})} \longrightarrow 0 \mbox{ as } n\rightarrow \infty.
	\end{equation}
	In the second part, we will show that
	\begin{equation} \label{eq 85}
		\limsup_{n \rightarrow \infty}\|\nabla(\beta(W_{n}^{(l)} + w_{n}^{(l)}) - \beta(W_{n}^{(l)})) \|_{L^{\frac{10}{7}}(I \times \mathbb{R}^{3})}  \longrightarrow 0 \mbox{ as }l\rightarrow \infty.
	\end{equation}

	\noindent\textit{\textbf{Part 1.}}  Note that
	$$\Big\|\nabla \Big(\sum_{j=1}^{l} \beta(q_{n}^{(j)}) - \beta(W_{n}^{(j)})\Big) \Big\|_{L^{\frac{10}{7}}(I \times \mathbb{R}^{3})} \leq C D_{n},$$
	where $D_{n}=\|\nabla (q_{n}^{(j_{1})} q_{n}^{(j_{2})} q_{n}^{(j_{3})} q_{n}^{(j_{4})} q_{n}^{(j_{5})})\|_{L^{\frac{10}{7}}(I\times \mathbb{R}^{3})}$, with at least two differents $j_k$, for k=1,2,3,4,5. 
	Now, we want to prove that
	$$D_{n} \longrightarrow 0, \ n\rightarrow \infty.$$
Assuming, for example, that $j_{1} \neq j_{2}$, we have
	\begin{equation}\label{86}
	\begin{split}
		D_{n}^{\frac{10}{7}} = &\int_{\mathbb{R}} \int_{\mathbb{R}^{3}} |\nabla q_{n}^{(j_{1})}q_{n}^{(j_{2})}(q_{n}^{(j_{k})})^{3}|^{\frac{10}{7}} \ dxdt \\ 
		\leq & C \int_{\mathbb{R}} \int_{\mathbb{R}^{3}} |\nabla q_{n}^{(j_{1})}q_{n}^{(j_{2})} |^{\frac{10}{7}} |q_{n}^{(j_{k})}|^{\frac{30}{7}}\ dxdt \\    & \mbox{    } + C\int_{\mathbb{R}} \int_{\mathbb{R}^{3}}|q_{n}^{(j_{1})}q_{n}^{(j_{2})} |^{\frac{10}{7}} |\nabla (q_{n}^{(j_{k})})^{3}|^{\frac{10}{7}}  dxdt.
		\end{split}
	\end{equation}
	To bound the first integral on the right-hand side of the inequality above, use Hölder's inequality to get
	\begin{eqnarray*}
		\int_{\mathbb{R}} \int_{\mathbb{R}^{3}} |\nabla q_{n}^{(j_{1})}q_{n}^{(j_{2})} |^{\frac{10}{7}} |q_{n}^{(j_{k})}|^{\frac{30}{7}}\ dxdt  
		& \leq & C \|q_{n}^{(j_{k})}\|^{3}_{L^{10}(\mathbb{R}^{4})} \Biggr(\int_{\mathbb{R}} \int_{\mathbb{R}^{3}}|\nabla q_{n}^{(j_{1})}q_{n}^{(j_{2})} |^{\frac{10}{4}} \ dxdt\Biggr)^{\frac{4}{7}}\\
		& \leq & C \Biggr(\int_{\mathbb{R}} \int_{\mathbb{R}^{3}}|\nabla q_{n}^{(j_{1})}q_{n}^{(j_{2})} |^{\frac{5}{2}} \ dxdt\Biggr)^{\frac{4}{7}}.
	\end{eqnarray*}
	This last term can be written as
	\begin{eqnarray*}
		\frac{1}{(h_{n}^{(j_{1})} h_{n}^{(j_{2})})^{\frac{5}{7}}}\Biggr(\int_{\mathbb{R}} \int_{\mathbb{R}^{3}}\Bigg|\nabla_{x} \psi^{(j_{1})}\Bigg(\frac{t-t_{n}^{(j_{1})}}{(h_{n}^{(j_{1})})^{2}},\frac{x-x_{n}^{(j_{1})}}{h_{n}^{(j_{1})}}\Bigg)\psi^{(j_{2})} \Bigg(\frac{t-t_{n}^{(j_{2})}}{(h_{n}^{(j_{2})})^{2}},\frac{x-x_{n}^{(j_{2})}}{h_{n}^{(j_{2})}}\Bigg)\Bigg|^{\frac{5}{2}} \ dxdt\Biggr)^{\frac{4}{7}}.\\
	\end{eqnarray*}
	Since $[h_{n}^{(j_{1})},\underline{x}^{(j_{1})},\underline{t}^{(j_{1})}],[h_{n}^{(j_{2})},\underline{x}^{(j_{2})},\underline{t}^{(j_{2})}]$  are orthogonal, assume  $\psi^{j_{1}}, \psi^{j_{2}}$ to be continuous and compactly supported an analyze the possible cases: 
	\begin{itemize}
		\item[$\bullet$] If $ \frac{h_{n}^{(j_{1})}}{h_{n}^{(j_{2})}} + \frac{h_{n}^{(j_{2})}}{h_{n}^{(j_{1})}} \longrightarrow +\infty,$ assume, for example, that $\frac{h_{n}^{(j_{1})}}{h_{n}^{(j_{2})}} \longrightarrow + \infty$ (the other case is analogous).
	\end{itemize}
	Using the change of variables $t=s(h_{n}^{(j_{2})})^{2} + t_{n}^{(j_{2})},$ $x=yh_{n}^{(j_{2})} + x_{n}^{(j_{2})},$ we have 
	\begin{equation*}
	\begin{split}
		&\frac{1}{(h_{n}^{(j_{1})} h_{n}^{(j_{2})})^{\frac{5}{7}}}\Biggr( \int_{\mathbb{R}^{4}}\Bigg|\nabla_{x} \psi^{(j_{1})}\Bigg(\frac{t_{n}^{(j_{2})}-t_{n}^{(j_{1})}}{(h_{n}^{(j_{1})})^{2}} + s\frac{(h_{n}^{(j_{2})})^{2}}{(h_{n}^{(j_{1})})^{2}},\frac{x_{n}^{(j_{2})}-x_{n}^{(j_{1})}}{h_{n}^{(j_{1})}} + y \frac{h_{n}^{(j_{2})}}{h_{n}^{(j_{1})}}\Bigg)\psi^{(j_{2})} (s,y)\Bigg|^{\frac{5}{2}} \ dyds (h_{n}^{(j_{2})})^{5}\Biggr)^{\frac{4}{7}}\\
		=& \frac{(h_{n}^{(j_{2})})^{\frac{5}{7}}}{(h_{n}^{(j_{1})})^{\frac{5}{7}}}\Biggr( \int_{\mathbb{R}^{4}}\Bigg|\nabla_{y} \psi^{(j_{1})}\Bigg(\frac{t_{n}^{(j_{2})}-t_{n}^{(j_{1})}}{(h_{n}^{(j_{1})})^{2}} + s\frac{(h_{n}^{(j_{2})})^{2}}{(h_{n}^{(j_{1})})^{2}},\frac{x_{n}^{(j_{2})}-x_{n}^{(j_{1})}}{h_{n}^{(j_{1})}} + y \frac{h_{n}^{(j_{2})}}{h_{n}^{(j_{1})}}\Bigg)\psi^{(j_{2})} (s,y)\Bigg|^{\frac{5}{2}} \ dyds \Biggr)^{\frac{4}{7}}
		 \rightarrow  0.
		\end{split}
	\end{equation*}
	\begin{itemize}
		\item[$\bullet$] If $h_{n}^{(j_{1})} =  h_{n}^{(j_{2})}$, using the same change of variables as above, we get
	\end{itemize}
	\begin{eqnarray*}
		\Biggr(\int_{\mathbb{R}} \int_{\mathbb{R}^{3}}\Bigg|\nabla_{y} \psi^{(j_{1})}\Bigg(\frac{t_{n}^{(j_{2})}-t_{n}^{(j_{1})}}{(h_{n}^{(j_{1})})^{2}} + s\frac{(h_{n}^{(j_{2})})^{2}}{(h_{n}^{(j_{1})})^{2}},\frac{x_{n}^{(j_{2})}-x_{n}^{(j_{1})}}{h_{n}^{(j_{1})}} + y \frac{h_{n}^{(j_{2})}}{h_{n}^{(j_{1})}}\Bigg)\psi^{(j_{2})} (s,y)\Bigg|^{\frac{5}{2}} \ dyds \Biggr)^{\frac{4}{7}}.\\
	\end{eqnarray*}
	Since $ \Big| \frac{t_{n}^{(j_{1})}-t_{n}^{(j_{2})}}{{h_{n}^{(j_{1})}}^{2}}\Big| + \Big|\frac{x_{n}^{(j_{1})}-x_{n}^{(j_{2})}}{h_{n}^{(j_{1})}}\Big| \longrightarrow + \infty$ as  $n\rightarrow \infty$, this integral tends to 0, which ensures that the first integral on the right-hand side of \eqref{86} converges to 0.
	
	Now, we examine the second integral on the right-hand side of \eqref{86}. Again, Hölder's inequality ensures that\
	\begin{eqnarray*}
		\int_{\mathbb{R}} \int_{\mathbb{R}^{3}}|q_{n}^{(j_{1})}q_{n}^{(j_{2})} |^{\frac{10}{7}} |\nabla (q_{n}^{(j_{k})})^{3}|^{\frac{10}{7}}  dxdt 
		& \leq & C\|q_{n}^{(j_{k})}\|^{\frac{20}{7}}_{L^{10}(\mathbb{R}^{4})}\|\nabla q_{n}^{(j_{k})}\|_{L^{\frac{10}{3}}(\mathbb{R}^{4})}^{\frac{10}{7}}\Big(\int_{\mathbb{R}} \int_{\mathbb{R}^{3}}|q_{n}^{(j_{1})}q_{n}^{(j_{2})} |^{5} \ dxdt\Big)^{\frac{2}{7}}\\
		& \leq & C\Big(\int_{\mathbb{R}} \int_{\mathbb{R}^{3}}|q_{n}^{(j_{1})}q_{n}^{(j_{2})} |^{5} \ dxdt\Big)^{\frac{2}{7}}
	\end{eqnarray*}
	and
	\begin{eqnarray*}
		&  & \Big(\int_{\mathbb{R}} \int_{\mathbb{R}^{3}}|q_{n}^{(j_{1})}q_{n}^{(j_{2})} |^{5} \ dxdt\Big)^{\frac{2}{7}} \\
		&=& \frac{1}{(h_{n}^{(j_{1})}h_{n}^{(j_{2})})^{\frac{5}{7}}}\Bigg(\int_{\mathbb{R}} \int_{\mathbb{R}^{3}} \Bigg| \psi^{(j_{1})}\Bigg(\frac{t-t_{n}^{(j_{1})}}{(h_{n}^{(j_{1})})^{2}},\frac{x-x_{n}^{(j_{1})}}{h_{n}^{(j_{1})}}\Bigg)\psi^{(j_{2})} \Bigg(\frac{t-t_{n}^{(j_{2})}}{(h_{n}^{(j_{2})})^{2}},\frac{x-x_{n}^{(j_{2})}}{h_{n}^{(j_{2})}}\Bigg)\Bigg|^{5} \ dxdt\Bigg)^{\frac{2}{7}}.
	\end{eqnarray*}
	Analogously to the previous case, one concludes that the second integral on the right-hand side of \eqref{86} converges to 0 as well, which shows the convergence \eqref{eq 84}.
	
	\vspace{0.2cm}
	
	\noindent\textit{\textbf{Part 2.}} By Leibnitz formula and H\"{o}lder inequality, we get
	\begin{eqnarray*}
		\|\nabla (\beta(W_{n}^{(l)} + w_{n}^{(l)}) - \beta(W_{n}^{(l)}))\|_{L^{\frac{10}{7}}(I\times \mathbb{R}^{3})} & \leq & C \Big(\|w_{n}^{(l)}\|_{L^{10}(I\times \mathbb{R}^{3})} |||W_{n}^{(l)} + w_{n}^{(l)}|||^{4}_{I} \\
		& & \mbox{ } + |||W_{n}^{(l)}|||^{3}_{I} \|W_{n}^{(l)} \nabla w_{n}^{(l)}\|_{L^{\frac{5}{2}}(I\times \mathbb{R}^{3})} \Big).
	\end{eqnarray*}
	Since that \eqref{eq 28} and \eqref{eq 74} hold, if we prove that
$$
		\limsup_{n \rightarrow \infty} \|W_{n}^{(l)} \nabla w_{n}^{(l)}\|_{L^{\frac{5}{2}}(I \times \mathbb{R}^{3})} \longrightarrow 0 \mbox{ as } l\rightarrow \infty,
$$
	then the proof of \eqref{eq 85} is complete. Indeed, the convergence of the series $\sum_{j \geq 1} \|\psi^{(j)}\|^{10}_{L^{10}(\mathbb{R} \times \mathbb{R}^{3})}$ implies that, for every $\varepsilon > 0$, there exists $l(\varepsilon)$ such that 
$$
		\sum_{j \geq l(\varepsilon)}\|\psi^{(j)}\|^{10}_{L^{10}(\mathbb{R}\times \mathbb{R}^{3})} \leq \varepsilon^{10}.
$$
	In particular, using H\"{o}lder's inequality,
	\begin{eqnarray*}
		\limsup_{n \rightarrow \infty} \Big\| \Big(\sum_{j=l(\varepsilon)}^{l} q_{n}^{(j)}\Big) \nabla w_{n}^{(l)} \Big\|^{10}_{L^{\frac{5}{2}}(I \times \mathbb{R}^{3})} & = & \limsup_{n \rightarrow \infty} \Big\|\sum_{j=l(\varepsilon)}^{l} q_{n}^{(j)} \Big\|^{10}_{L^{10}(I\times \mathbb{R}^{3})} \limsup_{n \rightarrow \infty} \|\nabla w_{n}^{(l)}\|^{10}_{L^{\frac{10}{3}}(I \times \mathbb{R}^{3})} \\
		& \leq & \sum_{j\geq l(\epsilon)} \|\psi^{(j)}\|^{10}_{L^{10}(\mathbb{R} \times \mathbb{R}^{3})}  \limsup_{n \rightarrow \infty} \|\nabla w_{n}^{(l)}\|^{10}_{L^{\frac{10}{3}}(I \times \mathbb{R}^{3})} \\
		& \leq & C \varepsilon^{10},
	\end{eqnarray*}
	where the last inequality follows from the fact that $\|\nabla w_{n}^{(l)}\|^{10}_{L^{\frac{10}{3}}(I\times \mathbb{R}^{3})} $ is uniformly bounded by Strichartz estimates. Therefore,
	\begin{equation*}
	\begin{split}
		\limsup_{n \rightarrow \infty} \|W_{n}^{(l)} \nabla w_{n}^{(l)}\|_{L^{\frac{5}{2}}(I \times \mathbb{R}^{3})}   =  & 	\limsup_{n \rightarrow \infty} \Big\|\Big(\sum_{j=1}^{l}q_{n}^{(j)}\Big) \nabla w_{n}^{(l)}\Big\|_{L^{\frac{5}{2}}(I \times \mathbb{R}^{3})} \\
		\leq & \limsup_{n \rightarrow \infty} \Big\|\Big(\sum_{j=1}^{l(\varepsilon)}q_{n}^{(j)}\Big) \nabla w_{n}^{(l)}\Big\|_{L^{\frac{5}{2}}(I \times \mathbb{R}^{3})}  \\&+  \limsup_{n \rightarrow \infty} \Big\|\Big(\sum_{j=l(\varepsilon)}^{l}q_{n}^{(j)}\Big) \nabla w_{n}^{(l)}\Big\|_{L^{\frac{5}{2}}(I \times \mathbb{R}^{3})} \\
		 \leq & 
		\limsup_{n \rightarrow \infty} \|W_{n}^{(l(\varepsilon))} \nabla w_{n}^{(l)}\|_{L^{\frac{5}{2}}(I \times \mathbb{R}^{3})} + C \varepsilon,
	\end{split}
	\end{equation*}
	for $l\geq l(\varepsilon)$.  Hence, our problem is reduced to prove that
	$$\limsup_{n \rightarrow \infty} \|W_{n}^{(l_{0})} \nabla w_{n}^{(l)}\|_{L^{\frac{5}{2}}(I\times \mathbb{R}^{3})} ,\longrightarrow 0 \mbox{ as } l\rightarrow \infty$$
	for every fixed $l_{0} \geq 1$.  Since $W_{n}^{(l_{0})} = \sum_{j=1}^{l_{0}} q_{n}^{(j)}$, we have to show that
$$
		\limsup_{n \rightarrow \infty} \|q_{n}^{(j)} \nabla w_{n}^{(l)} \|_{L^{\frac{5}{2}}(I \times \mathbb{R}^{3})} \longrightarrow 0 \mbox{ as } l\rightarrow \infty,
$$
	for every $l_{0} \geq j \geq 1$, i.e., 
	$$\limsup_{n \rightarrow \infty} \Big\|\frac{1}{\sqrt{h_{n}^{(j)}}} \psi^{(j)} \Big(\frac{t-t_{n}^{(j)}}{(h_{n}^{(j)})^{2}},\frac{x-x_{n}^{(j)}}{h_{n}^{(j)}}\Big) \nabla w_{n}^{(l)} \Big\|_{L^{\frac{5}{2}}(I \times \mathbb{R}^{3})} \longrightarrow 0 \mbox{ as }l\rightarrow \infty.$$
	To this end, change variables to get 
	$$\|q_{n}^{(j)} \nabla w_{n}^{(l)} \|_{L^{\frac{5}{2}}(I \times \mathbb{R}^{3})} = \|\psi^{(j)} \nabla \tilde{w}_{n}^{(l)} \|_{L^{\frac{5}{2}}(\mathbb{R}\times \mathbb{R}^{3})}$$
	where
	$$\tilde{w}_{n}^{(l)}(s,y) = \sqrt{h_{n}^{(j)}} w_{n}^{(l)}(t_{n}^{(j)} + (h_{n}^{(j)})^{2}s, x_{n}^{(j)} + h_{n}^{(j)}y).$$
	Observe that, by Lemma \ref{lemma3.3},
	$$\|w_{n}^{(l)}\|_{L^{10}(I \times \mathbb{R}^{3})}  = \|\tilde{w}_{n}^{(l)}\|_{L^{10}(\mathbb{R}\times \mathbb{R}^{3})} \ \mbox{ and } \ \ \|\nabla w_{n}^{(l)}\|_{L^{\frac{10}{3}}(I\times \mathbb{R}^{3})} = \|\nabla \tilde{w}_{n}^{(l)}\|_{L^{\frac{10}{3}}(\mathbb{R}\times \mathbb{R}^{3})}. $$
	By density, we can take $\psi^{(j)} \in C_{0}^{\infty}(\mathbb{R}^{4})$. Using H\"{o}lder's inequality, one sees that it is enough to prove that
	\begin{equation} \label{eq 89}
		\limsup_{n \rightarrow \infty} \|\nabla \tilde{w}_{n}^{(l)}\|_{L^{2}(\mathcal{B})} \longrightarrow 0 \mbox{ as } l\rightarrow \infty,
	\end{equation}
	where $\mathcal{B}$ is a fixed compact of $\mathbb{R} \times \mathbb{R}^{3}$. To this end, let $\nu_{n}^{l}$ be the function defined by 
		\begin{equation*}
	\nu_{n}^{l}(t,x)=	\left\{
		\begin{array}{ll}
		w_{n}^{(l)}(t,x), &\mbox{  if  } (t,x) \in \mathcal{B} ,  \\
			0, & \mbox{otherwise}. 
		\end{array}
		\right.
	\end{equation*}
	Then, $\nu_{n}^{l}$ is a solution for the linear Schrödinger equation, and we get, by Strichartz estimates, that
	\begin{eqnarray*}
		\limsup_{n \rightarrow \infty}\|\psi^{(j)} \nabla \tilde{w}_{n}^{(l)} \|_{L^{\frac{5}{2}}(\mathbb{R}\times \mathbb{R}^{3})} & \leq & 	\limsup_{n \rightarrow \infty}\|\psi^{(j)} \nabla \tilde{w}_{n}^{(l)} \|_{L^{\frac{5}{2}}(\mathcal{B})}\\ 
		& \leq & \limsup_{n \rightarrow \infty}	\|\psi^{(j)} \|_{L^{10}(\mathcal{B})} \|\nabla \tilde{w}_{n}^{(l)}\|_{L^{\frac{10}{3}}(\mathcal{B})}\\
		& \leq &  \limsup_{n \rightarrow \infty}	\|\psi^{(j)} \|_{L^{10}(\mathcal{B})} \|\nabla \tilde{\nu}_{n}^{(l)}\|_{L^{\frac{10}{3}}(\mathbb{R}^{4})} \\
		& \leq &  \limsup_{n \rightarrow \infty}	\|\psi^{(j)} \|_{L^{10}(\mathcal{B})} \|\nabla \tilde{\nu}_{n}^{(l)}\|_{L^{2}(\mathcal{B})} \\
		& \leq &  \limsup_{n \rightarrow \infty}	\|\psi^{(j)} \|_{L^{10}(\mathcal{B})} \|\nabla \tilde{w}_{n}^{(l)}\|_{L^{2}(\mathcal{B})}.
	\end{eqnarray*}
Applying Lemma \ref{lemma4.10} to $\tilde{w}_{n}^{(l)}$ gives
	$$	\|\nabla \tilde{w}_{n}^{(l)}\|_{L^{2}(\mathcal{B})} \leq C(\varepsilon) \|\tilde{w}_{n}^{(l)}\|_{L^{10}(\mathbb{R}\times \mathbb{R}^{3})}  + \varepsilon \|\nabla \tilde{w}_{n}^{(l)}(0)\|_{L^{2}(\mathbb{R}^{3})}.$$
	The invariance of the $L^{10}$ and $\dot{H}^{1}$ norms by the change of variables gives
	$$	\|\nabla \tilde{w}_{n}^{(l)}\|_{L^{2}(\mathcal{B})} \leq C(\varepsilon) \|w_{n}^{(l)}\|_{L^{10}(I\times \mathbb{R}^{3})}  + \varepsilon \|\nabla w_{n}^{(l)}(0)\|_{L^{2}(\mathbb{R}^{3})}.$$
	So, it follows that
	$$\limsup_{l \rightarrow \infty} \|\nabla \tilde{w}_{n}^{(l)}\|_{L^{2}(\mathcal{B})}\leq C \varepsilon.$$
	Since $\varepsilon$ is arbitrary, \eqref{eq 89} holds. This concludes the proof of Proposition \ref{prop4.9}.
\end{proof}
			
\subsection{Proof of the decomposition} We finally prove Theorem \ref{nonlineardec} following the ideas introduced by S. Keraani in \cite{keraani}. 

First of all, note that the nonlinear profile $q_{n}^{(j)}$ is globally well-defined. Indeed, 
for a bounded sequence $(\varphi_{n})\in\dot{H}(\mathbb{R}^{3})$ such that $\limsup_{n \rightarrow \infty}\|\varphi_{n}\|_{\dot{H}^{1}} < \lambda_{0},$ with $\lambda_{0}$ given by Definition \ref{def3}, and $(v_{n})$ (respectively $(u_{n}))$ the sequence of solutions for the linear equation (respectively nonlinear) with initial data $\varphi_{n}$,  Theorem \ref{lineardecomp} provides a decomposition of $v_{n}$ into linear concentrating solutions $p_{n}^{(j)}$. Thus, if we consider $u_{n}$ a sequence of solutions for the nonlinear equation with the same initial data of $v_{n}$ at $t=0$ and $q_{n}^{(j)}$ the nonlinear concentrating solutions associated to $p_{n}^{(j)}$ for every $j\geq 1$,  we have
	$$\|\nabla q_{n}^{(j)}(0)\|^{2}_{L^{2}} = \|\nabla p_{n}^{(j)}(0)\|^{2}_{L^{2}} \leq \limsup_{n \rightarrow \infty} \|\nabla v_{n}(0)\|_{L^{2}}^{2} \leq \|\nabla \varphi_{n}\|_{L^{2}}^{2} \leq \lambda_{0}^{2},$$
due to the almost orthogonality identity \eqref{eq 29}. With this in hand, we are in a position to prove the Theorem \ref{nonlineardec}.
	\begin{proof}[Proof of Theorem \ref{nonlineardec}]
Let us consider 
	$r_{n}^{(l)}(t,x) = u_{n}(t,x) - \sum_{j=1}^{l} q_{n}^{(j)}(t,x) - w_{n}^{(l)}(t,x).$
	We need to prove the convergence
		$$	\limsup_{n \rightarrow \infty} (\|\nabla r_{n}^{(l)}\|_{L^{\frac{10}{3}}([0,T];L^{\frac{10}{3}}(\mathbb{R}^{3}))}  + \|r_{n}^{(l)}\|_{L^{10}([0,T];L^{10}(\mathbb{R}^{3}))} + \|r_{n}^{(l)}\|_{L^{\infty}([0,T];\dot{H}^{1}(\mathbb{R}^{3}))} ) \longrightarrow 0 \mbox{ as } \ l\rightarrow \infty.$$
To this end, let $\beta(z) = |z|^{4}z$, $W_{n}^{(l)} = \sum_{j=1}^{l} q_{n}^{(j)},$ and 
	$$f_{n}^{(l)} = \sum_{j=1}^{l} \beta(q_{n}^{(j)}) - \beta \Big( \sum_{j=1}^{l} q_{n}^{(j)} + w_{n}^{(l)} + r_{n}^{(l)}\Big).$$
		The function $r_{n}^{(l)}$ satisfies
	\begin{equation*}
		\left\{
		\begin{array}{lr}
			i\partial_{t} r_{n}^{(l)} + \Delta r_{n}^{(l)}  = f_{n}^{(l)}, \\
			r_{n}^{(l)}(0)= \sum_{j=1}^{l}(p_{n}^{(j)}-q_{n}^{(j)})(0)=0.
		\end{array}
		\right.
	\end{equation*}
	Introduce the norm
	$$|||g|||_{I} = \|g\|_{L^{10}(I \times \mathbb{R}^{3})} + \|\nabla g\|_{L^{\frac{10}{3}}(I \times \mathbb{R}^{3})} .$$
Note that, by Strichartz estimates, for any $v$ solution of the linear Schr\"{o}dinger equation with initial data $\varphi \in \dot{H}^{1}$, one has
$$
		|||v|||_{I} = \|v\|_{L^{10}_{t}L^{10}_{x}} +\|\nabla v\|_{L^{\frac{10}{3}}_{t}L^{\frac{10}{3}}_{x}}  \leq  C \|\nabla e^{it\Delta }\varphi\|_{L^{2}_{x}}	\leq C \|\nabla \varphi\|_{L^{2}_{x}}.
$$
From now on, $\gamma_{n}^{(l)}(a) = \|\nabla r_{n}^{(l)}(a)\|_{L^{2}_{x}},$ for every $a \in [0, T]$. Applying Lemma \ref{prop4.4} to $r_{n}^{(l)}$ on $I=[0,T]$, we obtain
	\begin{equation}\label{rrrrrrraa} 
	|||r_{n}^{(l)}|||_{I} + \sup_{t \in I} \|\nabla r_{n}^{(l)}(t)\|_{L^{2}} \leq C \Big(  \|\nabla f_{n}^{(l)}\|_{L^{\frac{10}{7}}(I\times \mathbb{R}^{3})}\Big).
	\end{equation}
	We estimate the right-hand side of inequality \eqref{rrrrrrraa} by
	\begin{equation} \label{eq 70}
	\begin{split}
		\|\nabla f_{n}^{(l)}\|_{L^{\frac{10}{7}}(I\times \mathbb{R}^{3})} \leq & \Big\| \nabla \Big( \sum_{j=1}^{l} \beta(q_{n}^{(j)}) - \beta(W_{n}^{(l)})\Big) \Big\|_{L^{\frac{10}{7}}(I\times \mathbb{R}^{3})} \\
		 & \mbox{ } + \Big\|\nabla\Big[\beta(W_{n}^{(l)} + w_{n}^{(l)}) - \beta(W_{n}^{(l)})\Big] \Big\|_{L^{\frac{10}{7}}(I\times \mathbb{R}^{3})}  \\
		& \mbox{ } + \Big\|\nabla\Big[\beta(W_{n}^{(l)} + w_{n}^{(l)} + r_{n}^{(l)}) - \beta(W_{n}^{(l)} + w_{n}^{(l)})\Big] \Big\|_{L^{\frac{10}{7}}(I\times \mathbb{R}^{3})} .
		\end{split}
	\end{equation}
	Furthermore, a combination of Leibnitz formula and H\"{o}lder inequality gives that \eqref{eq 70} can be bounded as
	\begin{equation} \label{eq 71}
		\begin{split}
		&\Big\|\nabla\Big[\beta(W_{n}^{(l)} + w_{n}^{(l)} + r_{n}^{(l)}) - \beta(W_{n}^{(l)} + w_{n}^{(l)})\Big] \Big\|_{L^{\frac{10}{7}}(I\times \mathbb{R}^{3})} \leq\\& C\Big(|||W_{n}^{(l)} + w_{n}^{(l)}|||^{3}_{I} \| W_{n}^{(l)} + w_{n}^{(l)}\|_{L^{10}(I \times \mathbb{R}^{3})} |||r_{n}^{(l)}|||_{I} + \sum_{\alpha =2}^{5}|||W_{n}^{(l)} + w_{n}^{(l)}|||^{5-\alpha}_{I} |||r_{n}^{(l)}|||^{\alpha}_{I}\Big).
		\end{split}
	\end{equation}
	Denote
	\begin{equation} \label{eq 72}
		\delta_{n}^{(l)} = \Big\|\nabla\Big[\beta(W_{n}^{(l)} + w_{n}^{(l)}) - \beta(W_{n}^{(l)})\Big] \Big\|_{L^{\frac{10}{7}}(I\times \mathbb{R}^{3})} +  \Big\| \nabla \Big( \sum_{j=1}^{l} \beta(q_{n}^{(j)}) - \beta(W_{n}^{(l)})\Big) \Big\|_{L^{\frac{10}{7}}(I\times \mathbb{R}^{3})}.
	\end{equation}
	Using \eqref{eq 70}, \eqref{eq 71} and \eqref{eq 72} into \eqref{rrrrrrraa}, it follows that
		\begin{equation}\label{eq 73}
			\begin{split}
		|||r_{n}^{(l)}|||_{I} + \sup_{t \in I} \|\nabla r_{n}^{(l)}(t)\|_{L^{2}}  \leq & C \Big(  \delta_{n}^{(l)} + \sum_{\alpha =2}^{5}|||W_{n}^{(l)} + w_{n}^{(l)}|||^{5-\alpha}_{I} |||r_{n}^{(l)}|||^{\alpha}_{I}\\
		& \mbox{ } +  |||W_{n}^{(l)} + w_{n}^{(l)}|||^{3}_{I} \| W_{n}^{(l)} + w_{n}^{(l)}\|_{L^{10}(I \times \mathbb{R}^{3})} |||r_{n}^{(l)}|||_{I} \Big).
			\end{split}
	\end{equation}
In view of bound \eqref{eq 73} and Proposition \ref{prop4.5}, we get
\begin{equation} \label{eq 77}
	|||r_{n}^{(l)}|||_{I} + \sup_{t \in I} \|\nabla r_{n}^{(l)}(t)\|_{L^{2}} 
	\leq  C \Big( \gamma_{n}^{(l)}(a) + \delta_{n}^{(l)} + \sum_{\alpha =2}^{5}|||r_{n}^{(l)}|||^{\alpha}_{I} + \| W_{n}^{(l)} + w_{n}^{(l)}\|_{L^{10}(I \times \mathbb{R}^{3})} |||r_{n}^{(l)}|||_{I} \Big)
\end{equation}
for all $l\geq1$ and $n \geq N(l)$.
Applying \eqref{eq 77} on an interval $I_{n}^{i}$,  provided by Proposition \ref{prop37}, one gets 
\begin{equation*}
	|||r_{n}^{(l)}|||_{I_{n}^{i}} + \sup_{t \in I_{n}^{i}} \|\nabla r_{n}^{(l)}(t)\|_{L^{2}} 
	\leq  C \Big( \gamma_{n}^{(l)}(a_{n}^{i}) + \delta_{n}^{(l)} + \sum_{\alpha =2}^{5}|||r_{n}^{(l)}|||^{\alpha}_{I_{n}^{i}} + 2\varepsilon |||r_{n}^{(l)}|||_{I_{n}^{i}} \Big),
\end{equation*}
for all $l\geq1$ and $n\geq N(l)$. So, choosing $\varepsilon$ so that $C \epsilon < \frac{1}{4}$, we obtain
\begin{equation} \label{eq 93}
	|||r_{n}^{(l)}|||_{I_{n}^{i}} + \sup_{t \in I_{n}^{i}} \|\nabla r_{n}^{(l)}(t)\|_{L^{2}} 
	\leq  C \Big( \gamma_{n}^{(l)}(a_{n}^{i}) + \delta_{n}^{(l)} + \sum_{\alpha =2}^{5}|||r_{n}^{(l)}|||^{\alpha}_{I_{n}^{i}}\Big).
\end{equation}
Now, we use an iterative process to achieve the result. For $i=1$, \eqref{eq 93} reads
$$
	|||r_{n}^{(l)}|||_{I_{n}^{1}} + \sup_{t \in I_{n}^{1}} \|\nabla r_{n}^{(l)}(t)\|_{L^{2}} 
	\leq  C \Big( \gamma_{n}^{(l)}(0) + \delta_{n}^{(l)} + \sum_{\alpha =2}^{5}|||r_{n}^{(l)}|||^{\alpha}_{I_{n}^{1}}\Big).
$$
Recall that, in view of the definition of $\gamma_{n}^{(l)}$, we have
\begin{equation} \label{eq 95}
	\gamma_{n}^{(l)}(0) = \|\nabla r_{n}^{(l)}(0)\|_{L^{2}} = \Big\|\nabla \Big(\sum_{j=1}^{l}(p_{n}^{(j)} - q_{n}^{(j)})(0)\Big)\Big\|_{L^{2}} =0 
\end{equation}
for all $l \geq 1$. Due to \eqref{eq 83} and \eqref{eq 95}, it follows that, for all large enough $l$, there exists $N(l)$, such that if $n \geq N(l)$, then
$
	\gamma_{n}^{(l)}(0) + \delta_{n}^{(l)} \leq a_{0}(c).
$
 Denote by $M_{n}^{l}$ the function defined on $I_{n}^{1}=[0, a_{n}^{1}]$
by
$$M_{n}^{l}(s) = |||r_{n}^{(l)}|||_{[0,s]} + \frac{s}{a_{n}^{1}} \sup_{t \in [0,s]} \|\nabla r_{n}^{(l)}(t)\|_{L^{2}}.$$
It is clear that \eqref{eq 93} still holds if we replace $I_{n}^{1}=[0, a_{n}^{1}]$ by $[0, s]$ for every
$s\in I_{n}^{1}$ .Thus,
$$M_{n}^{l}(s) \leq C \Big( \gamma_{n}^{(l)}(0) + \delta_{n}^{(l)} + \sum_{\alpha =2}^{5} (M_{n}^{l})^{\alpha}(s)\Big).$$
Hence, the function $M_{n}^{l} $ satisfies the conditions of Lemma \ref{lemma4.7} for large $l$ and $n \geq N(l)$. So
\begin{equation} \label{eq 97}
	M_{n}^{l}(a_{n}^{1}) = |||r_{n}^{(l)}|||_{I_{n}^{1}} + \sup_{t \in I_{n}^{1}} \|\nabla r_{n}^{(l)}(t)\|_{L^{2}} \leq 2c(\gamma_{n}^{(l)}(0) + \delta_{n}^{(l)}),
\end{equation}
for large $l$ and $n \geq N(l)$. 
Using \eqref{eq 83}, \eqref{eq 95} and \eqref{eq 97}, one obtains
$$\limsup_{n \rightarrow \infty}\Big( |||r_{n}^{(l)}|||_{I_{n}^{1}} + \sup_{t \in I_{n}^{1}}\|\nabla r_{n}^{(l)}(t)\|_{L^{2}}\Big)\longrightarrow 0 \mbox{ as } l\rightarrow \infty.$$
On the other hand, we have
$$\gamma_{n}^{(l)}(a_{n}^{1}) \leq \sup_{t \in I_{n}^{1}} \|\nabla r_{n}^{(l)}(t)\|_{L^{2}},$$
which gives
$$\limsup_{n \rightarrow \infty} \gamma_{n}^{(l)}(a_{n}^{1}) \longrightarrow 0 \mbox{ as } l\rightarrow \infty.$$
This allows us to repeat the same argument on the interval $I_{n}^{2}=[a_{n}^{1} , a_{n}^{2}]$. We get
$$	 |||r_{n}^{(l)}|||_{I_{n}^{2}} + \sup_{t \in I_{n}^{2}} \|\nabla r_{n}^{(l)}(t)\|_{L^{2}} \leq c(\gamma_{n}^{(l)}(a_{n}^{1}) + \delta_{n}^{(l)}).$$
Thus
$$\limsup_{n \rightarrow \infty}\Big( |||r_{n}^{(l)}|||_{I_{n}^{2}} + \sup_{t \in I_{n}^{2}}\|\nabla r_{n}^{(l)}(t)\|_{L^{2}}\Big)\longrightarrow 0 \mbox{ as } l\rightarrow \infty.$$
Iterating this process, we get
$$\limsup_{n \rightarrow \infty}\Big( |||r_{n}^{(l)}|||_{I_{n}^{i}} + \sup_{t \in I_{n}^{i}}\|\nabla r_{n}^{(l)}(t)\|_{L^{2}}\Big)\longrightarrow 0 \mbox{ as } l\rightarrow \infty$$
for all $1 \leq i \leq p$.
Since $p$ does not depend on $n$ and $l$, we get
$$\limsup_{n \rightarrow \infty}\Big( |||r_{n}^{(l)}|||_{[0,T]} + \sup_{t \in [0,T]}\|\nabla r_{n}^{(l)}(t)\|_{L^{2}}\Big)\longrightarrow 0 \mbox{ as }  l\rightarrow \infty,$$
which concludes the proof.
\end{proof}
	\subsection{Profile decomposition of the limit energy}
For $u$ a solutions of the nonlinear Schr\"{o}dinger equation \eqref{eq 100}, we denote its nonlinear energy density by
$$\mathcal{E}(t)(t,x) = \frac{1}{2} |\nabla u(t,x)|^{2} + \frac{1}{6}|u(t,x)|^{6}.$$
For a sequence $u_{n}$ of solution with initial data bounded in $\dot{H}^{1}(\mathbb{R}^{3})$, the corresponding nonlinear energy density is bounded in $L^{\infty}([0,T],L^{1})$ and so in the space of bounded measures on $[0 , T ] \times \mathbb{R}^{3}$. This allows us to consider, up to a subsequence, its weak* limit. 
The following theorem shows that the energy limit follows the same profile decomposition as $u_{n}$. This will be a crucial result that will allow the use of a microlocal defect measure on each profile and then apply the linearization argument.

\begin{theorem} \label{energydec} Let $u_{n}$ be a sequence of solutions to
	\begin{equation} \label{eq 100}
		i\partial_{t} u_{n} + \Delta u_{n} -|u_{n}|^{4}u_{n} = 0,
	\end{equation}
	with $u_{n}(0)$  convergent to 0 in ${L}^{2}(\mathbb{R}^{3})$.
	The nonlinear energy density limit of $u_{n}$ (up to a subsequence) is
	$$\mathcal{E}(t,x) = \sum_{j=1}^{\infty}  e^{(j)}(t,x) + e_{f}(t,x),$$
	where $ e^{(j)}$ is the energy density limit of $q_{n}^{(j)}$ (following the notation of Theorem \ref{nonlineardec}) and
	$$e_{f} = \lim_{l \rightarrow \infty} \lim_{n\rightarrow \infty} e(w_{n}^{(l)}),$$
	where the limits are considered up to a subsequence and in the weak* sense.
		In particular, $e_{f}$ can be written as
	$$e_{f}(t,x) = \int_{\xi \in S^{2}} \mu(t,x,d\xi).$$
	Moreover, $\mathcal{E}$ is also the limit of the linear energy density
	$$\mathcal{E}_{lim}(u_{n})(t,x) = \frac{1}{2} |\nabla u_{n}(t,x)|^{2}.$$
\end{theorem}
\begin{proof} The proof of this result is a direct consequence of Theorem \ref{nonlineardec}. Indeed, noting that $\|u_{n}\|_{L^{10}([0,T]\times \mathbb{R}^{3})} \leq C$,  it follows, by an interpolation argument, that $$\|u_{n}\|_{L^{2}([0,T]\times \mathbb{R}^{3})} \rightarrow 0 \implies \|u_{n}\|_{L^{6}([0,T]\times \mathbb{R}^{3})} \rightarrow 0,\quad \text{as} \quad n\rightarrow \infty.$$ 
Therefore, $\mathcal{E}$ is the limit of $b(u_{n},u_{n})$, with  $b(f,g) = \nabla f(t,x) \cdot \overline{\nabla g(t,x)}.$

	Now, we have to compute the limit of $b(u_{n},u_{n})$ using the decomposition of Theorem \ref{nonlineardec}. We set  $s_{n}^{(l)} = \sum_{j=1}^{l} q_{n}^{(j)}$, for $l \in \mathbb{N}$, and so
	$$b(u_{n},u_{n}) = b(s_{n}^{(l)},s_{n}^{(l)}) + b(w_{n}^{(l)},w_{n}^{(l)}) + 2b(s_{n}^{(l)},w_{n}^{(l)}) + 2b(u_{n},r_{n}^{(l)}) - b(r_{n}^{(l)},r_{n}^{(l)}).$$
	The convergence \eqref{eq 68} gives $$\limsup_{n \rightarrow \infty} \|2b(u_{n},r_{n}^{(l)}) -b(r_{n}^{(l)},r_{n}^{(l)})\|_{L^{1}([0,T]\times \mathbb{R}^{3})}\to0 \mbox{ as } l \rightarrow \infty.$$  So, defining $e_{r}^{(l)} = w* \lim_{n\rightarrow \infty}(2b(u_{n},r_{n}^{(l)}) -b(r_{n}^{(l)},r_{n}^{(l)}))$, we have 
	$$e_{r}^{(l)} \longrightarrow 0 \mbox{ as } l\rightarrow \infty.$$
	Let $\varphi(t,x) = \varphi_{1}(t)\cdot\varphi_{2} (x) \in C_{0}^{\infty}((0 ,T) \times \mathbb{R}^{3})$. It remains to estimate
	$$\int_{0}^{T} \int_{\mathbb{R}^{3}} \varphi b(s_{n}^{(l)},w_{n}^{(l)}) = \sum_{j=1}^{l} \int_{0}^{T}\varphi_{1} \int_{\mathbb{R}^{3}} \varphi_{2} b(q_{n}^{(j)},w_{n}^{(l)}) $$ for each fixed $l$. 
	To this end, first note that, since $b(q_{n}^{(j)},w_{n}^{(l)})$ is bounded in $L^{\infty}((0,T),L^{1}( \mathbb{R}^{3})),$ we can assume, up to an arbitrary small error, that $\varphi_{1}$ is supported in $\{t < t_{\infty}^{(j)}\}$ or $ \{t > t_{\infty}^{(j)}\}$ (replace $\varphi_{1}$ by $(1-\Psi)(t)\varphi_{1}$ with $\Psi(t_{\infty}^{(j)}) = 1$ and $\|\Psi\|_{L^{1}(0,T)}$ small). On each interval, Theorem \ref{teo4.1} allows to replace $q_{n}^{(j)}$ by a linear concentrating solution. Then, by Lemma \ref{lemma3.16},  we get the weak convergence of $b(s_{n}^{(l)},w_{n}^{(l)})$ to zero, for each fixed $l$. Indeed, by Lemma \ref{lemma3.16}, $D_{h_{n}}^{(j)} w_{n}^{(l)} \rightharpoonup 0$, $1 \leq j \leq l$, which means,
	$$\sqrt{h_{n}^{(j)}} w_{n}^{(l)}(t_{n}^{(j)} + (h_{n}^{(j)})^{2}s, x_{n}^{(j)} + h_{n}^{(j)}y) \rightharpoonup 0.$$ It is enough to compute $\int_{\mathbb{R}^{3}} \nabla_{x}w_{n}^{(l)}(t,x) \cdot \nabla_{x} p_{n}^{(j)}(t,x) \ dx.$ We have
	\begin{equation*}
	\begin{split}
	&	\int_{\mathbb{R}^{3}} \nabla_{x}w_{n}^{(l)}(t,x) \cdot \nabla_{x} p_{n}^{(j)}(t,x) \ dx =  \int_{\mathbb{R}^{3}}  \nabla_{x}w_{n}^{(l)}(t,x) \cdot \nabla_{x} \frac{1}{\sqrt{h_{n}^{(j)}}} \varphi^{(j)} \Big(\frac{t-t_{n}^{(j)}}{(h_{n}^{(j)})^{2}}, \frac{x-x_{n}^{(j)}}{h_{n}^{(j)}}\Big) \ dx \\
		& =   \int_{\mathbb{R}^{3}}  \nabla_{x}w_{n}^{(l)}(t_{n}^{(j)} + (h_{n}^{(j)})^{2}s, x_{n}^{(j)} + h_{n}^{(j)}y) \cdot \nabla_{x} \frac{1}{\sqrt{h_{n}^{(j)}}} \varphi^{(j)}(s,y) \ (h_{n}^{(j)})^{3} dy \\
		& =  \int_{\mathbb{R}^{3}}  \nabla_{y} \sqrt{h_{n}^{(j)}}w_{n}^{(l)}(t_{n}^{(j)} + (h_{n}^{(j)})^{2}s, x_{n}^{(j)} + h_{n}^{(j)}y) \cdot \nabla_{y}  \varphi^{(j)}(s,y) \ dy \longrightarrow 0,
		\end{split}
	\end{equation*}
	as $n \rightarrow \infty$.
	
	Lemma \ref{lemma3.14} and the orthogonality of the cores of concentration yields
	$D_{h_{n}}^{(j)}p_{n}^{(j')}\rightharpoonup 0$,
	for $ j \neq j'$ and $p_{n}^{(j')}$ a concentrating solution at rate $[\underline{h}^{(j')},\underline{x}^{(j')},\underline{t}^{(j')}]$. Then,
	the same argument as before gives
	$$b(s_{n}^{(l)},s_{n}^{(l)}) \rightharpoonup \sum_{j=1}^{l} e^{(j)}.$$
So, we have proved that for any $l \in \mathbb{N}$,
	$$b(u_{n},u_{n}) \rightharpoonup \mathcal{E} = \sum_{j=1}^{l} e^{(j)} + e_{w}^{(l)} + e_{r}^{(l)} \mbox{ as } n\rightarrow \infty,$$
	where $e_{w}^{(l)}$ is the weak* limit of $b(w_{n}^{(l)},w_{n}^{(l)})$ and $e_{r}^{(l)}$ satisfies $e_{r}^{(l)} \rightarrow 0$ as  $l\rightarrow \infty.$ Since $e_{w}^{(l)}$ is the weak* limit of a sequence of solutions to the linear Schr\"{o}dinger equation with initial data convergent to zero in $L^{2}$, we can use Proposition \ref{prop6.1} to conclude that $b(w_{n}^{(l)},w_{n}^{(l)})$ converges (locally) to a positive measure $e_{f}$. Hence,
	$$\mathcal{E}= \sum_{j=1}^{\infty} e^{(j)} + e_{f}$$
	and the result is proven. 
\end{proof}

\section{Exponential stabilizability: Proof of Theorem  \ref{mainteo}} \label{sec4}
It is well-known in control theory that the energy associated with the system \eqref{eq 4} is exponentially stable if the observability inequality 
\begin{equation}\label{obser}
	E(u)(0) \leq C \int_{0}^{T}\int_{\mathbb{R}^{3}}|(1-\Delta)^{-\frac{1}{2}}a\partial_{t}u|^{2} \ dx dt
\end{equation}
is verified.  Here, we consider $a \in C^{\infty}(\mathbb{R}^{3})$ satisfying \eqref{eq 3}. So, $\omega : = \big(\mathbb{R}^{3} \backslash B_{R}(0)\big)$ satisfies the following geometric control condition: 
\begin{assumption}\label{assum5.1}
	There exists $T_{0} > 0$ such that every geodesic travelling at speed 1 meets $\omega$ in a time $t < T_{0}$.
\end{assumption}

Roughly speaking, the proof of the stabilizability consists of the analysis of possible sequences contradicting the observability estimate.
\subsection{Linearization argument}   The first step of the proof is to show that such a sequence is linearizable because its behavior is close to solutions of the linear equation.

\begin{lemma} \label{lemma5.3} Let $T >T_{0}$ and $u_{n}$ be a sequence of solutions to
	\begin{equation} \label{eq 101}
		\left\{
		\begin{array}{ll}
			i\partial_{t} u_{n} + \Delta u_{n} - u_{n} - |u_{n}|^{4}u_{n}  - a(1-\Delta)^{-1}a\partial_{t} u_{n} =0, &\ \mbox{ on } [0,T] \times \mathbb{R}^{3}, \\
			u_{n}(0)= u_{0,n},& \ \mbox{ in } {H}^{1}(\mathbb{R}^{3}) 
		\end{array}
		\right.
	\end{equation}
	satisfying
	\begin{equation} \label{eq 102}
		u_{0,n} \rightarrow0 \mbox{ in } {L}^{2}(\mathbb{R}^{3})  \text{ as } n\rightarrow \infty,
	\end{equation}
	and
	\begin{equation} \label{eq 103}
		\int_{0}^{T} \int_{\mathbb{R}^{3}} |(1-\Delta)^{-\frac{1}{2}}a \partial_{t}u_{n}|^{2} \ dxdt \longrightarrow 0 \mbox{ as }n\rightarrow \infty.
	\end{equation}
	 Consider the profile decomposition according to Theorem \ref{nonlineardec} in a subinterval $[t_{0}, t_{0} + L] \subset [0, T ]$ with $T_{0} < L $. Then, for any $0 < \varepsilon < L - T_{0}$, this decomposition does not contain any nonlinear concentrating solution with $t_{\infty}^{(j)} \in [t_{0} ,t_{0} + \varepsilon]$ and $u_{n}$ is linearizable in $[t_{0} ,t_{0} + \varepsilon ],$ i.e.,
	 $$\|u_{n} - v_{n}\|_{L^{10}([t_{0},t_{0} + \varepsilon]\times \mathbb{R}^{3})} + \|u_{n} -v_{n}\|_{L^{\infty}([t_{0},t_{0} + \varepsilon]; {H}^{1}(\mathbb{R}^{3}))} \longrightarrow 0 \mbox{ as } n\rightarrow \infty,$$
	 where $v_{n}$ is the solution of
	 \begin{equation*}
	 	\left\{
	 	\begin{array}{ll}
	 		i\partial_{t} v_{n} + \Delta v_{n} - v_{n}=0, &\ \mbox{ on } [0,T] \times \mathbb{R}^{3}, \\
	 		v_{n}(0)= u_{0,n},& \ \mbox{ in }  {H}^{1}(\mathbb{R}^{3}).
	 	\end{array}
	 	\right.
	 \end{equation*}
\end{lemma}
\begin{proof}
	With no loss of generality, we will consider the interval $[0,L]$ instead of $[t_{0},t_{0} + L]$ to keep the notation simple.
	
	\vspace{0.2cm}
	
\noindent\textbf{Claim 1:} The sequence $u_{n}$ is convergent to 0 in $L^{2}([0,T] \times \mathbb{R}^{3})$. 
	
		\vspace{0.2cm}
Indeed, multiplying the first equation of \eqref{eq 101} by $\overline{u}_{n}$ and taking its imaginary part, we obtain the estimate 
	\begin{eqnarray*}
		\frac{1}{2}\|u_{n}(t)\|_{L^{2}} \leq \frac{1}{2} \|u_{n}(0)\|_{L^{2}} + \int_{0}^{t} \|a(1-\Delta)^{-1}a\partial_{t}u_{n}\|_{L^{2}} \|u_{n}\|_{L^{2}} \ ds,
	\end{eqnarray*}
	and Claim 1 follows due to the convergences \eqref{eq 102} and \eqref{eq 103}.
\vspace{0.2cm}

\noindent\textbf{Claim 2:} The sequence $u_{n}$ is convergent to 0 in $ L^{2}_{loc}((0,L); {H}^{1}_{loc} (\omega))$. 

\vspace{0.2cm}

From now on, we will use several times the operator $Jv = (1-ia(x)(1 -\Delta )^{-1} a(x))v$ as a pseudodifferential operator of order 0 (see Theorem \ref{JJJ} for details about the properties of this operator). 

Since, by hypothesis, 
	$$\|(1-\Delta)^{-\frac{1}{2}}a \partial_{t}u_{n}\|_{L^{2}(([0,L];\mathbb{R}^{3})} \rightarrow 0 \mbox{ as } n\rightarrow \infty,$$
	one has
	$$  \|(1-\Delta)^{-\frac{1}{2}}a (-iJ^{-1}(I-\Delta) u_{n}  -iJ^{-1}|u_{n}|^{4}u_{n})\|_{L^{2}([0,L];\mathbb{R}^{3})} \rightarrow 0, \ n\rightarrow \infty.$$
Observe that
	\begin{equation*}
	\begin{split}
		\|(1-\Delta)^{-\frac{1}{2}}a iJ^{-1}(I-\Delta) u_{n}  \|_{L^{2}((0,L);\mathbb{R}^{3})} 
		 \leq &  \|(1-\Delta)^{-\frac{1}{2}}a (iJ^{-1}(1-\Delta) u_{n}  +iJ^{-1}|u_{n}|^{4}u_{n})\|_{L^{2}([0,L];\mathbb{R}^{3})}  \\
		&   + \|(1-\Delta)^{-\frac{1}{2}}a iJ^{-1}|u_{n}|^{4}u_{n})\|_{L^{2}([0,L];\mathbb{R}^{3})} \\
		 \leq &  \|(1-\Delta)^{-\frac{1}{2}}a (iJ^{-1}(1-\Delta) u_{n}  + iJ^{-1}|u_{n}|^{4}u_{n})\|_{L^{2}([0,L];\mathbb{R}^{3})}  \\
		&    + \|u_{n}^{5}\|_{L^{2}([0,L];H^{-1}(\mathbb{R}^{3}))}  \rightarrow 0 \mbox{ as } n\rightarrow \infty,
		\end{split}
	\end{equation*}
due to the converge
	\begin{eqnarray*}
		\|u_{n}^{5}\|^{2}_{L^{2}((0,L);H^{-1}(\mathbb{R}^{3}))} 
		 \leq  \sup_{t \in [0,L]} \|u_{n}(t)\|^{\frac{5}{3}}_{L^{2}} \int_{0}^{L}\|u_{n}(t)\|^{\frac{25}{3}}_{L^{10}} \ dt
		 \leq  \sup_{t \in [0,L]} \|u_{n}(t)\|^{\frac{5}{3}}_{L^{2}} \|u_{n}\|^{\frac{3}{25}}_{L^{\frac{25}{3}}_{t}L^{10}_{x}} 
		 \rightarrow 0,
	\end{eqnarray*} 
by interpolation arguments. 
	Hence, taking into account the convergence in  $L^{2}([0,T] \times \mathbb{R}^{3})$ obtained in Claim 1, for every $\chi \in C_{0}^{\infty}((0,L)\times \mathbb{R}^{3}),$ we have
	$$ \|(1-\Delta)^{-\frac{1}{2}}a J^{-1}(1-\Delta) \chi u_{n}  \|_{L^{2}(0,L)\times \mathbb{R}^{3})} \rightarrow 0,$$
	which is equivalent to 
	$$\Big\langle (1-\Delta)(J^{-1})^{*}a(1-\Delta)^{-1}aJ^{-1}(1-\Delta) \chi u_{n}, \chi u_{n} \Big\rangle_{L^{2}((0,L)\times \mathbb{R}^{3})} \rightarrow 0. $$
	This means, using Proposition \ref{prop6.1} (Appendix), that
	$$\int_{(0,L)\times\mathbb{R}^{3}\times S^{3}} \dfrac{(1 + |\xi|^{2})a^{2}}{1 + |\xi|^{2}} (1 + |\xi|^{2})\ d \mu(t,x,\xi) = 0.$$
Thus	$$ \int_{(0,L)\times \omega \times S^{3}} 1+ |\xi|^{2} \ d \mu(t,x,\xi) = 0,$$
	i.e.,
	$$u_{n} \longrightarrow 0 \ \mbox{in} \ L^{2}_{loc}((0,L); {H}^{1}_{loc} (\omega)) \mbox{ as }n\rightarrow \infty,$$
	showing the claim. 

\vspace{0.2cm}
	
	Now, let $\widetilde{u}_{n}$ be a solution to 
	\begin{equation} \label{pertu}
		\left\{
		\begin{array}{lr}
			i\partial_{t} \widetilde{u}_{n} + \Delta \widetilde{u}_{n} - \widetilde{u}_{n} - |\widetilde{u}_{n}|^{4}\widetilde{u}_{n}  =0, \ \mbox{ on } [0,T] \times \mathbb{R}^{3},\\
			\widetilde{u}_{n}(0)= u_{0,n} \in {H}^{1}(\mathbb{R}^{3}) .
		\end{array}
		\right.
	\end{equation}
	By the convergence \eqref{eq 103} and Lemma \ref{lemma4.15}, we get 
	$$\widetilde{u}_{n} \longrightarrow 0 \ \mbox{in} \ L^{2}_{loc}((0,L); {H}^{1}_{loc} (\omega)) \mbox{ as } n\rightarrow \infty.$$
Let $w_{n} = e^{it}\widetilde{u}_{n}$. It satisfies 
	\begin{equation} \label{origin}
		\left\{
		\begin{array}{lr}
			i\partial_{t} w_{n} + \Delta w_{n}  - |w_{n}|^{4}w_{n}  =0, \ \mbox{ on } [0,T] \times \mathbb{R}^{3},\\
			w_{n}(0)= u_{0,n} 
		\end{array}
		\right.
	\end{equation}
	and 
	$$w_{n} \longrightarrow 0 \ \mbox{in} \ L^{2}_{loc}((0,L); \dot{H}^{1}_{loc} (\omega)) \mbox{ as }n\rightarrow \infty $$
and so 
	$$ |\nabla w_{n}(t)|^{2} \longrightarrow 0 \ \mbox{in } L^{1}L^{1},$$
Using the notation of Theorem \ref{energydec}, this gives $e = 0$ on $(0, L) \times \omega$ (locally). Since all the measures in the decomposition of $\mathcal{E}$ are positive, we get the same result for any nonlinear concentrating solution in the decomposition of $w_{n}$, that is, $e_{j} = 0$ in $(0,L) \times \omega$ (locally), and
	$$|\nabla q_{n}^{(j)}|^{2} \rightharpoonup 0 \mbox{ in } L^{1}_{loc}((0,L)\times \omega)$$
	which gives us
	$$\int_{0}^{L}\int_{\omega} \varphi |\nabla q_{n}^{(j)}|^{2} \longrightarrow 0,$$
for all $\varphi \in C_{0}^{\infty}$. Therefore, $$ q_{n}^{(j)} \longrightarrow 0  \ \mbox{in} \ L^{2}_{loc}((0,L);\dot{H}_{loc}^{1} (\omega)) \mbox{ as } n\rightarrow \infty$$
	and if $\mu^{(j)}$ is the microlocal defect measure of $q_{n}^{(j)}$, we have
	\begin{equation} \label{eq 104}
		\mu^{(j)} \equiv 0 \ \mbox{in } (0,L) \times \omega \times S^{3}.
	\end{equation}

It is important to point out that above, from the moment we return to the original system \eqref{origin}, we begin to consider the norm in the space $\dot{H}^{1}$, since the energy of this system is related to this space. When we recover the perturbed system \eqref{pertu} later on, we return to using the norms in the space $H^{1}$.

	Assume that $t_{\infty}^{(j)} \in [0,\varepsilon]$ for some $j\in \mathbb{N}$, so that the interval $(t_{\infty}^{(j)},L]$ has lenght greater than $T_{0}$. Denote by $p_{n}^{(j)}$ the linear concentrating solution approaching $q_{n}^{(j)}$ in the interval $I_{n}^{3,\Lambda} = (t_{n}^{(j)} + \Lambda (h_{n}^{(j)})^{2},L]$ according to the notation of Theorem \ref{teo4.1}, so that, for any $t_{\infty}^{(j)} <t <L$, we have
	$$\|q_{n}^{(j)} - p_{n}^{(j)}\|_{L^{10}([t,L]\times \mathbb{R}^{3})} + \|q_{n}^{(j)} -p_{n}^{(j)}\|_{L^{\infty}([t,L]; \dot{H}^{1}(\mathbb{R}^{3}))} \longrightarrow 0 \mbox{ as } n\rightarrow \infty.$$
In particular, $\mu^{(j)}$ is also attached to $p_{n}^{(j)}$ in the time interval $(t_{\infty}^{(j)},L]$. 

\vspace{0.2cm}

\noindent\textbf{Claim 3:}  $p_{n}^{(j)}$ is bounded in $\dot{H}^{1}(\mathbb{R}^{3})$ and $\|p_{n}^{(j)}(t)\|_{L^{2}} \rightarrow 0$ as $n\rightarrow \infty$.

\vspace{0.2cm}

Remember that $p_{n}^{(j)}$ is a solution of the linear Schr\"{o}dinger equation. If $p_{n}^{(j)}$ is a linear concentrating solution, we can consider $$p_{n}^{(j)}(t,x) = \frac{1}{\sqrt{h_{n}^{(j)}}} \varphi^{(j)}\Big(\frac{t-t_{n}^{(j)}}{(h_{n}^{(j)})^{2}},\frac{x-x_{n}^{(j)}}{h_{n}^{(j)}}\Big),$$ and so 
	\begin{equation*}
	\begin{split}
		\|p_{n}^{(j)}(t)\|_{L^{2}} =  \frac{1}{\sqrt{h_{n}^{(j)}}} \Big(\int_{\mathbb{R}^{3}} |\varphi^{(j)}(s,y)|^{2} (h_{n}^{(j)})^{3} \ dy\Big)^{\frac{1}{2}} = h_{n}^{(j)}\|\varphi^{(j)}(s)\|_{L^{2}} \leq C h_{n}^{(j)}\|\varphi^{(j)}(s)\|_{L^{6}} \rightarrow 0,
		\end{split}
	\end{equation*}
	as $n \rightarrow \infty$, since we can assume $\varphi^{(j)}(s) \in C_{0}^{\infty}(\mathbb{R}^{3})$. Thus, $p_{n}^{(j)}$'s measure propagates along the geodesics of the $\mathbb{R}^{3}$ and we have
	$\mu^{(j)} \equiv 0 \mbox{  in  } (t^{(j)}_{\infty},L) \times \mathbb{R}^{3} \times S^{3},$ since $|L - t_{\infty}^{(j)}| > T_{0}$ ensure that the geometric control condition is still verified in the interval $[t_{\infty}^{(j)}, L]$  when combined with \eqref{eq 104}. This means that 
	$$p_{n}^{(j)} \rightarrow 0 \mbox{  in  } L^{2}_{loc}((t_{\infty}^{(j)},L); H_{loc}^{1}(\mathbb{R}^{3})),$$
	showing Claim 3.
	
	Finally, solving the equation satisfied by $p_{n}^{(j)}$ with initial data $p_{n}^{(j)}(t_{0})$, where $t_{0} \in (t_{\infty}^{(j)}, L)$ is such that $\|p_{n}^{(j)}(t_{0})\|_{H^{1}}\rightarrow 0$, one has the strong convergence $p_{n}^{(j)}\rightarrow0$ in the space $L^{\infty}([t^{(j)}_{\infty}, L], H_{loc}^{1}(\mathbb{R}^{3}))$. 
	In particular, $p_{n}^{(j)}(t_{\infty}^{(j)}) \rightarrow 0 \mbox{  in  } \dot{H}^{1}_{loc}(\mathbb{R}^{3})$, so the measure $\mu^{(j,\infty)}$ associated to $p_{n}^{(j)}(t_{\infty}^{(j)})$ satisfies $\mu^{(j,\infty)} \equiv 0 \mbox{  in  } \mathbb{R}^{3} \times S^{2}.$ On other hand, since $p_{n}^{(j)}(t^{(j)}_{\infty}) = \frac{1}{\sqrt{h_{n}}} \varphi^{(j)} \Big(\frac{x-x_{\infty}^{(j)}}{h_{n}}\Big)$, we can compute $\mu^{(j,\infty)}$ directly. To this end, note that
	\begin{equation*}
	\begin{split}
		\langle A(x,D_{x}) &\nabla p_{n}^{(j)}(t^{(j)}_{\infty}), \nabla p_{n}^{(j)}(t^{(j)}_{\infty}) \rangle_{L^{2}}  \\=&  \frac{1}{(2\pi)^{3}} \int_{\mathbb{R}^{3}}\int_{\mathbb{R}^{3}}\int_{\mathbb{R}^{3}} a(x, \xi) e^{i(x-y)\xi} |\xi|^{2} p_{n}^{(j)}(t^{(j)}_{\infty})(y) \overline{p_{n}^{(j)}(t^{(j)}_{\infty})(x)} \ dydxd\xi \\
		 = & \frac{1}{(2\pi)^{3}}\frac{1}{h_{n}} \int_{\mathbb{R}^{3}}\int_{\mathbb{R}^{3}}\int_{\mathbb{R}^{3}} a(x, \xi) e^{i(x-y)\xi} |\xi|^{2} \varphi^{(j)} \Big(\frac{y-x_{\infty}^{(j)}}{h_{n}}\Big) \overline{\varphi^{(j)} \Big(\frac{y-x_{\infty}^{(j)}}{h_{n}}\Big)} \ dydxd\xi \\
		 = & \frac{h_{n}^{5}}{(2\pi)^{3}} \int_{\mathbb{R}^{3}}\int_{\mathbb{R}^{3}}\int_{\mathbb{R}^{3}} a(h_{n}\tilde{x} + x_{\infty}^{(j)}, \xi) e^{ih_{n}(\tilde{x}-\tilde{y})\xi} |\xi|^{2} \varphi^{(j)} (\tilde{y}) \overline{\varphi^{(j)} (\tilde{x})} \ d\tilde{y}d\tilde{x}d\xi \\ 
		 =  &\frac{h_{n}^{2}}{(2\pi)^{3}} \int_{\mathbb{R}^{3}}\int_{\mathbb{R}^{3}}\int_{\mathbb{R}^{3}} a(h_{n}\tilde{x} + x_{\infty}^{(j)}, \frac{\tilde{\xi}}{h_{n}})  e^{i(\tilde{x}-\tilde{y})\tilde{\xi}} \Big|\frac{\tilde{\xi}}{h_{n}} \Big|^{2} \varphi^{(j)} (\tilde{y}) \overline{\varphi^{(j)} (\tilde{x})} \ d\tilde{y}d\tilde{x}d\tilde{\xi}\\ 
		 = & \frac{1}{(2\pi)^{3}} \int_{\mathbb{R}^{3}}\int_{\mathbb{R}^{3}}\int_{\mathbb{R}^{3}} a(h_{n}\tilde{x} + x_{\infty}^{(j)}, \tilde{\xi})  e^{i(\tilde{x}-\tilde{y})\tilde{\xi}}|\tilde{\xi}|^{2} \varphi^{(j)} (\tilde{y}) \overline{\varphi^{(j)} (\tilde{x})} \ d\tilde{y}d\tilde{x}d\tilde{\xi}\\ 
		 =&  \frac{1}{(2\pi)^{3}} \int_{\mathbb{R}^{3}} a(h_{n}\tilde{x} + x_{\infty}^{(j)}, \tilde{\xi}) |\tilde{\xi}|^{2} |\widehat{\varphi^{(j)} }(\tilde{\xi})|^{2} \ d\tilde{\xi}\rightarrow  \frac{1}{(2\pi)^{3}} \int_{\mathbb{R}^{3}} a( x_{\infty}^{(j)}, \tilde{\xi}) |\tilde{\xi}|^{2} |\widehat{\varphi^{(j)} }(\tilde{\xi})|^{2} \ d\tilde{\xi}.
		\end{split}
	\end{equation*}
Using polar coordinates, we get
	$$\mu^{(j,\infty)} = \delta_{x-x_{\infty}^{(j)}} \otimes \Phi(\theta) \ d\theta,$$
	where $\Phi(\theta) = \frac{1}{(2\pi)^{3} }\int_{-\infty}^{\infty} |r\theta|^{2} |\widehat{\varphi^{(j)}}(r\theta)|^{2} r^{2} \ dr.$
	Therefore, $p_{n}^{(j)}(t_{\infty}^{(j)}) \equiv 0$, and the conservation of the energy yields
	$$\|p_{n}^{(j)}(t)\|_{\dot{H}^{1}(\mathbb{R}^{3})} = \|p_{n}^{(j)}(t_{\infty}^{(j)})\|_{\dot{H}^{1}(\mathbb{R}^{3})} = 0, \forall t \in (t_{\infty}^{(j)}, L].$$		
	Moreover, 
	$$\|q_{n}^{(j)}(t)\|_{\dot{H}^{1}(\mathbb{R}^{3})} \rightarrow 0, \forall t \in (t_{\infty}^{(j)}, L].$$
	Arguing in the same way as before, one obtains $q_{n}^{(j)} \equiv 0 \mbox{ in } (t_{\infty}^{(j)},L]$ as expected, since $q_{n}^{(j)}(t_{\infty}^{(j)}) = \frac{1}{\sqrt{h_{n}}} \psi^{(j)}\Big(\frac{x-x_{\infty}^{(j)}}{h_{n}}\Big).$
	Then, for the profile decomposition of $w_{n}$ in the interval $[0,L]$, namely, 
	$$w_{n} = \sum_{j=1}^{l} q_{n}^{(j)} + w_{n}^{(l)} + r_{n}^{(l)},$$
	we have proved that $t_{n}^{(j)} \in (\varepsilon, L],$ since $t_{n}^{(j)} \in [0,\varepsilon]$ implies $q_{n}^{(j)} \equiv 0$. Thus, Theorem \ref{teo4.1} provides a linear concentrating solution $p_{n}^{(j)}$ such that 
	$$\limsup_{n\rightarrow \infty} \Big(\|q_{n}^{(j)} - p_{n}^{(j)}\|_{L^{10}([0,\varepsilon]\times \mathbb{R}^{3})} + \|q_{n}^{(j)} -p_{n}^{(j)}\|_{L^{\infty}([0,\varepsilon]; \dot{H}^{1}(\mathbb{R}^{3}))} \Big) = 0,$$ while Lemma \ref{lemma3.10} gives
	$$\limsup_{n \rightarrow \infty} \|p_{n}^{(j)}\|_{L^{10}([0,\varepsilon]\times \mathbb{R}^{3})} =0.$$
	Moreover, Theorems \ref{lineardecomp} and \ref{nonlineardec} ensure	$$\limsup_{n \rightarrow \infty}\|w_{n}^{(l)} + r_{n}^{(l)}\|_{L^{10}([0,\varepsilon]\times \mathbb{R}^{3})} \longrightarrow 0 \mbox{ as } l\rightarrow \infty.$$
Therefore,
	$$\limsup_{n \rightarrow \infty}\|w_{n}\|_{L^{10}([0,\varepsilon]\times \mathbb{R}^{3})} = 0$$
	and, hence,
	$$\limsup_{n \rightarrow \infty}\|\widetilde{u}_{n}\|_{L^{10}([0,\varepsilon]\times \mathbb{R}^{3})} = 0.$$
	Thus,
	$$\|\nabla |\widetilde{u}_{n}|^{4}\widetilde{u}_{n}\|_{L^{2}([0,\varepsilon];L^{\frac{6}{5}} (\mathbb{R}^{3}))} \longrightarrow 0, \ n\rightarrow \infty.$$
Since
	\begin{eqnarray*}
			\|\nabla |\widetilde{u}_{n}|^{4}\widetilde{u}_{n}\|_{L^{2}([0,\varepsilon];L^{\frac{6}{5}} (\mathbb{R}^{3}))} & \leq & \|\widetilde{u}_{n}\|^{4}_{L^{10}([0,\varepsilon]\times \mathbb{R}^{3})}\|\nabla \widetilde{u}_{n}\|_{L^{10}([0,\varepsilon];L^{\frac{30}{13}}( \mathbb{R}^{3}))},
	\end{eqnarray*}
we have that $\widetilde{u}_{n}$ is linearizable on $[0, \varepsilon]$. Indeed, using Remark \ref{obs8}, note that
	\begin{equation} \label{eq 106}
	\begin{split}
		\|\widetilde{u}_{n}-v_{n}\|_{L^{10}([0,\varepsilon]\times \mathbb{R}^{3})} + \|\widetilde{u}_{n} -v_{n}\|_{L^{\infty}([0,\varepsilon]; {H}^{1}(\mathbb{R}^{3}))}  \leq & \|\nabla |\widetilde{u}_{n}|^{4}\widetilde{u}_{n}\|_{L^{\frac{10}{7}}([0,\varepsilon];L^{\frac{10}{7}} (\mathbb{R}^{3})} \\&+ \|\widetilde{u}_{n}^{5}\|_{L^{1}([0,\varepsilon];L^{2}(\mathbb{R}^{3}))} \\
		 \leq & C \|\widetilde{u}_{n}\|^{4}_{L^{10}([0,\varepsilon]\times \mathbb{R}^{3})}\|\nabla \widetilde{u}_{n}\|_{L^{\frac{10}{3}}([0,\varepsilon]\times \mathbb{R}^{3})} \\&+ C\|\widetilde{u}_{n}\|^{5}_{L^{10}([0,\varepsilon]\times \mathbb{R}^{3})} \\\rightarrow& 0,
		 \end{split}
	\end{equation}
	as $n \rightarrow \infty.$
	It follows that
	\begin{equation*}
	\begin{split}
		\|u_{n}-v_{n}\|_{L^{10}([0,\varepsilon]\times \mathbb{R}^{3})} + \|u_{n} -v_{n}\|_{L^{\infty}([0,\varepsilon]; {H}^{1}(\mathbb{R}^{3}))}  &\leq 	\|u_{n} - \widetilde{u}_{n}\|_{L^{10}([0,\varepsilon]\times \mathbb{R}^{3})} + \|u_{n} -\widetilde{u}_{n}\|_{L^{\infty}([0,\varepsilon]; {H}^{1}(\mathbb{R}^{3}))} \\
		& + 	\|\widetilde{u}_{n}-v_{n}\|_{L^{10}([0,\varepsilon]\times \mathbb{R}^{3})} + \|\widetilde{u}_{n} -v_{n}\|_{L^{\infty}([0,\varepsilon]; {H}^{1}(\mathbb{R}^{3}))} \\
		&\rightarrow  0,
		 \end{split}
	\end{equation*}
	as $n \rightarrow \infty$, due to \eqref{eq 103}, \eqref{eq 106} and Lemma \ref{lemma4.15}.
\end{proof}

With this in mind, the next proposition gives that a sequence of solutions for the nonlinear system is close to the solutions for the linear system.

\begin{proposition} \label{prop5.2}
	Under the assumptions of Lemma \ref{lemma5.3}, we have that $u_{n}$ is linearizable on $[0,t]$ for any $t<T - T_{0}$, that is,
	$$\|u_{n} - v_{n}\|_{L^{10}([0,t]\times \mathbb{R}^{3})} + \|u_{n} -v_{n}\|_{L^{\infty}([0,t]; {H}^{1}(\mathbb{R}^{3}))} \longrightarrow 0 \mbox{ as }n\rightarrow \infty,$$
	where $v_{n}$ is the solution of
	\begin{equation*} 
		\left\{
		\begin{array}{ll}
			i\partial_{t} v_{n} + \Delta v_{n} - v_{n}=0, &\ \mbox{ on } [0,T] \times \mathbb{R}^{3}, \\
			v_{n}(0)= u_{0,n},& \ \mbox{ in }  {H}^{1}(\mathbb{R}^{3}).
		\end{array}
		\right.
	\end{equation*}
\end{proposition}
\begin{proof}
	Let $$t_{*} = \sup\{s \in [0,T]; \lim_{n} \|u_{n} - v_{n}\|_{L^{10}([0,s]\times \mathbb{R}^{3})} + \|u_{n} -v_{n}\|_{L^{\infty}([0,s]; {H}^{1}(\mathbb{R}^{3}))} = 0\}.$$ We claim that $t_{*} \geq T-T_{0}.$
Indeed, suppose, by contradiction, that this does not hold, so we can find an interval $[t_{*} - \varepsilon, t_{*} - \varepsilon + L] \subset [0,T]$ with $T_{0} < L $ and $ 0< 2 \varepsilon < L-T_{0}$ (if $t_{*}=0$, take the interval $[0,L]\subset [0,T]$). It follows from Lemma \ref{lemma5.3} that $u_{n}$ is linearizable on $[t_{*} - \varepsilon, t_{*} + \varepsilon]$. The definition of $t_{*}$ gives $\lim_{n\rightarrow \infty}  \|u_{n} - v_{n}\|_{L^{10}([0,t_{*}- \varepsilon]\times \mathbb{R}^{3})} + \|u_{n} -v_{n}\|_{L^{\infty}([0,t_{*} - \varepsilon]; {H}^{1}(\mathbb{R}^{3}))} = 0$. So, we have proved that $\lim_{n\rightarrow\infty}  \|u_{n} - \tilde{v}_{n} \|_{L^{10}([t_{*}- \varepsilon,t_{*}+ \varepsilon]\times \mathbb{R}^{3})} + \|u_{n} -\tilde{v}_{n} \|_{L^{\infty}([t_{*}- \varepsilon,t_{*} + \varepsilon]; {H}^{1}(\mathbb{R}^{3}))} = 0$, where $\tilde{v}_{n}$ is a solution of
	$$i\partial_{t} \tilde{v}_{n} + \Delta \tilde{v}_{n}  - \tilde{v}_{n} = 0, \ \ \tilde{v}_{n}(t_{*} - \varepsilon) = u_{n}(t_{*} - \varepsilon).$$
	This yields $\lim_{n\rightarrow\infty}  \|u_{n} - v_{n}\|_{L^{10}([0,t_{*}+ \varepsilon]\times \mathbb{R}^{3})} + \|u_{n} -v_{n}\|_{L^{\infty}([0,t_{*} + \varepsilon]; {H}^{1}(\mathbb{R}^{3}))} = 0$. Indeed, we have 
	\begin{equation*}
	\begin{split}
		\sup_{t \in [0,t_{*} + \varepsilon]}\|u_{n}(t)-v_{n}(t)\|_{{H}^{1}(\mathbb{R}^{3})}  \leq & 	\sup_{t \in [0,t_{*} - \varepsilon]}\|u_{n}(t)-v_{n}(t)\|_{{H}^{1}(\mathbb{R}^{3})} \\&+ \sup_{t \in [t_{*} - \varepsilon,t_{*} + \varepsilon]}\|u_{n}(t)-v_{n}(t)\|_{{H}^{1}(\mathbb{R}^{3})},
		\end{split}
	\end{equation*}
	where the first term of the right-hand side converges to 0 as $n$ tends to $\infty$. For the second term, we have 
	\begin{equation*}
	\begin{split}
		\sup_{t \in [t_{*} - \varepsilon,t_{*} + \varepsilon]}\|u_{n}(t)-v_{n}(t)\|_{{H}^{1}(\mathbb{R}^{3})}  \leq & 		 \sup_{t \in [t_{*} - \varepsilon,t_{*} + \varepsilon]}\|u_{n}(t)-\widetilde{v}_{n}(t)\|_{{H}^{1}(\mathbb{R}^{3})} \\&+ \sup_{t \in [t_{*} - \varepsilon,t_{*} + \varepsilon]}\|\widetilde{v}_{n}(t)-v_{n}(t)\|_{{H}^{1}(\mathbb{R}^{3})} \\
	  \leq & \sup_{t \in [t_{*} - \varepsilon,t_{*} + \varepsilon]}\|u_{n}(t)-\widetilde{v}_{n}(t)\|_{{H}^{1}(\mathbb{R}^{3})} \\&+ \|u_{n}(t_{*} - \varepsilon)-v_{n}(t_{*} - \varepsilon)\|_{{H}^{1}(\mathbb{R}^{3})} \\
		 \leq & \sup_{t \in [t_{*} - \varepsilon,t_{*} + \varepsilon]}\|u_{n}(t)-\widetilde{v}_{n}(t)\|_{{H}^{1}(\mathbb{R}^{3})} \\&+ 	\sup_{t \in [0,t_{*} - \varepsilon]}\|u_{n}(t)-v_{n}(t)\|_{{H}^{1}(\mathbb{R}^{3})}
		\rightarrow 0,
		\end{split}
	\end{equation*}
	as $n \rightarrow \infty$. Now, we estimate the $L^{10}$ norm as 
$$
		\|u_{n}-v_{n}\|^{10}_{L^{10}([0,t_{*}+\varepsilon] \times \mathbb{R}^{3})}  \leq  \int_{0}^{t_{*}-\varepsilon} 	\|u_{n}-v_{n}\|^{10}_{L^{10}(\mathbb{R}^{3})} \ dt + \int_{t_{*}-\varepsilon}^{t_{*}+\varepsilon} 	\|u_{n}-v_{n}\|^{10}_{L^{10}(\mathbb{R}^{3})} \ dt,	
$$
	where the first term of the right-hand side converges to 0 as $n$ tends to $+ \infty$. For the second term, we have 
	\begin{eqnarray*}
		\|u_{n}-v_{n}\|_{L^{10}([t_{*}-\varepsilon,t_{*}+\varepsilon] \times \mathbb{R}^{3})} & \leq & 	\|u_{n}-\widetilde{v}_{n}\|_{L^{10}([t_{*}-\varepsilon,t_{*}+\varepsilon] \times \mathbb{R}^{3})} + 	\|\widetilde{v}_{n}-v_{n}\|_{L^{10}([t_{*}-\varepsilon,t_{*}+\varepsilon] \times \mathbb{R}^{3})} \\
		& \leq & 	\|u_{n}-\widetilde{v}_{n}\|_{L^{10}([t_{*}-\varepsilon,t_{*}+\varepsilon] \times \mathbb{R}^{3})} + 	\|u_{n}(t_{*}-\varepsilon)-v_{n}(t_{*}-\varepsilon)\|_{H^{1}(\mathbb{R}^{3})} \\
		& \leq & 	\|u_{n}-\widetilde{v}_{n}\|_{L^{10}([t_{*}-\varepsilon,t_{*}+\varepsilon] \times \mathbb{R}^{3})} + 	\sup_{t \in [0,t_{*} - \varepsilon]}\|u_{n}(t)-v_{n}(t)\|_{H^{1}(\mathbb{R}^{3})} \\
		& \rightarrow & 0,
	\end{eqnarray*}
	as $n \rightarrow \infty$, using Strichartz estimates, which contradicts the definition of $t_{*}$.
\end{proof}

\subsection{Weak observability estimate} The desired observability estimate \eqref{obser} is a consequence of the following weak observability estimates. 
\begin{theorem}\label{weakobs}
	Let $T > T_{0}$ and $\lambda_{0}>0$ from Definition \ref{def3}. There exists $C>0$ such that any solution 
	$u$ of the system
	\begin{equation} \label{eq 107}
		\left\{
		\begin{array}{lr}
			i\partial_{t} u + \Delta u -u - |u|^{4}u - a(1-\Delta)^{-1}a\partial_{t} u=0, \ \mbox{ on } [0,T] \times \mathbb{R}^{3}, \\
			u(0)= u_{0} \in {H}^{1}(\mathbb{R}^{3}), \\
			\|u_{0}\|_{H^{1}} \leq \lambda_{0},
		\end{array}
		\right.
	\end{equation}
	satisfies 
	\begin{equation}\label{eq 108}
		E(u)(0) \leq C \Big(\int_{0}^{T}\int_{\mathbb{R}^{3}} |(1-\Delta)^{-\frac{1}{2}}a\partial_{t}u|^{2} \ dxdt + \|u_{0}\|_{H^{-1}(\mathbb{R}^{3})}E(u)(0)\Big).
	\end{equation}
\end{theorem}
\begin{proof}
Remember that 
	$$E(u)(t) = \frac{1}{2}\|u(t)\|^{2}_{L^{2}} + \frac{1}{2}\|\nabla u(t)\|^{2}_{L^{2}} + \frac{1}{6}\|u(t)\|^{6}_{L^{6}}.$$
We argue by contradiction. Suppose that \eqref{eq 108} does not holds, so there exists a sequence $\{u_{n}\}$ of solutions to system \eqref{eq 107} such that 
	\begin{equation}\label{eq 109}
		\Big(\int_{0}^{T}\int_{\mathbb{R}^{3}} |(1-\Delta)^{-\frac{1}{2}}a\partial_{t}u_{n}|^{2} \ dxdt + \|u_{0,n}\|_{H^{-1}(\mathbb{R}^{3})}E(u_{n})(0)\Big) \leq \frac{1}{n}E(u_{n})(0).
	\end{equation}
Let $ \alpha_{n} = \big(E(u_{n})(0)\big)^{\frac{1}{2}}$. Sobolev's embedding for the $L^{6}$ norm ensures $\alpha_{n} \leq C(\lambda_{0})$. So, up to a subsequence, we may assume that $\alpha_{n} \rightarrow \alpha \geq 0$. We divide the analysis into the cases $\alpha > 0$ and $\alpha = 0$.  
	
	\
	
	\noindent$\bullet$ Case 1:  $\alpha_{n} \rightarrow \alpha>0$.
	
	\
	
	Note that $ \|u_{0,n}\|_{H^{-1}(\mathbb{R}^{3})} \rightarrow 0$. Using the inequality
	$$ \|u_{0,n}\|_{L^{2}(\mathbb{R}^{3})} \leq \|u_{0,n}\|^{\frac{1}{2}}_{H^{-1}(\mathbb{R}^{3})}\|u_{0,n}\|^{\frac{1}{2}}_{H^{1}(\mathbb{R}^{3})},$$ one obtains that $ \|u_{0,n}\|_{L^{2}(\mathbb{R}^{3})} \rightarrow 0$. Therefore, we apply Proposition \ref{prop5.2} and conclude that $\{u_{n}\}$ is linearizable in an interval $[0,L]$ with $L>T_{0}$, i.e., 
	$$\|u_{n}-v_{n}\|_{L^{10}([0,L]\times \mathbb{R}^{3})} + \|u_{n}-v_{n}\|_{L^{\infty}([0,L];H^{1}(\mathbb{R}^{3}))} \rightarrow 0 \mbox{ as } n \rightarrow \infty,$$
	where $v_{n}$ is a solution of
	\begin{equation*} 
		\left\{
		\begin{array}{lr}
			i\partial_{t} v_{n} + \Delta v_{n} -v_{n} =0, \ \mbox{ on } [0,T] \times \mathbb{R}^{3}, \\
			v_{n}(0)= u_{0,n}.
		\end{array}
		\right.
	\end{equation*}
Since $u_{0,n}\rightarrow0$ in $L^{2}(\mathbb{R}^{3})$, we get $\|u_{n}(t)\|_{L^{2}} \rightarrow 0$, $\forall t\in [0,T]$.
	Then, $\|v_{n}(t)\|_{L^{2}} \rightarrow 0$, $\forall t\in [0,L].$  Note that 
	$$\partial_{t}u_{n} =-iJ^{-1}\Big((1-\Delta)u_{n} + |u_{n}|^{4}u_{n}\Big) = -iJ^{-1}(I-\Delta)u_{n} - iJ^{-1}(|u_{n}|^{4}u_{n}),$$ where $J$ is given as in the proof of Theorem \ref{JJJ}, and
$$\|(1-\Delta)^{-\frac{1}{2}}a \partial_{t}u_{n}\|_{L^{2}((0,L);\mathbb{R}^{3})} \rightarrow 0\mbox{ as } n\rightarrow \infty.$$
This implies 
$$ \|(1-\Delta)^{-\frac{1}{2}}a (-iJ^{-1}(I-\Delta)u_{n} - iJ^{-1}(|u_{n}|^{4}u_{n}))\|_{L^{2}((0,L);\mathbb{R}^{3})} \rightarrow 0 \mbox{ as }  n\rightarrow \infty.$$ 
So, similarly to Claim 2 in the proof of Lemma \ref{lemma5.3}, we get
	$$u_{n} \rightarrow 0 \ in \ L^{2}_{loc}\big((0,L);H^{1}_{loc}(\omega)\big).$$
Additionally, we get
$$v_{n} \rightarrow 0 \ \mbox{  in  } L^{2}((0,L);H^{1}(\mathbb{R}^{3}\backslash B_{R+1}(0))).$$
Indeed, note that
	\begin{eqnarray*}
		\|v_{n}\|_{L^{2}([0,T];H^{1}(\mathbb{R}^{3}\backslash B_{R+1}(0)))} & \leq & \|v_{n} -u_{n}\|_{L^{2}([0,T];H^{1}(\mathbb{R}^{3}\backslash B_{R+1}(0)))} + \|u_{n}\|_{L^{2}([0,T];H^{1}(\mathbb{R}^{3}\backslash B_{R+1}(0)))}.
	\end{eqnarray*}
Now, we have
	\begin{eqnarray*}
		\|u_{n}\|_{L^{2}([0,T];H^{1}(\mathbb{R}^{3}\backslash B_{R+1}(0)))}   
		& \leq &  \Big\|[a,(1-\Delta)^{-1}J]J^{-1}(1-\Delta)u_{n} \Big\|_{L^{2}([0,T];{H}^{1}(\mathbb{R}^{3}))} \\
		&    & \mbox{    } +\Big\| (1-\Delta)^{-1}JaJ^{-1}(1-\Delta)u_{n}\Big\|_{L^{2}([0,T];{H}^{1}(\mathbb{R}^{3}))}\\
		& \leq & C \|u_{n}\|_{L^{2}([0,T]\times \mathbb{R}^{3})} + \Big\| (1-\Delta)^{-1}JaJ^{-1}(1-\Delta)u_{n}\Big\|_{L^{2}([0,T];{H}^{1}(\mathbb{R}^{3}))},
	\end{eqnarray*}
and
	\begin{equation*}
	\begin{split}
		\Big\| (1-\Delta)^{-1}JaJ^{-1}(1-\Delta)u_{n}\Big\|_{L^{2}([0,T];{H}^{1}(\mathbb{R}^{3}))} \leq & \Big\| (1-\Delta)^{-1}Ja\Big(i\partial_{t}u_{n}- J^{-1}|u_{n}|^{4}u_{n}\Big)\Big\|_{L^{2}([0,T];{H}^{1}(\mathbb{R}^{3}))}\\
		 \leq & \|(1-\Delta)^{-1}Ja\partial_{t}u_{n}\|_{L^{2}([0,T];{H}^{1}(\mathbb{R}^{3}))}  \\
		    & + \|(1-\Delta)^{-1}JaJ^{-1}|u_{n}|^{4}u_{n}\|_{L^{2}([0,T];{H}^{1}(\mathbb{R}^{3}))} \\
	 \leq & \|(1-\Delta)^{-1}J(1-\Delta)^{\frac{1}{2}}(1-\Delta)^{-\frac{1}{2}}a\partial_{t}u_{n}\|_{L^{2}([0,T];{H}^{1}(\mathbb{R}^{3}))} \\
		&   + C\|u_{n}^{5}\|_{L^{2}([0,T];{H}^{-1}(\mathbb{R}^{3}))}.
		\end{split}
	\end{equation*}
	So, these estimates together yield
	\begin{eqnarray*}
		\|u_{n}\|_{L^{2}([0,T];H^{1}(\mathbb{R}^{3}\backslash B_{R+1}(0)))}  
		& \leq & C \|(1-\Delta)^{-\frac{1}{2}}a\partial_{t}u_{n}\|_{L^{2}([0,T]\times \mathbb{R}^{3})} + C \|u_{n}\|_{L^{2}([0,T]\times \mathbb{R}^{3})}\\
		&   & \mbox{   } + C\|u_{n}^{5}\|_{L^{2}([0,T];{H}^{-1}(\mathbb{R}^{3}))}\\
		&\leq & C \|(1-\Delta)^{-\frac{1}{2}}a\partial_{t}u_{n}\|_{L^{2}([0,T]\times \mathbb{R}^{3})} + C \|u_{n}\|_{L^{2}([0,T]\times \mathbb{R}^{3})}\\
		&   & \mbox{   } + C\|u_{n}\|^{5}_{L^{10}([0,T];L^{6}(\mathbb{R}^{3}))},
	\end{eqnarray*}
and thus
	\begin{eqnarray*}
		\int_{0}^{L}	\|u_{n}(t)\|^{2}_{{H}^{1}(\mathbb{R}^{3}\backslash B_{R+1}(0))}   \ dt \longrightarrow 0
	\end{eqnarray*}
	as $n \rightarrow \infty$. Now, using the interpolation
	\begin{eqnarray*}
		\|u_{n}(t)\|_{L^{6}} \leq \|u_{n}(t)\|^{\frac{1}{6}}_{L^{2}}\|u_{n}(t)\|^{\frac{5}{6}}_{L^{10}}
	\end{eqnarray*} 
	and the bound
	\begin{eqnarray*}
		\int_{0}^{L}\|u_{n}(t)\|^{10}_{L^{6}} \ dt & \leq &\int_{0}^{L}\|u_{n}(t)\|^{\frac{5}{3}}_{L^{2}}\|u_{n}(t)\|^{\frac{25}{3}}_{L^{10}} \ dt \\
		& \leq & \sup_{t \in [0,L]} \|u_{n}(t)\|^{\frac{5}{3}}_{L^{2}} \|u_{n}\|^{\frac{3}{25}}_{L^{\frac{25}{3}}_{t}L^{10}_{x}} \\
		& \leq & \sup_{t \in [0,L]} \|u_{n}(t)\|^{\frac{5}{3}}_{L^{2}} \|u_{n}\|^{\frac{3}{25}}_{L^{10}_{t}L^{10}_{x}}  \rightarrow 0,
	\end{eqnarray*}  
we get the desired convergence
	$$v_{n} \rightarrow 0 \ \mbox{  in  } L^{2}_{loc}((0,L);H^{1}(\mathbb{R}^{3})).$$
	
	Finally, choosing $t_{0} \in (0,L)$ such that $\|v_{n}(t_{0})\|_{H^{1}(\mathbb{R}^{3})} \rightarrow 0$ and solving the equation satisfied by $v_{n}$, we obtain 	$$\|v_{n}(t)\|_{H^{1}(\mathbb{R}^{3})} = \|v_{n}(t_{0})\|_{H^{1}(\mathbb{R}^{3})}  \rightarrow 0,$$
	for all $t \in [0,L]$. So 
	$$v_{n} \rightarrow 0 \mbox{ in }  L^{\infty}([0,L];H^{1}(\mathbb{R}^{3}))$$
 which implies 
	$$v_{n}(0) \rightarrow 0  \mbox{ in } H^{1}(\mathbb{R}^{3}),$$ 
which is a contradiction.

	\
	
	\noindent$\bullet$ Case 2:  $\alpha_{n} \rightarrow 0$.
	
	\
	
Estimate \eqref{eq 109} ensures that 
	$$	\int_{0}^{T}\int_{\mathbb{R}^{3}} |(1-\Delta)^{-\frac{1}{2}}a\partial_{t}u_{n}|^{2} \ dxdt  \leq \frac{1}{n}E(u_{n})(0).$$
Define $w_{n} = \frac{u_{n}}{\alpha_{n}}$, where the sequence $\{w_{n}\}$ satisfies
	\begin{equation}\label{eq 110}
		i\partial_{t}w_{n} + \Delta w_{n} - w_{n} -\alpha_{n}^{4} |w_{n}|^{4}w_{n} - a(1-\Delta)^{-1}a\partial_{t}w_{n} = 0 
	\end{equation}
	and 
$$
		\int_{0}^{T}\int_{\mathbb{R}^{3}} |(1-\Delta)^{-\frac{1}{2}}a\partial_{t}w_{n}|^{2} \ dxdt  \leq \frac{1}{n}.
$$
 Now, note that there exists $C>0$, depending on $\lambda_0$, such that 
	$$\frac{1}{C}\|u_{n}(t)\|^{2}_{H^{1}} \leq E(u_{n})(t) \leq C \|u_{n}(t)\|^{2}_{H^{1}},$$
	for all $t \in [0,T]$. Consequently, we get
	\begin{equation*}
		\|w_{n}(t)\|_{H^{1}} = \dfrac{\|u_{n}(t)\|_{H^{1}}}{\sqrt{E(u_{n})(0)}} \leq C \dfrac{\sqrt{E(u_{n})(t)}}{\sqrt{E(u_{n})(0)}} \leq C,
	\end{equation*}
	and
	\begin{equation}\label{eq 112}
		\|w_{n}(0)\|_{H^{1}} = \dfrac{\|u_{n}(0)\|_{H^{1}}}{\sqrt{E(u_{n})(0)}} \geq \frac{1}{\sqrt{C}} \dfrac{\|u_{n}(0)\|_{H^{1}}}{\|u_{n}(0)\|_{H^{1}}} \geq \frac{1}{\sqrt{C}}.
	\end{equation}
So, $ \|w_{n}(0)\|_{H^{1}} \approx 1$ and $w_{n}$ is bounded in $L^{\infty}([0,T];H^{1}(\mathbb{R}^{3})).$ Due to the Strichartz estimates (see Proposition \ref{prop4.12}) for the solutions of the equation \eqref{eq 110}, there exists $C>0$, such that
	\begin{eqnarray*}
		\|\nabla w_{n}\|_{L^{10}([0,T];L^{\frac{30}{13}}(\mathbb{R}^{3}))} & \leq & C\Big(\|w_{n}(0)\|_{H^{1}} + \alpha_{n}^{4} \|\nabla w_{n}\|_{L^{10}([0,T];L^{\frac{30}{13}}(\mathbb{R}^{3}))} \|w_{n}\|^{4}_{L^{10}([0,T];L^{10}(\mathbb{R}^{3}))} \\
		&   &  \mbox{   } + \alpha_{n}^{4} \|w_{n}\|^{5}_{L^{10}([0,T];L^{10}(\mathbb{R}^{3}))}\Big) \\
		& \leq &C	\Big(1 + \alpha_{n}^{4} \|\nabla w_{n}\|^{5}_{L^{10}([0,T];L^{\frac{30}{13}}(\mathbb{R}^{3}))}  \Big).
	\end{eqnarray*}
	A bootstrap argument gives us that $ \|\nabla w_{n}\|_{L^{10}([0,T];L^{\frac{30}{13}}(\mathbb{R}^{3}))} $ is bounded and, thus, $\|w_{n}\|_{L^{10}([0,T];L^{10}(\mathbb{R}^{3}))}$ is bounded, due to the Sobolev embedding. Additionally, if we consider the sequence $\{ w_{n}\}$ satisfying the Cauchy problem
	\begin{equation} \label{rarara}
		\left\{
		\begin{array}{lr}
			i\partial_{t} \tilde{w}_{n} + \Delta \tilde{w}_{n} -\tilde{w}_{n} -a(1-\Delta)^{-1}a\partial_{t}\tilde{w}_{n} =0, \ \mbox{ on } [0,T] \times \mathbb{R}^{3}, \\
			\tilde{w}_{n}(0)= w_{n}(0),
		\end{array}
		\right.
	\end{equation}
	an application of Proposition \ref{prop4.12} gives
		\begin{equation*}
		\begin{split}
		\|w_{n}-\tilde{w}_{n}\|_{L^{10}([0,T];L^{10}(\mathbb{R}^{3}))} &+ \|w_{n}-\tilde{w}_{n}\|_{L^{\infty}([0,T];H^{1}(\mathbb{R}^{3}))} \\
		 \leq & C\Big(\alpha_{n}^{4}\|\nabla w_{n}\|_{L^{10}([0,T];L^{\frac{30}{13}}(\mathbb{R}^{3}))}\|w_{n}\|^{4}_{L^{10}([0,T];L^{10}(\mathbb{R}^{3}))} \\
		&   + \alpha_{n}^{4}\|w_{n}\|^{5}_{L^{10}([0,T];L^{10}(\mathbb{R}^{3}))}  \Big) \\
		 \leq & C\Big(\alpha_{n}^{4} \|\nabla w_{n}\|^{5}_{L^{10}([0,T];L^{\frac{30}{13}}(\mathbb{R}^{3}))}\Big)		 \rightarrow 0,
		\end{split}
	\end{equation*}
as $n\rightarrow \infty$.
	
	We have  \begin{equation}\label{ttttt}\|(1-\Delta)^{-\frac{1}{2}}a\partial_{t}\tilde{w}_{n}\|_{L^{2}([0,T];L^{2}(\mathbb{R}^{3}))} \rightarrow 0,\end{equation} as $n\rightarrow \infty.$ Indeed,  
	\begin{equation*}
	\begin{split}
		\|(1-\Delta)^{-\frac{1}{2}}a\partial_{t}\tilde{w}_{n}\|_{L^{2}([0,T];L^{2}(\mathbb{R}^{3}))} 
	 \leq & 	\|(1-\Delta)^{-\frac{1}{2}}a(\partial_{t}\tilde{w}_{n}-\partial_{t}w_{n})\|_{L^{2}([0,T];L^{2}(\mathbb{R}^{3}))} \\
		&    + 	\|(1-\Delta)^{-\frac{1}{2}}a\partial_{t}w_{n}\|_{L^{2}([0,T];L^{2}(\mathbb{R}^{3}))} \\
		 \leq & \|\partial_{t}\tilde{w}_{n}-\partial_{t}w_{n}\|_{L^{2}([0,T];H^{-1}(\mathbb{R}^{3}))} \\
		 &+ 	\|(1-\Delta)^{-\frac{1}{2}}a\partial_{t}w_{n}\|_{L^{2}([0,T];L^{2}(\mathbb{R}^{3}))},
		 \end{split}
	\end{equation*}
and 	\begin{equation*}
	\begin{split}
		\|\partial_{t}\tilde{w}_{n}&-\partial_{t}w_{n}\|_{L^{2}([0,T];H^{-1}(\mathbb{R}^{3}))}  \\
		\leq&  \|-iJ^{-1}(I-\Delta)\tilde{w}_{n} +iJ^{-1}(I-\Delta)w_{n} + iJ^{-1}\alpha_{n}^{4}|w_{n}|^{4}w_{n}\|_{L^{2}([0,T];H^{-1}(\mathbb{R}^{3}))} \\
		 \leq & \|J^{-1}(I-\Delta)(\tilde{w}_{n}-w_{n})\|_{L^{2}([0,T];H^{-1}(\mathbb{R}^{3}))}\\& + \alpha_{n}^{4}\|J^{-1}w_{n}^{5}\|_{L^{2}([0,T];H^{-1}(\mathbb{R}^{3}))} \\
		 \leq & C \|\tilde{w}_{n}-w_{n}\|_{L^{2}([0,T];H^{1}(\mathbb{R}^{3}))} + C \alpha_{n}^{4}\|w_{n}\|^{5}_{L^{10}([0,T];L^{6}(\mathbb{R}^{3}))} \rightarrow 0,
		 \end{split}
	\end{equation*}
as $n\rightarrow \infty$, and $J$ is given in Theorem \ref{JJJ}. 	

Now, since $\{\tilde{w}_{n}\}$ is bounded in $L^{\infty}([0,T];H^{1}(\mathbb{R}^{3}))$, we can extract a subsequence (still denoted by $\{\tilde{w}_{n}\}$) such that $\tilde{w}_{n}(t) \rightharpoonup w(t)$ weakly. Passing to the limit in \eqref{rarara}, and taking into account the convergence  \eqref{ttttt}, the function $w$ satisfies
	\begin{equation*} 
		\left\{
		\begin{array}{ll}
			i\partial_{t} w + \Delta w -w  =0,&\mbox{ on } (0,T) \times \mathbb{R}^{3}, \\
			\partial_{t} w = 0,& \mbox{ on } (0,T) \times \mathbb{R}^{3}\backslash B_{R+1}(0).
		\end{array}
		\right.
	\end{equation*}
Let $v=\partial_{t} w$. Taking the derivative with respect to time in the first equation of the system above, we have that $v$ satisfies
	\begin{equation*} 
		\left\{
		\begin{array}{ll}
			i\partial_{t} v + \Delta v -v  =0,& \mbox{ on } (0,T) \times \mathbb{R}^{3},\\
			v = 0, & \mbox{ on } (0,T) \times \mathbb{R}^{3}\backslash B_{R+1}(0).
		\end{array}
		\right.
	\end{equation*}
Since $v \in C^{\infty}\big((0,T) \times \mathbb{R}^{3}\big)$ (see, for instance,  \cite[Proposition 2.3]{RoZhaJDE}) and satisfies 
	\begin{equation*} 
		\left\{
		\begin{array}{ll}
			i\partial_{t} v + \Delta v -v  =0,& \mbox{ on } (0,T) \times \mathbb{R}^{3}, \\
			v = 0, & \mbox{ on } (0,T) \times \mathbb{R}^{3}\backslash B_{R+1}(0),
		\end{array}
		\right.
	\end{equation*}
	by an unique continuation property (see \cite{RoMe}), $v\equiv0$ on $(0,T) \times \mathbb{R}^{3}$. Therefore, $\partial_{t} w\equiv 0$ in $(0,T) \times \mathbb{R}^{3}$ and 
	$\Delta w - w = 0.$
	Multiplying this equation by $\overline{w}$ and integrating by parts, we get
	$$\int_{\mathbb{R}^{3}}|\nabla w|^{2} \ dx + \int_{\mathbb{R}^{3}}|w|^{2} \ dx =0,$$
which implies	$w \equiv 0.$ Therefore, $\tilde{w}_{n} \rightharpoonup 0$ in $H^{1}(\mathbb{R}^{3})$.
	
	Finally, let us prove that 
	$$\tilde{w}_{n} \rightarrow 0 \text{ in }  L^{2}_{loc}\big((0,T);H^{1}_{loc}(\mathbb{R}^{3}\backslash B_{R+1}(0))\big).$$	Due to the convergence
	$$\|(1-\Delta)^{-\frac{1}{2}}a\partial_{t}\tilde{w}_{n}\|_{L^{2}([0,T] \times \mathbb{R}^{3})} \longrightarrow 0$$
	we get	$$ \| (1-\Delta)^{-\frac{1}{2}}aJ^{-1}(I-\Delta)\chi \tilde{w}_{n}\|_{L^{2}([0,T] \times \mathbb{R}^{3})}\longrightarrow 0,$$ for $\chi \in C_{0}^{\infty}([0,T] \times \mathbb{R}^{3})$ and $J$ given as in the proof of Theorem \ref{JJJ}. Indeed,
	\begin{equation*}
	\begin{split}
		\| (1-\Delta)^{-\frac{1}{2}}aJ^{-1}(I-\Delta)\chi \tilde{w}_{n}\|_{L^{2}([0,T] \times \mathbb{R}^{3})}  = & \| [(1-\Delta)^{-\frac{1}{2}}aJ^{-1}(I-\Delta),\chi] \tilde{w}_{n}\|_{L^{2}([0,T] \times \mathbb{R}^{3})}\\
		   & + \|\chi (1-\Delta)^{-\frac{1}{2}}aJ^{-1}(I-\Delta) \tilde{w}_{n}\|_{L^{2}([0,T] \times \mathbb{R}^{3})} \\
		 \leq & \|\chi_{B} \tilde{w}_{n}\|_{L^{2}([0,T] \times \mathbb{R}^{3})} \\&+ \|\chi (1-\Delta)^{-\frac{1}{2}}aJ^{-1}(I-\Delta) \tilde{w}_{n}\|_{L^{2}([0,T] \times \mathbb{R}^{3})},
		\end{split}
	\end{equation*}
and this yields that 
	$$\Big\langle (1-\Delta)(J^{-1})^{*}a(1-\Delta)^{-1}aJ^{-1}(1-\Delta)\chi \tilde{w}_{n}, \chi \tilde{w}_{n} \Big\rangle_{L^{2}((0,T)\times \mathbb{R}^{3})}\longrightarrow0.$$
	Thus, Proposition \ref{prop6.1} gives us
	$$\int_{(0,T)\times \omega \times S^{2}} 1+|\xi|^{2} \ d\mu = 0.$$
Moreover, Corollary \ref{coro6.2} ensures that 
	\begin{equation}\label{eq 115}
		\tilde{w}_{n} \rightarrow0 \mbox{   in   } L^{2}_{loc}\big((0,T);H^{1}_{loc}(\mathbb{R}^{3})\big).
	\end{equation}
	On other hand, since $\|(1-\Delta)^{-\frac{1}{2}}a\partial_{t}\tilde{w}_{n}\|_{L^{2}([0,T] \times \mathbb{R}^{3})} \longrightarrow 0$, we get 
	$$\|a\partial_{t}\tilde{w}_{n}\|_{L^{2}([0,T];H^{-1}( \mathbb{R}^{3}))} \longrightarrow 0.$$
	Let $\chi_{\omega} \in C^{\infty}(\mathbb{R}^{3})$ such that $\chi_{\omega} =1$ on $\supp(a)$. Then
	\begin{equation*}
	\begin{split}
		\|ai\partial_{t}\tilde{w}_{n}\|_{L^{2}([0,T];H^{-1}( \mathbb{R}^{3}))}  = & \|aJ^{-1}(1-\Delta)\tilde{w}_{n}\|_{L^{2}([0,T];H^{-1}( \mathbb{R}^{3}))} \\
		 \geq & \eta \|\chi_{\omega}J^{-1}(1-\Delta)\tilde{w}_{n}\|_{L^{2}([0,T];H^{-1}( \mathbb{R}^{3}))} \\
		 \geq & \eta \|J^{-1}\chi_{\omega}(1-\Delta)\tilde{w}_{n}\|_{L^{2}([0,T];H^{-1}( \mathbb{R}^{3}))} \\
		 \geq & C \|(1-\Delta)\chi_{\omega}\tilde{w}_{n}\|_{L^{2}([0,T];H^{-1}( \mathbb{R}^{3}))} \\& - C\|[(1-\Delta),\chi_{\omega}]\tilde{w}_{n}\|_{L^{2}([0,T];H^{-1}( \mathbb{R}^{3}))} .
		 \end{split}
	\end{equation*}
	For $\chi_{B} \in C_{0}^{\infty}(\mathbb{R}^{3})$, we have
	\begin{eqnarray*}
		\|[(1-\Delta),\chi_{\omega}]\tilde{w}_{n}\|_{L^{2}([0,T];H^{-1}( \mathbb{R}^{3}))} & = & 	\|[(1-\Delta),\chi_{\omega}]\chi_{B}\tilde{w}_{n}\|_{L^{2}([0,T];H^{-1}( \mathbb{R}^{3}))} \\
		& \leq& C \|\chi_{B}\tilde{w}_{n}\|_{L^{2}([0,T];L^{2}( \mathbb{R}^{3}))} \rightarrow  0.
	\end{eqnarray*}
Hence, 
	\begin{equation*}
	\begin{split}
		\|(1-\Delta)\chi_{\omega}\tilde{w}_{n}\|_{L^{2}([0,T];H^{-1}( \mathbb{R}^{3}))}  \leq & C\|[(1-\Delta),\chi_{\omega}]\tilde{w}_{n}\|_{L^{2}([0,T];H^{-1}( \mathbb{R}^{3}))}\\& + 	C \|ai\partial_{t}\tilde{w}_{n}\|_{L^{2}([0,T];H^{-1}( \mathbb{R}^{3}))} \\
		&\rightarrow 0.
		\end{split}
	\end{equation*}
	Then,
	\begin{eqnarray*}
		\|\chi_{\omega}\tilde{w}_{n}\|_{L^{2}([0,T];H^{1}( \mathbb{R}^{3}))} & = & 	\|(1-\Delta)^{-1}(1-\Delta)\chi_{\omega}\tilde{w}_{n}\|_{L^{2}([0,T];H^{1}( \mathbb{R}^{3}))} \\
		& \leq & 	\|(1-\Delta)\chi_{\omega}\tilde{w}_{n}\|_{L^{2}([0,T];H^{-1}( \mathbb{R}^{3}))}  \rightarrow  0.
	\end{eqnarray*}
	This means that 
	\begin{equation}\label{eq 116}
		\tilde{w}_{n} \longrightarrow 0 \mbox{   in   } L^{2}\big([0,T];H^{1}(\mathbb{R}^{3}\backslash B_{R+1}(0))\big).
	\end{equation}
	By \eqref{eq 115} and \eqref{eq 116}, we conclude that 
	\begin{equation*}
		\tilde{w}_{n} \longrightarrow 0 \mbox{   in   } L^{2}_{loc}\big((0,T);H^{1}(\mathbb{R}^{3})\big).
	\end{equation*}
	So, choosing $t_{0} \in (0,T)$ such that $\|\tilde{w}_{n}(t_{0})\|_{H^{1}} \rightarrow0$ and solving the equation satisfied by $\tilde{w}_{n}$ with $\tilde{w}_{n}(t_{0})$ as initial data, we have 
	$$\tilde{w}_{n}(t) = e^{i(t-t_{0})(\Delta - I)}\tilde{w}_{n}(t_{0}) + \int_{t_{0}}^{t}e^{i(t-\tau)(\Delta - I)}a(1-\Delta)^{-1}a\partial_{t}\tilde{w}_{n} \ d\tau .$$
	Hence,
	\begin{eqnarray*}
		\|\tilde{w}_{n}(t)\|_{H^{1}} & \leq & c \|\tilde{w}_{n}(t_{0})\|_{H^{1}} + c\|a(1-\Delta)^{-1}a\partial_{t}\tilde{w}_{n}\|_{L^{1}([0,T];H^{1})}  \rightarrow 0.
	\end{eqnarray*}
	Therefore,
	$$\tilde{w}_{n} \longrightarrow 0 \mbox{   in   } L^{\infty}\big([0,T];H^{1}(\mathbb{R}^{3})\big),$$
	and 
	$$\|w_{n}(0)\|_{H^{1}} = \|\tilde{w}_{n}(0)\|_{H^{1}}\rightarrow 0,$$
which is a contradiction with \eqref{eq 112}.
\end{proof}

Now, we finally complete the proof of the Theorem \ref{mainteo}.
\begin{proof}[Proof of Theorem \ref{mainteo}]
	
	Fix $T>0$ such that Theorem \ref{weakobs} applies. Then, there exists $\varepsilon > 0$ such that for any $u_{0}$ satisfying
	\begin{equation}\label{desi}
		\|u_{0}\|_{H^{1}} \leq \lambda_{0}; \ \ \ \|u_{0}\|_{H^{-1}} \leq \varepsilon ,
	\end{equation}
	the strong observability estimate \eqref{obser} holds for any solution of the damped equation \eqref{eq 4}. This means that there exists $B>0$ such that any solution of the damped equation satisfying \eqref{desi} fulfills
	\begin{equation}\label{decay}
		E(u)(T) \leq (1-B)E(u)(0).
	\end{equation}
	Choose $N \in \mathbb{N}$ large enough such that $(1 - B)^{N} C(\lambda_{0}) \leq {\varepsilon^{2}}.$
Corollary \ref{coro4.40} allow us to choose $\delta$ small enough such that the assumptions $\|u_{0}\|_{H^{1}} \leq R_{0}$ and $\|u_{0}\|_{H^{-1}} \leq \delta$ imply $\|u(nT)\|_{H^{-1}} \leq \varepsilon$, for $ 0\leq n \leq N.$ So, with that choice, we have $E(u)(NT) \leq (1 - B)^{N} E(u)(0).$ Then, by the decreasing of energy, we have $\|u(t)\|^{2}_{H^{-1}}  \leq \varepsilon^{2}$ for all $t  \geq NT$.   Hence, the decay estimate \eqref{decay} holds in each interval $[nT,(n + 1)T],$ $n \in \mathbb{N}$, and 
$$E(u)(nT) \leq (1-B)^{n}E(u)(0),$$
giving the desired result. 
\end{proof}

\appendix

	\section{A review of the Cauchy problem}\label{apA}

	\subsection{Existence}  
		In this section, we review some results for the initial value problem
	\begin{equation} \label{eq 5}
		\left\{
		\begin{array}{ll}
			i\partial_{t} u + \Delta u -u -|u|^{4}u = g,& (t,x) \in \mathbb{R} \times \mathbb{R}^{3}, \\
			u(0) = u_{0} \in {H}^{1}(\mathbb{R}^{3}). 
		\end{array}
		\right.
	\end{equation}
	where $g \in L^{\infty}_{loc}(\mathbb{R},H^{1}(\mathbb{R}^{3}))$. 
	We begin with some definitions.
	\begin{definition} Let s $\in \mathbb{R}$. The homogeneous Sobolev space $\dot{H}^s\left(\mathbb{R}^d\right)$ is the space of tempered distributions $u$ over $\mathbb{R}^d$ which have Fourier transform belonging to $L_{l o c}^1\left(\mathbb{R}^d\right)$ and satisfy
		$$
		\|u\|_{\dot{H}^s}^2 \stackrel{\text { def }}{=} \int_{\mathbb{R}^d}|\xi|^{2 s}|\widehat{u}(\xi)|^2 d \xi<\infty .
		$$
		We note that the spaces $\dot{H}^s$ and $\dot{H}^{s^{\prime}}$ are not comparable for inclusion.
	\end{definition}
	
	\begin{definition} A pair $(q,r)$ is called $L^{2}$-admissible if $r \in [2,6)$ and q satisfies 
		\begin{equation*} 
			\frac{2}{q} + \frac{3}{r} = \frac{3}{2}. 
		\end{equation*}
		A pair $(q,r)$ is called $H^{1}$-admissible if $r \in [6,+\infty)$ and q satisfies 
		\begin{equation*} 
			\frac{2}{q} + \frac{3}{r} = \frac{1}{2}. 
		\end{equation*}
	\end{definition}
\begin{remark}
	If $(q,r)$ is a $L^{2}$-admissible pair, then $2 \leq q \leq \infty$. Note that the pair $(\infty,2)$ is always $L^{2}$-admissible. The pair $(2,\frac{2N}{N-2})$ is $L^{2}$- admissible if $N> 3$.
\end{remark}
	With these definitions in hand, we present two results that are paramount to prove that the Cauchy problem \eqref{eq 5} is well-posed. The first one gives the so-called \textit{Strichartz estimates} and the second one is a standard Sobolev embedding. These results can be found in \cite{cazenave_book, KeelTao}.
	\begin{lemma}\label{strichartz}
		Let $(q,r)$ be a $L^{2}$-admissible pair. We have
		\begin{equation}\label{item_i}
			\|e^{it\Delta}h\|_{L^{q}_{t}L^{r}_{x}} \leq c\|h\|_{L^{2}},
		\end{equation}
		\begin{equation}\label{item_ii}
			\Bigg\|\int_{-\infty}^{+\infty}e^{i(t-\tau)\Delta}g\ d\tau\Bigg\|_{L^{q}_{t}L^{r}_{x}} +\Bigg\|\int_{0}^{t}e^{i(t-\tau)\Delta}g\ d\tau\Bigg\|_{L^{q}_{t}L^{r}_{x}} \leq c\|g\|_{L^{q'}_{t}L^{r'}_{x}} ,
		\end{equation}
		and
		\begin{equation*}
			\Bigg\|\int_{-\infty}^{+\infty}e^{it\Delta}g(\tau)\ d\tau\Bigg\|_{L^{2}_{x}} \leq C\|g\|_{L^{q'}_{t}L^{r'}_{x}}.
		\end{equation*}
		Additionally, we have
		\begin{equation}\label{rmk1}
			\Bigg\|\int_{-\infty}^{+\infty}e^{i(t-\tau)\Delta}g(\tau)\ d\tau\Bigg\|_{L^{q}_{t}L^{r}_{x}} \leq C \|g\|_{L^{m'}_{t}L^{n'}_{x}}
		\end{equation}
		where $(q, r)$, $(m, n)$ are any $L^{2}$-admissible pair, wich is an generalization of \eqref{item_ii}.
	\end{lemma}
Define the $S(I),$ $W(I)$ and $Z(I)$ norms for an interval $I$ by
	$$\|u\|_{S(I)} = \|u\|_{L^{10}(I;L^{10}(\mathbb{R}^{3}))}, \ \  \|u\|_{Z(I)} = \|u\|_{L^{10}(I;L^{\frac{30}{13}}(\mathbb{R}^{3}))} \ \mbox{and} \ \ \ \|u\|_{W(I)} = \|u\|_{L^{\frac{10}{3}}(I;L^{\frac{10}{3}}(\mathbb{R}^{3}))}.$$
	The first theorem gives us the existence of the solution to problem \eqref{eq 5}. The proof is similar to the proof of Theorem 2.4 from \cite{BrCaCaSa} and, thus, we will omit it. 
	\begin{theorem} 
		Let  $u_{0} \in H^{1}(\mathbb{R}^{3})$. , with $\|u_{0}\|_{H^{1}}\leq A$. If $\|u_{0}\|_{H^{1}}$ is small enough, there exists an unique $u \in C(\mathbb{R}_{+}, H^{1}(\mathbb{R}^{3}))$ solution of \eqref{eq 5} with
	\begin{equation*}
		\|u\|_{S([0,T])} < \infty , \ \  \|\nabla u\|_{W([0,T])} < \infty  \ \mbox{  and  } \ \|\nabla u\|_{Z([0,T])} < \infty
	\end{equation*} 
for all $T>0$.
	\end{theorem}

 Now, we prove a result that ensures the existence of solutions for the ${H}^{1}$ critical nonlinear Schr\"{o}dinger equation with a damping term, that is, changing $g$ by $a(x)(1-\Delta)^{-1}a(x)\partial_{t}u$ in the system \eqref{eq 5}. 
	
	\begin{theorem}\label{JJJ}
		Let $T>0$, $u_{0} \in {H}^{1}(\mathbb{R}^{3})$, 
		 and $a(x)\in C^{\infty}(\mathbb{R}^{3})$ a non-negative real valued function. If $\|u_{0}\|_{H^{1}}$ is small enough, then there exists an unique $u \in C(\mathbb{R}_{+},{H}^{1}(\mathbb{R}^{3}))$, solution of the system
		\begin{equation} \label{eq 8}
			\left\{
			\begin{array}{ll}
				i\partial_{t} u + \Delta u -u -|u|^{4}u -a(x)(1-\Delta)^{-1}a(x)\partial_{t} u = 0, & (t,x) \in [0,T] \times \mathbb{R}^{3}, \\
				u(0) = u_{0}, &x \in \mathbb{R}^{3},
			\end{array}
			\right.
		\end{equation}
		with  $\|u\|_{S([0,T])} < \infty , \ \  \|\nabla u\|_{W([0,T])} < \infty  \ \mbox{  and  } \ \|\nabla u\|_{Z([0,T])} < \infty $
		for all $T <\infty$.
	\end{theorem}
	\begin{proof}
		We claim that the operator $Jv = (1-ia(x)(1 -\Delta )^{-1} a(x))v$ is a pseudodifferential operator of order 0 which defines an isomorphism on the space $H^{s}(\mathbb{R}^{3})$, for $s \in \mathbb{R}$ and, in particular, on $L^{p}(\mathbb{R}^{3})$.
		Indeed, we can write $J$ as $J = I + J_{1}$, where $J_{1}$ is an anti-self-adjoint operator on $L^{2}(\mathbb{R}^{3})$. Thus $J$ is an isomorphism on $L^{2}(\mathbb{R}^{3})$ and, due to the ellipticity, on $H^{s}(\mathbb{R}^{3})$, for $s > 0$. Moreover, $J^{-1}$ (considered, for example, acting on $L^{2}([0,T]\times \mathbb{R}^{3})$) is a
		pseudodifferential operator of order 0 and satisfies $J^{-1} = 1-J_{1}J^{-1}$. 
		
		We denote $v=Ju$ and write the system \eqref{eq 8} as
		\begin{equation*}
			\left\{
			\begin{array}{ll}
				\partial_{t} v -i \Delta v -R_{0}v +i|u|^{4}u = 0,&(t,x) \in [0,T] \times \mathbb{R}^{3}, \\
				v(0) = v_{0}=Ju_{0}, &x \in \mathbb{R}^{3},
			\end{array}
			\right.
		\end{equation*}
		where $R_{0} = -i\Delta J_{1}J^{-1} + iJ^{-1}$ is a pseudodifferential operator of order 0.
		This Cauchy problem is equivalent to the integral equation 
		$$v(t) = e^{it\Delta}v_{0} + \int_{0}^{t}e^{i(t-\tau)\Delta}[R_{0}v - i|u|^{4}u] \ d\tau.$$
		Let $I=[0,T]$ and consider the set
		$X_{I}$
		with norm
		$$\|v\|_{X_{I}} = \sup_{t \in I} \|\nabla v(t)\|_{L^{2}} + \sup_{t \in I} \| v(t)\|_{L^{2}} +\|v\|_{S(I)} + \|\nabla v\|_{W(I)}.$$
		We now set $B_{R} = \Big\{v \in X_{I}; \ \ \|v\|_{X_{I}} \leq R\Big\}$, where $R>0$. By Duhamel's formula, we define the functional
		\begin{equation}\label{eq 9}
			\Phi_{u_{0}}(v)(t) = e^{it\Delta}v_{0} + \int_{0}^{t} e^{i(t-\tau)\Delta}R_{0}v\ d\tau - \int_{0}^{t} e^{i(t-\tau)\Delta} i|u|^{4}u \ d\tau 
		\end{equation}
		Our goal is to show that this functional has a fixed point, considering $\Phi_{u_{0}}$ in a suitable ball $B_{R} $. We first show that we can choose $R$ such that $\Phi(v): B_{R} \longrightarrow B_{R}$. 
		Indeed, by \eqref{eq 9}, we get
		\begin{eqnarray*}
			\|\nabla \Phi_{u_{0}}(v)\|_{L^{2}_{x}} & \leq & \|\nabla e^{it\Delta}v_{0}\|_{L^{2}} + \Big\| \int_{0}^{t} \nabla e^{i(t-\tau)\Delta }|u|^{4}u\ d\tau \Big\|_{L^{2}_{x}} +  \Big\| \int_{0}^{t} \nabla e^{i(t-\tau)\Delta }R_{0}v\ d\tau \Big\|_{L^{2}_{x}}\\
			& \leq &  \|\nabla v_{0}\|_{L^{2}} + C\|\nabla|u|^{4}u\|_{L^{\frac{10}{7}}_{t} L^{\frac{10}{7}}_{x}} + C \|\nabla R_{0}v\|_{L^{1}_{t}L^{2}_{x}}\\
			& \leq &\|\nabla v_{0}\|_{L^{2}} + 
			C \|u\|_{S(I)}^{4}\|\nabla u\|_{W(I)} + C\|[\nabla,R_{0}]v\|_{L^{1}_{t}L^{2}_{x}} + C\|R_{0}\nabla v\|_{L^{1}_{t}L^{2}_{x}}.
		\end{eqnarray*}
		On the other hand, observe that
		\begin{eqnarray*}
			\|\nabla u\|_{W(I)}
			& = & \|[\nabla,J^{-1}]v + J^{-1}\nabla v\|_{W(I)} \\
			& \leq & C \|v\|_{W(I)}  + C \|\nabla v\|_{W(I)}.
		\end{eqnarray*}
		Then, 
				\begin{equation*}
		\begin{split}
			\|\nabla \Phi_{u_{0}}(v)\|_{L^{2}_{x}} \leq &\|\nabla v_{0}\|_{L^{2}} + 
			C \|v\|_{S(I)}^{4}\big(  \|v\|_{W(I)}  +  \|\nabla v\|_{W(I)}\big) + C\|[\nabla,R_{0}]v\|_{L^{1}_{t}L^{2}_{x}} + \|R_{0}\nabla v\|_{L^{1}_{t}L^{2}_{x}} \\
			\leq & \| v_{0}\|_{H^{1}} + 
			C \|v\|_{S(I)}^{4} \|v\|_{W(I)}  + C \|v\|_{S(I)}^{4} \|\nabla v\|_{W(I)}    \\&+ CT\sup_{t \in I}\|v(t)\|_{L^{2}} + CT\sup_{t \in I}\|\nabla v(t)\|_{L^{2}}. 
			\end{split}
		\end{equation*} 
By interpolation, one has
		\begin{eqnarray*}
			\|v(t)\|_{L^{\frac{10}{3}}} & \leq & \|v(t)\|^{\frac{2}{5}}_{L^{2}}\|v(t)\|^{\frac{3}{5}}_{L^{6}},
		\end{eqnarray*}
which ensures
		\begin{eqnarray*}
			\int_{0}^{T}\|v(t)\|^{\frac{10}{3}}_{L^{\frac{10}{3}}} \ dt 
			 \leq T \sup_{t \in I} \|v(t)\|^{\frac{4}{3}}_{L^{2}}\sup_{t \in I} \|v(t)\|^{2}_{L^{6}} 
			 \leq  T \|v\|^{\frac{4}{3}}_{X_{I}}  \|v\|^{2}_{X_{I}} 
			 \leq  T \|v\|^{\frac{10}{3}}_{X_{I}} ,
		\end{eqnarray*}
implying that
		$$ \|v\|_{W(I)} \leq T^{\frac{3}{10}} \|v\|_{X_{I}}.$$
		Hence,
		\begin{eqnarray*}
			\|\nabla \Phi_{u_{0}}(v)\|_{L^{2}_{x}} 
			& \leq & \|\nabla v_{0}\|_{L^{2}} + 
			C T^{\frac{3}{10}} \|v\|_{S(I)}^{4} \|v\|_{X_{I}}  + C \|v\|_{S(I)}^{4} \|\nabla v\|_{W(I)} \\
			&   & \mbox{   }+ CT\sup_{t \in I}\|v(t)\|_{L^{2}} + CT\sup_{t \in I}\|\nabla v(t)\|_{L^{2}} \\
			& \leq & C \|v_{0}\|_{H^{1}} + CT^{\frac{3}{10}}\|v\|^{5}_{X_{I}} + C\|v\|^{5}_{X_{I}} + CT \|v\|_{X_{I}},
		\end{eqnarray*} 
		where, for these inequalities, we have used Lemma \ref{strichartz}, precisely, inequalities \ref{item_i}, with $(q,r) =\displaystyle{\Big(\frac{10}{3},\frac{10}{3}\Big)}$, and \ref{rmk1}, with $(q,r) =\displaystyle{\Big(\frac{10}{3},\frac{10}{3}\Big)}$ and $(m,n)= (\infty,2)$. Note that, 
		\begin{eqnarray*}
			\|\Phi_{u_{0}}(v)\|_{L^{2}_{x}} 
			& \leq & C \|v_{0}\|_{H^{1}} + CT\|v\|^{5}_{X_{I}} + CT \|v\|_{X_{I}}
		\end{eqnarray*}
and
		\begin{eqnarray*}
			\|\nabla \Phi_{u_{0}}(v)\|_{W(I)} 
			& \leq &\|\nabla v_{0}\|_{L^{2}} + 
			C \|u\|_{S(I)}^{4}\|\nabla u\|_{W(I)} + C\|[\nabla,R_{0}]v\|_{L^{1}_{t}L^{2}_{x}} + C\|R_{0}\nabla v\|_{L^{1}_{t}L^{2}_{x}}.
		\end{eqnarray*}
Similarly as before, one can get
		\begin{eqnarray*}
			\|\nabla \Phi_{u_{0}}(v)\|_{W(I)}
			& \leq & \|\nabla v_{0}\|_{L^{2}} + 
			C T^{\frac{3}{10}} \|v\|_{S(I)}^{4} \|v\|_{X_{I}}  + C \|v\|_{S(I)}^{4} \|\nabla v\|_{W(I)} \\
			&   & \mbox{   }+ CT\sup_{t \in I}\|v(t)\|_{L^{2}} + CT\sup_{t \in I}\|\nabla v(t)\|_{L^{2}}\\
			& \leq & C \|v_{0}\|_{H^{1}} + CT^{\frac{3}{10}}\|v\|^{5}_{X_{I}} + C\|v\|^{5}_{X_{I}} + CT \|v\|_{X_{I}}.
		\end{eqnarray*} 
		Finally,
		\begin{eqnarray*}
			\|\Phi_{u_{0}}(v)\|_{S(I)}
			& \leq & C \|v_{0}\|_{H^{1}} + CT^{\frac{3}{10}}\|v\|^{5}_{X_{I}} + C\|v\|^{5}_{X_{I}} + CT \|v\|_{X_{I}},
		\end{eqnarray*}
		where we used  Lemma \ref{strichartz}, inequalities \ref{item_i}, with $(q,r) =\displaystyle{\Big(10,\frac{30}{13}\Big)}$, and \ref{rmk1}, with $(q,r) =\Big(10,\frac{30}{13}\Big)$ and again $(m,n)= (\infty,2)$.  Putting all these pieces of information together means that 
		\begin{eqnarray*}
			\|\Phi_{u_{0}}(v)\|_{X_{I}} & \leq & C \| v_{0}\|_{H^{1}} 
			+ CT^{\frac{3}{10}}\|v\|^{5}_{X_{I}} + C\|v\|^{5}_{X_{I}} + CT \|v\|_{X_{I}}.
		\end{eqnarray*}
Now, choosing $T < \min\big\{1,\frac{1}{4C}\big\}$, $A< \frac{R}{8C}$ and $R < \frac{1}{(4C)^{\frac{1}{4}}}$, we conclude that $\Phi_{u_{0}}$ reproduces the ball $B_{R}$ into itself.
		
		Now, let us prove that  $\Phi(v)$ is a contraction. To this end, consider the two systems 
		\begin{equation*} 
			\left\{
			\begin{array}{ll}
				i\partial_{t} u + \Delta u -u - |u|^{4}u -a(1-\Delta)^{-1}a\partial_{t}u= 0,& (t,x) \in [0,T] \times \mathbb{R}^{3}, \\
				u(0) = u_{0}, & x \in \mathbb{R}^{3},
			\end{array}
			\right.
		\end{equation*}
		and 
		\begin{equation*} 
			\left\{
			\begin{array}{ll}
				i\partial_{t} z + \Delta z - z - |z|^{4}z -a(1-\Delta)^{-1}a\partial_{t}z= 0, & (t,x) \in [0,T] \times \mathbb{R}^{3}, \\
				z(0) = u_{0}, &x \in \mathbb{R}^{3}.
			\end{array}
			\right.
		\end{equation*}
Performing the same transformation carried out at the beginning of the proof, we have 
		\begin{equation*}
			\left\{
			\begin{array}{lll}
				\partial_{t} v -i \Delta v -R_{0}v + i |u|^{4}u = 0,& (t,x) \in [0,T] \times \mathbb{R}^{3},\\
				v=Ju,\\
				v(0) =v_{0}= Ju_{0}, & x \in \mathbb{R}^{3},
			\end{array}
			\right.
		\end{equation*}
		and 
		\begin{equation*}
			\left\{
			\begin{array}{lll}
				\partial_{t} w -i \Delta w - R_{0}w + i|z|^{4}z = 0, & (t,x) \in [0,T] \times \mathbb{R}^{3},\\
				w=Jz,\\
				w(0) = w_{0}=v_{0} = J u_{0}, &x \in \mathbb{R}^{3}.
			\end{array}
			\right.
		\end{equation*}
		Using Duhamel's formula, 
		$$\Phi_{u_{0}}(v) - \Phi_{u_{0}}(w) = \int_{0}^{t}e^{i(t-\tau)\Delta}R_{0}(v-w) \ d\tau - \int_{0}^{t} e^{i(t-\tau)\Delta}i\Big(|u|^{4}u -|z|^{4}z\Big)\ d\tau .$$
Computations which are analogous to the ones in the previous step ensure
		\begin{eqnarray*}
			\|\nabla \Phi_{u_{0}}(v)-\nabla \Phi_{u_{0}}(w)\|_{L^{2}_{x}}  & \leq & C T^{\frac{3}{10}}R^{4}\| v- w\|_{X_{I}} +  C R^{4}\| v- w\|_{X_{I}} + CT\| v- w\|_{X_{I}}
		\end{eqnarray*}
		and
		\begin{eqnarray*}
			\|\Phi_{u_{0}}(v)-\Phi_{u_{0}}(w)\|_{X_{I}} 
			& \leq & CT^{\frac{1}{2}}R^{4}\|v-w\|_{X_{I}}+ C T^{\frac{3}{10}}R^{4}\| v- w\|_{X_{I}} \\
			&    & \mbox{      }+  C R^{4}\| v- w\|_{X_{I}} + CT\| v- w\|_{X_{I}}.
		\end{eqnarray*}
		These give local existence as long as one chooses small enough constants $T, R$. Global existence is obtained via energy estimates, for details, see \cite[Remark 1]{BrCaCaSa}.
	\end{proof}
	
		\subsection{Auxiliary results} 
		We present two results that were used in this work. The first one ensures that the solution of the nonhomogeneous damped Schr\"{o}dinger equation satisfies a certain inequality:
		\begin{proposition}\label{prop4.12}
		Let $u \in C([a,b];{H}^{1}(\mathbb{R}^{3}))$ be a solution of the damped Schr\"{o}dinger equation 
		$$i\partial_{t}v + \Delta v -v -a(1-\Delta)^{-1}a\partial_{t}v = f,$$
		on $I=[a,b]$ with $\nabla f \in L^{2}(I;L^{\frac{6}{5}}  (\mathbb{R}^{3}))$ and $ f \in L^{1}(I; L^{2}( \mathbb{R}^{3})).$ Thus, the following inequality holds
		\begin{equation*}
			\|\nabla v\|_{L^{\frac{10}{3}}(I;L^{\frac{10}{3}}(\mathbb{R}^{3}))} + \|\nabla v\|_{L^{10}(I;L^{\frac{30}{13}}(\mathbb{R}^{3}))}
			+  \sup_{t \in I} \|v(t)\|_{L^{2}} + \sup_{t \in I} \|\nabla v(t)\|_{L^{2}} 
		\end{equation*}
		\begin{equation*}
			\leq C \Big(\| v(a)\|_{H^{1}} + \|\nabla f\|_{L^{2}(I;L^{\frac{6}{5}}(\mathbb{R}^{3}))} + \|f\|_{L^{1}(I; L^{2}( \mathbb{R}^{3}))} \Big).
		\end{equation*}
		
	\end{proposition}
	\begin{proof}
		The solution $v$ satisfies
		$$v(t) = e^{it\Delta}v(a) + \int_{a}^{t} e^{i(t-\tau)\Delta} f \ d\tau + \int_{a}^{t} e^{i(t-\tau)\Delta} [v + a(1-\Delta)^{-1}a\partial_{t}v] \ d\tau .$$
So, 
		\begin{eqnarray*}
			\|v(t)\|_{L^{2}} 
			& \leq &  C \|v(a)\|_{H^{1}}  + C \| f\|_{L^{1}(I; L^{2}( \mathbb{R}^{3}))} + C_{I}\sup_{t \in I} \|v(t)\|_{ L^{2}( \mathbb{R}^{3})},
		\end{eqnarray*}
		and
		\begin{eqnarray*}
			\|\nabla v(t)\|_{L^{2}} 
			& \leq & C \|v(a)\|_{H^{1}}  + C \|\nabla f\|_{L^{2}(I;L^{\frac{6}{5}}(\mathbb{R}^{3}))} + C_{I} \sup_{t \in I} \|\nabla v(t)\|_ {L^{2}( \mathbb{R}^{3})} + C_{I} \sup_{t \in I} \| v(t)\|_ {L^{2}( \mathbb{R}^{3})} \\
			&   & \mbox{   } +  C \|f\|_{L^{1}(I; L^{2}( \mathbb{R}^{3}))}.
		\end{eqnarray*}
Additionally, we get 
		\begin{eqnarray*}
			\|v\|_{L^{10}(I \times \mathbb{R}^{3})} 
			& \leq & C \|v(a)\|_{H^{1}}  + C \|\nabla f\|_{L^{2}(I;L^{\frac{6}{5}}(\mathbb{R}^{3}))} + C_{I} \sup_{t \in I} \|\nabla v(t)\|_ {L^{2}( \mathbb{R}^{3})} + C_{I} \sup_{t \in I} \| v(t)\|_ {L^{2}( \mathbb{R}^{3})} \\
			&   & \mbox{   } +  C \|f\|_{L^{1}(I; L^{2}( \mathbb{R}^{3}))}
		\end{eqnarray*}
and
		\begin{eqnarray*}
			\|\nabla v\|_{L^{\frac{10}{3}}_{t} L^{\frac{10}{3}}_{x}}
			& \leq & C \|v(a)\|_{H^{1}}  + C \|\nabla f\|_{L^{2}(I;L^{\frac{6}{5}}(\mathbb{R}^{3}))} + C_{I} \sup_{t \in I} \|\nabla v(t)\|_ {L^{2}( \mathbb{R}^{3})} + C_{I} \sup_{t \in I} \| v(t)\|_ {L^{2}( \mathbb{R}^{3})} \\
			&   & \mbox{   } +  C \|f\|_{L^{1}(I; L^{2}( \mathbb{R}^{3}))}.
		\end{eqnarray*}
		Putting together these inequalities, one obtains
				\begin{equation*}
			\|\nabla v\|_{L^{\frac{10}{3}}(I;L^{\frac{10}{3}}(\mathbb{R}^{3}))} + \|\nabla v\|_{L^{10}(I;L^{\frac{30}{13}}(\mathbb{R}^{3}))}
			+  \sup_{t \in I} \|v(t)\|_{L^{2}} + \sup_{t \in I} \|\nabla v(t)\|_{L^{2}} 
		\end{equation*}
		\begin{equation*}
			\leq C \| v(a)\|_{H^{1}} + C\|\nabla f\|_{L^{2}(I;L^{\frac{6}{5}}(\mathbb{R}^{3}))} + C\|f\|_{L^{1}(I; L^{2}( \mathbb{R}^{3}))} 	+  C_{I}\sup_{t \in I} \|v(t)\|_{L^{2}} + \sup_{t \in I} C_{I}\|\nabla v(t)\|_{L^{2}} ,
		\end{equation*}
and the estimate hold is the length of $I$ is small enough. The large-time result follows by a bootstrap argument. 
	\end{proof}
\begin{remark}\label{obs8} The same result is also true for the nonhomogeneous Schr\"{o}dinger equation $i\partial_{t}v + \Delta v -v = f,$ where $\nabla f \in L^{2}(I;L^{\frac{6}{5}}  (\mathbb{R}^{3}))$ and $ f \in L^{1}(I; L^{2}( \mathbb{R}^{3}))$.
\end{remark}

Finally, we state and prove a result obtained as a consequence of the existence of solutions in the Strichartz spaces. This result was shown by Cazenave and Weissler in \cite{cazenave}.
	\begin{proposition}\label{prop2.5}
		If $u$ is a solution of 
		\begin{equation*}
			\left\{
			\begin{array}{ll}
				i\partial_{t} u + \Delta u -|u|^{4}u  =0,& \mbox{ on } \mathbb{R} \times \mathbb{R}^{3},\\
				u(t_{0})= u_{0},& \mbox{ on } \mathbb{R},
			\end{array}
			\right.
		\end{equation*}
 such that $u \in L^{10}(\mathbb{R}^{4})$ and $u \in L^{10}(\mathbb{R};L^{\frac{30}{13}}(\mathbb{R}^{3}))$, where $u_{0} \in H^{1}(\mathbb{R}^{3})$, $\|u_{0}\|_{H^{1}} < \lambda$, $\lambda>0$ small enough,  there exists $u_{\pm} \in \dot{H}^{1}(\mathbb{R}^{3})$ such that
		$$\lim_{t \rightarrow \pm\infty} \|u(t) - e^{it\Delta}u_{\pm}\|_{\dot{H}^{1}} =0.$$
	\end{proposition}
	\begin{proof}
		Note that 
		\begin{eqnarray*}
			\Big\| \int_{t}^{+\infty} \nabla e^{i(t-\tau)\Delta } |u|^{4}u \ d\tau \Big\|_{L^{2}} & \leq & C \|\nabla |u|^{4}u\|_{L^{2}(t,+\infty) L^{\frac{6}{5}}(\mathbb{R}^{3})} \\
			& \leq & C \|u\|^{4}_{L^{10}(t,+\infty) L^{10}(\mathbb{R}^{3})}\|\nabla u\|_{L^{10}(t,+\infty) L^{\frac{30}{13}}(\mathbb{R}^{3})} \rightarrow 0
		\end{eqnarray*} 
		as $t \rightarrow + \infty$. Then, with
		$$u(t) = e^{i(t-t_{0})\Delta}u_{0} + \int_{t_{0}}^{t}e^{i(t-\tau)\Delta}  |u|^{4}u \ d\tau,$$
		taking
		$$u_{+}= e^{-it_{0}\Delta} u_{0} + \int_{t_{0}}^{+\infty} e^{-i\tau\Delta} |u|^{4}u \ d\tau\quad \text{and}\quad u_{-}= e^{-it_{0}\Delta} u_{0} - \int_{-\infty}^{t_{0}} e^{-i\tau\Delta} |u|^{4}u \ d\tau,$$
the result holds.
	\end{proof}

	\section{Propagation results for the linear Schr\"{o}dinger equation}\label{apB}
	In this appendix, we collect some results of propagation for solutions of the linear Schr\"{o}dinger equation following the ideas contained in \cite{DeGeLe}. The results presented here are essential to prove our main result, that is, the exponential stabilizability result.
	
	\begin{lemma} \label{lemma4.15}
	Let $u_{n}, \tilde{u}_{n}$ be two sequences of solutions for
	\begin{equation*} 
		\left\{
		\begin{array}{ll}
			i\partial_{t} u_{n} + \Delta u_{n} - u_{n} - |u_{n}|^{4}u_{n} = a(1-\Delta)^{-1}a\partial_{t} u_{n}, \ \mbox{ on } [0,T] \times \mathbb{R}^{3},\\
			u_{n}(0)= u_{0,n},  \mbox{ bounded in } {H}^{1}(\mathbb{R}^{3}) \mbox{ with } \|u_{0,n}\|_{H^{1}} < \lambda_{0},
		\end{array}
		\right.
	\end{equation*}
	and
	\begin{equation*} 
		\left\{
		\begin{array}{lr}
			i\partial_{t} \tilde{u}_{n}+ \Delta \tilde{u}_{n} -\tilde{u}_{n} - |\tilde{u}_{n}|^{4}\tilde{u}_{n} = 0, \ \mbox{ on } [0,T] \times \mathbb{R}^{3},\\
			\tilde{u}_{n}(0)= \tilde{u}_{0,n}, \ \mbox{ bounded in } {H}^{1}(\mathbb{R}^{3}) \mbox{ with } \|\tilde{u}_{0,n}\|_{H^{1}} < \lambda_{0},
		\end{array}
		\right.
	\end{equation*}
	respectively, with $\|u_{n,0} - \tilde{u}_{n,0}\|_{{H}^{1}} \rightarrow 0$ and $\|(1-\Delta)^{-\frac{1}{2}}a \partial_{t}u_{n}\|_{L^{2}([0,T];L^{2}(\mathbb{R}^{3}))} \rightarrow 0$ as $n\rightarrow \infty.$ Then,
	\begin{equation*}
		\|u_{n}-\tilde{u}_{n}\|_{L^{10}([0,T]\times \mathbb{R}^{3})} + \|\nabla (u_{n} -\tilde{u}_{n})\|_{L^{\frac{10}{3}}_{t}L^{\frac{10}{3}}_{x}} + \sup_{t \in [0,T]} \|\nabla (u_{n} -\tilde{u}_{n})\|_{L^{2}} +  \sup_{t \in [0,T]} \|u_{n} -\tilde{u}_{n}\|_{L^{2}} \rightarrow 0
	\end{equation*} as $n\rightarrow \infty$. 
\end{lemma}
\begin{proof}
Let $r_{n} = u_{n} - \tilde{u}_{n}$. It satisfies the system
	\begin{equation*} 
		\left\{
		\begin{array}{ll}
			i\partial_{t} r_{n} + \Delta r_{n} - r_{n} - |u_{n}|^{4}u_{n}+ |\tilde{u}_{n}|^{4}\tilde{u}_{n} = a(1-\Delta)^{-1}a\partial_{t} u_{n},&\mbox{ in } [0,T] \times \mathbb{R}^{3},\\
			r_{n}(0) = u_{0,n} - \tilde{u}_{0,n}.
		\end{array}
		\right.
	\end{equation*}
	Denote $$|||.|||_{[0,T]} = 	\|\cdot\|_{L^{10}([0,T]\times \mathbb{R}^{3})} + \|\nabla \cdot\|_{L^{\frac{10}{3}}_{t}L^{\frac{10}{3}}_{x}} + \|\nabla \cdot\|_{L^{10}_{t}L^{\frac{30}{13}}_{x}}.$$ 
Strichartz's estimates give us that
	\begin{equation}\label{jajajajajaj}
	\begin{split}
		|||r_{n}|||_{[0,T]} &+ \sup_{t \in [0,T]} \|\nabla r_{n}(t)\|_{L^{2}} + \sup_{t \in [0,T]} \|r_{n}(t)\|_{L^{2}} 
		 \leq  	\| r_{n}(0)\|_{H^{1}}  + \|\nabla (u_{n}^{5} - \tilde{u}_{n}^{5})\|_{L^{2}_{t}L^{\frac{6}{5}}_{x}}  \\
		&   + \| a(1-\Delta)^{-1}a\partial_{t} u_{n}\|_{L^{1}_{t}H^{1}_{x}} + \|u_{n}^{5} - \tilde{u}_{n}^{5}\|_{L^{1}_{t}L^{2}_{x}}.
	\end{split}
	\end{equation}
	Thus, on the one hand, we have
	\begin{eqnarray*} 
		\| a(1-\Delta)^{-1}a\partial_{t} u_{n}\|_{L^{1}_{t}H^{1}_{x}}
		& \leq & C \| a(1-\Delta)^{-1}a\partial_{t} u_{n}\|_{L^{2}_{t}H^{1}_{x}} \nonumber \\
		& \leq & C \|(1-\Delta)^{-\frac{1}{2}}a\partial_{t} u_{n}\|_{L^{2}_{t}L^{2}_{x}} \rightarrow 0, 
	\end{eqnarray*}
	as $n\rightarrow \infty$. On the other hand, 
	\begin{equation*}
	\begin{split}
		\|\nabla (u_{n}^{5} - \tilde{u}_{n}^{5})\|_{L^{2}_{t}L^{\frac{6}{5}}_{x}} 
		 \leq &  \|u_{n}\|_{{L^{10}_{t}L^{10}_{x}}}^{4}\|\nabla u_{n} - \nabla \widetilde{u}_{n}\|_{L^{10}_{t}L^{\frac{30}{13}}_{x}} \\&+ \|u_{n}-\widetilde{u}_{n}\|_{{L^{10}_{t}L^{10}_{x}}}\| \nabla \widetilde{u}_{n}\|_{L^{10}_{t}L^{\frac{30}{13}}_{x}}  \|u_{n}\|^{3}_{{L^{10}_{t}L^{10}_{x}}} \\ &  + \|u_{n}-\widetilde{u}_{n}\|_{{L^{10}_{t}L^{10}_{x}}} \| \nabla \widetilde{u}_{n}\|_{L^{10}_{t}L^{\frac{30}{13}}_{x}}  \|\widetilde{u}_{n}\|^{3}_{{L^{10}_{t}L^{10}_{x}}} \\
		 \leq & C\|\nabla r_{n}\|_{L^{10}_{t}L^{\frac{30}{13}}_{x}} \Big( \| \nabla \widetilde{u}_{n}\|^{4}_{L^{10}_{t}L^{\frac{30}{13}}_{x}} + \| \nabla {u}_{n}\|^{4}_{L^{10}_{t}L^{\frac{30}{13}}_{x}}\\&  + \| \nabla \widetilde{u}_{n}\|_{L^{10}_{t}L^{\frac{30}{13}}_{x}}\|\nabla u_{n}\|^{3}_{L^{10}_{t}L^{\frac{30}{13}}_{x}}\Big)
		\end{split}
	\end{equation*}
	and
	\begin{eqnarray*}
		\| u_{n}^{5} - \tilde{u}_{n}^{5}\|_{L^{1}_{t}L^{2}_{x}} 
		& \leq & 	\| u_{n} - \tilde{u}_{n}\|_{L^{5}_{t}L^{10}_{x}} \Big(\|u_{n}\|^{4}_{L^{5}_{t}L^{10}_{x}} + \|\tilde{u}_{n}\|^{4}_{L^{5}_{t}L^{10}_{x}} \Big) \\
		& \leq & C\|\nabla r_{n}\|_{L^{10}_{t}L^{\frac{30}{13}}_{x}} \Big( \| \nabla \widetilde{u}_{n}\|^{4}_{L^{10}_{t}L^{\frac{30}{13}}_{x}} + \| \nabla {u}_{n}\|^{4}_{L^{10}_{t}L^{\frac{30}{13}}_{x}} \Big).
	\end{eqnarray*}
So, dividing the interval $[0 , T ]$ in a finite number of intervals $I_{i,n} = [a_{i,n},a_{i+1,n}]$, $1 \leq i \leq N$, such that $$ C\Big(\|\nabla  u_{n}\|^{4}_{L^{10}_{t}L^{\frac{30}{13}}_{x}}  + \| \nabla \widetilde{u}_{n}\|^{4}_{L^{10}_{t}L^{\frac{30}{13}}_{x}} + \| \nabla \widetilde{u}_{n}\|_{L^{10}_{t}L^{\frac{30}{13}}_{x}}\|\nabla u_{n}\|^{3}_{L^{10}_{t}L^{\frac{30}{13}}_{x}}\Big) \leq \frac{1}{2},$$ the terms of \eqref{jajajajajaj} can be controlled. We iterate this estimate $N$ times, which gives the result.
\end{proof}
\begin{lemma} \label{lemma4.39}
	Let $T>0$. There exists $C>0$ such that any solution $u$ to 
	\begin{equation} \label{sis}
		\left\{
		\begin{array}{ll}
			i\partial_{t} u + \Delta u -u - |u|^{4}u = a(1-\Delta)^{-1}a\partial_{t} u, &\ \mbox{ on } [0,T] \times \mathbb{R}^{3}, \\
			u(0)= u_{0}, \quad  \|u_{0}\|_{H^{1}} \leq \lambda_{0},
		\end{array}
		\right.
	\end{equation}
with $\lambda_0$ given by \eqref{def3}, satisfies 
	$\|u\|_{L^{\infty}([0,T];L^{2}(\mathbb{R}^{3}))} \leq C \|u_{0}\|_{L^{2}(\mathbb{R}^{3})}.$
\end{lemma}
\begin{proof}
	First, notice that $u \in L^{7}([0,T];L^{14}(\mathbb{R}^{3}))$. By a Sobolev embedding, Strichartz estimates, and an interpolation argument, we get $u \in L^{4}([0,T];L^{12}(\mathbb{R}^{3}))$. Observe that $V=|u|^{4} \in L^{1}([0,T];L^{3}(\mathbb{R}^{3}))$. 	Multiplying the first equation of \eqref{sis} by $\overline{u}$, integrating and taking the imaginary part yields 
		\begin{eqnarray*}
		\|u\|^{2}_{L^{\infty}([0,t];L^{2})} \leq 2C(t + \|V\|_{L^{1}([0,t];L^{3})} )\|u\|^{2}_{L^{\infty}([0,t];L^{2})} + \|u(0)\|^{2}_{L^{2}}.
	\end{eqnarray*}
	We can divide the interval $[0,T]$ into a finite number of intervals $[a_{i},a_{i+1}]$, $i=1,...,N$, such that $2C(t + \|V\|_{L^{1}([a_{i},a_{i+1}];L^{3})} )<1/4$. In each of these intervals, we have
	$	\|u\|^{2}_{L^{\infty}([a_{i},a_{i+1}];L^{2})} \leq C \|u(a_{i})\|^{2}_{L^{2}}.$
	We obtain the expected result by iteration. The final constant $C$ only depends on $\lambda_{0}$ and $T$.
\end{proof}
As a consequence of the previous result, we have the following corollary.
\begin{corollary}\label{coro4.40}
	Let $T>0$. For any $\varepsilon>0$, there exists $\delta>0$ such that any solution $u$  satisfying \eqref{eq 104} and 
	$\|u_{0}\|_{H^{-1}} \leq \delta$ satisfies $\|u(T)\|_{H^{-1}} \leq \varepsilon.$
\end{corollary}
\begin{proof}
	By Lemma \ref{lemma4.39}, we have
	$\|u(T)\|_{H^{-1}} \leq C \|u(0)\|_{L^{2}}.$
	However, by an interpolation argument between $H^{s}(\mathbb{R}^{3})$ spaces, $s\in \mathbb{R}$, we have
$$
		\|u(0)\|_{L^{2}}  \leq   \|u(0)\|^{\frac{1}{2}}_{H^{-1}}\|u(0)\|^{\frac{1}{2}}_{H^{1}}  \leq  \lambda_{0}^{\frac{1}{2}} \|u(0)\|^{\frac{1}{2}}_{H^{-1}}.
$$
	Then, 
	$$\|u(T)\|_{H^{-1}}  \leq C \|u(0)\|_{L^{2}} \leq C \lambda_{0}^{\frac{1}{2}} \|u(0)\|^{\frac{1}{2}}_{H^{-1}}.$$
	Taking $\delta = \frac{\varepsilon^{2}}{C^{2}\lambda_{0}}$, we conclude that $\|u(T)\|_{H^{-1}}  \leq  \varepsilon.$
\end{proof}

The next proposition gives us the propagation of compactness. 
	\begin{proposition} \label{prop6.1}
		Let $L = i\partial_{t} + \Delta + R_{0}$ where $R_{0}(t, x, D_{x})$ is a tangential pseudodifferential operator of order 0 and $\{u_{n}\}$ a sequence of functions satisfying,		\begin{equation} \label{eq 121}
			\sup_{t \in [0,T]} \|\chi u_{n}(t)\|_{H^{1}(\mathbb{R}^{3})} \leq C, \ \ \sup_{t \in [0,T]}\|\chi u_{n}(t)\|_{L^{2}(\mathbb{R}^{3})} \rightarrow 0 \ \ \mbox{and } \int_{0}^{T}\|Lu_{n}(t)\|^{2}_{L^{2}} \ dt \rightarrow 0,
		\end{equation}
for every $\chi \in C_{0}^{\infty}(\mathbb{R}^{3})$, with $\chi(x) = 1$ when $ x \in supp(\chi)=K$. There exist a subsequence $\{u_{n'}\}$ of $\{u_{n}\}$ and a positive measure $\mu$ on $(0,T) \times \mathbb{R}^{3} \times S^{3}$ such that, for every tangential pseudodifferential operator $A = A(t, x, D_{x})$ of order 2 with principal symbol $\sigma(A) = a_{2}(t,x,\xi),$ one has
		\begin{equation} \label{eq 122}
			\langle A(t, x, D_{x}) \chi u_{n'},\chi u_{n'}\rangle_{L^{2}} \longrightarrow \int_{(0,T) \times \mathbb{R}^{3} \times S^{3}} a_{2}(t,x,\xi) \ d\mu(t,x,\xi).
		\end{equation}
		Moreover, if $G_{s}$ denotes the geodesic flow on $\mathbb{R}^{3} \times S^{2}$, one has, for every $s \in \mathbb{R}$,
		\begin{equation} \label{eq 123}
			G_{s}(\mu) = \mu.
		\end{equation}
		In other words, $\mu$ is invariant by the geodesic flow ``at fixed t."
	\end{proposition}
	\begin{proof}
		The construction of the tangential microlocal defect measures $\mu$ satisfying \eqref{eq 122} is classical and can be found in \cite{Gerard}. The first estimate in \eqref{eq 121} combined with a separability argument allows to find a subsequence $\{u_{n'}\}$ such that the left-hand side of \eqref{eq 122} converges for all $A$. Then, the second estimate in \eqref{eq 121} and the Gårding inequality imply the existence of some positive measure $\mu$ such that \eqref{eq 122} holds.
		
		For the propagation, i.e., property \eqref{eq 123}, we consider $\varphi = \varphi(t) \in C_{0}^{\infty}(0,T)$, $B(x,D_{x})$ a pseudo-differential operator of order $1$ with principal symbol $b_{1}$, $A(t, x, D_{x}) = \varphi(t)B(x,D_{x})$ and, for $\varepsilon > 0$, $A_{\varepsilon} = \varphi B_{\varepsilon} = Ae^{\varepsilon \Delta}$. Moreover, denote
		$$\alpha_{n}^{\varepsilon} = \big(Lu_{n},A^{*}_{\varepsilon}u_{n}\big)_{L^{2}([0,T]\times \mathbb{R}^{3})} - \big(A_{\varepsilon}u_{n},Lu_{n}\big)_{L^{2}([0,T]\times \mathbb{R}^{3})}. $$
		By the assumption \eqref{eq 121}, $\sup_{\varepsilon} \alpha_{n}^{\varepsilon} \rightarrow 0$ if $n \rightarrow \infty$. On the other hand,
		\begin{eqnarray*}
			\alpha_{n}^{\varepsilon} & = & \big(i\partial_{t} u_{n}+ \Delta u_{n} + R_{0}u_{n},A^{*}_{\varepsilon}u_{n}\big)_{L^{2}([0,T]\times \mathbb{R}^{3})} - \big(A_{\varepsilon}u_{n},i\partial_{t} u_{n}+ \Delta u_{n} + R_{0}u_{n}\big)_{L^{2}([0,T]\times \mathbb{R}^{3})} \\
			& = &  i\big((\partial_{t}A_{\varepsilon}) u_{n},u_{n}\big)_{L^{2}([0,T]\times \mathbb{R}^{3})} + \big( A_{\varepsilon}\Delta u_{n} ,u_{n}\big)_{L^{2}([0,T]\times \mathbb{R}^{3})} + \big( A_{\varepsilon}R_{0}u_{n},u_{n}\big)_{L^{2}([0,T]\times \mathbb{R}^{3})} \\
			&   &  \mbox{  } -i\big((\partial_{t}A_{\varepsilon}) u_{n},u_{n}\big)_{L^{2}([0,T]\times \mathbb{R}^{3})} - \big(\Delta A_{\varepsilon}u_{n},  u_{n} \big)_{L^{2}([0,T]\times \mathbb{R}^{3})} - \big(R^{*}_{0}A_{\varepsilon}u_{n}, u_{n}\big)_{L^{2}([0,T]\times \mathbb{R}^{3})} \\
			& = & \big([A_{\varepsilon},\Delta]u_{n},u_{n}\big)_{L^{2}([0,T]\times \mathbb{R}^{3})} + \big([A_{\varepsilon}R_{0}-R^{*}_{0}A_{\varepsilon}]u_{n},u_{n}\big)_{L^{2}([0,T]\times \mathbb{R}^{3})}.
		\end{eqnarray*}
		Observe that taking $\big([A_{\varepsilon}R_{0}-R^{*}_{0}A_{\varepsilon}]u_{n},u_{n}\big)_{L^{2}([0,T]\times \mathbb{R}^{3})} = \beta^{\varepsilon}_{n}$, we have $\sup_{\varepsilon} \beta_{n}^{\varepsilon} \rightarrow 0$ as $n \rightarrow \infty$. Finally, passing to the limit as $\varepsilon \rightarrow 0$, we
		obtain
		$$\big(\varphi[B,\Delta]u_{n},u_{n}\big)_{L^{2}([0,T]\times \mathbb{R}^{3})} \rightarrow 0$$
		as $n \rightarrow \infty$. This means that, for any $\chi \in C_{0}^{\infty}(\mathbb{R}^{3})$, one has
		\begin{equation}\label{eq 124}
			\big(\chi \varphi[B,\Delta]u_{n},\chi u_{n}\big)_{L^{2}([0,T]\times \mathbb{R}^{3})} \rightarrow 0
		\end{equation}
		as $n \rightarrow \infty$. Let $D:= \varphi[B,\Delta]$. Note that $D$ is a pseudodifferential operator of order two and, moreover, 
		\begin{eqnarray*}
			\big( \varphi[B,\Delta]\chi u_{n},\chi u_{n}\big)_{L^{2}([0,T]\times \mathbb{R}^{3})} & = &	\big( D\chi u_{n},\chi u_{n}\big)_{L^{2}([0,T]\times \mathbb{R}^{3})}\\
			& = & \big( [D,\chi] u_{n},\chi u_{n}\big)_{L^{2}([0,T]\times \mathbb{R}^{3})} + \big( \chi Du_{n},\chi u_{n}\big)_{L^{2}([0,T]\times \mathbb{R}^{3})} \\
			& \rightarrow & 0
		\end{eqnarray*}
		as $n \rightarrow \infty$, using \eqref{eq 124} and 
		\begin{eqnarray*}
			\big( [D,\chi] u_{n},\chi u_{n}\big)_{L^{2}([0,T]\times \mathbb{R}^{3})}  \leq \|[D,\chi] u_{n}\|_{L^{2}} \|\chi u_{n}\|_{L^{2}}
			 \leq C \|u_{n}\|_{H^{1}} \|\chi u_{n}\|_{L^{2}}. 	
		\end{eqnarray*}
		In view of \eqref{eq 122}, one has
		\begin{equation*}
			\int_{(0,T)\times \mathbb{R}^{3} \times S^{3}} \varphi\{|\xi|^{2}_{x},b_{1}\} \ d\mu(t,x,\xi) = 0.
		\end{equation*}
		This identity expresses property \eqref{eq 123} and completes the proof.
	\end{proof}
	With the propagation of compactness in hands, the propagation of regularity holds. 
	\begin{corollary} \label{coro6.2}
		Assume that $\omega \subset \mathbb{R}^{3}$ satisfies Assumption \ref{assum5.1}. Let $\{u_{n}\}$ be a sequence of functions bounded in $L^{\infty}([0,T], H^{1}(\mathbb{R}^{3}))$, converging to 0 in $L^{2}$, and satisfying
		\begin{equation} \label{eq 125}
			\left\{
			\begin{array}{lr}
				i\partial_{t} u_{n} + \Delta u_{n} \rightarrow 0 \mbox{ in } L^{2}([0,T],H^{1}(\mathbb{R}^{3})),\\
				u_{n} \rightarrow 0\ \mbox{ in }L^{2}([0,T],H^{1}_{loc}(\omega)).
			\end{array}
			\right.
		\end{equation}
		Then, $\{u_{n}\}$ strongly converges to 0 in $L^{\infty}([0,T], H^{1}_{loc}(\mathbb{R}^{3}))$.
	\end{corollary}
	\begin{proof}
		By Proposition \ref{prop6.1}, we can attach to the sequence $(u_{n})$ a microlocal defect measure in $L^{2}((0, T), H^{1}(\mathbb{R}^{3}))$ that propagates with infinite speed along the geodesics of $\mathbb{R}^{3}$. Using the second equation of \eqref{eq 125}, we can deduce that
		$$\mu = 0 \mbox{  on  } (0,T) \times \omega \times S^{3},$$
		which yields, by the propagation \eqref{eq 123} and Assumption \ref{assum5.1}, $\mu=0$ on $(0,T)\times \mathbb{R}^{3} \times S^{3}$. This means that $u_{n} \rightarrow 0$ in $L^{2}_{loc}((0,T);H^{1}_{loc}(\mathbb{R}^{3}))$.
		Finally, solving the first equation of \eqref{eq 125} with initial data $u_{n}(t_{0})$, where $t_{0} \in (0, T)$ is such that $\|u_{n}(t_{0})\|_{H^{1}_{loc}} \rightarrow 0$, this implies the strong convergence $u_{n}(t)\rightarrow0$ in the space $L^{\infty}([0, T], H^{1}_{loc}(\mathbb{R}^{3}))$.
	\end{proof}

	\subsection*{Acknowledgment}
This work is part of Carvalho's Ph.D. thesis at the Department of Mathematics of the Universidade Federal de Pernambuco. It was mostly done while the third author was visiting Universit\'e de Lorraine. The third author thanks the host institution for its warm hospitality.


\begin{thebibliography}{99}
		
\bibitem{bahouri} H. Bahouri and P. G\'{e}rard, \textit{High frequency approximation of critical nonlinear wave equations}, Amer. J. Math. 121 (1999), 131--175.
		
\bibitem{BaPu} L. Baudouin and J.-P. Puel, \textit{Uniqueness and stability in an inverse problem for the Schr\"{o}dinger equation}, Inverse Problems, 18 (2001), 1537--1554.

\bibitem{BoCa} C. A. Bortot and M. Cavalcanti, \textit{Asymptotic Stability for the Damped Schr\"odinger Equation on Noncompact Riemannian Manifolds and Exterior Domains}, Communications in Partial Differential Equations, 39(9) (2014), 1791--1820.
		
\bibitem{bourgain} J. Bourgain, \textit{Global wellposedness of defocusing critical nonlinear Schr\"odinger equation in the radial case}, J. Amer. Math. Soc., 12:1 (1999), 145--171.
			
\bibitem{bourgain1} J. Bourgain, New global well-posedness results for nonlinear Schr\"oodinger equations, \textit{AMS Publications}, 1999.
			
\bibitem{BrCaCaSa}P. Braz e Silva, R. A. Capistrano-Filho, J. D. N. Carvalho, and D. dos Santos Ferreira, \textit{Control of the Schr\"odinger equation in $\mathbb{R}^3$: The critical case}, arXiv:2404.07749 [math.AP].

\bibitem{BurqMoyano}N. Burq and I.Moyano, \textit{A remark on the logarithmic decay of the damped wave and Schrödinger equations on a compact Riemannian manifold}, Port. Math. 80 (2023), 369--390.
			
\bibitem{CaPa} R. A. Capistrano--Filho and A. Pampu, \textit{The fractional Schr\"odinger equation on compact manifolds: Global controllability results}, Mathematische Zeitschrift, 301 (2022), 3817--3848.
		
\bibitem{CaGa} L. Cardoulis and P. Gaitan, \textit{Simultaneous identification of diffusion coefficient and the potential for the Schr\"{o}dinger operator with only one observation}, Inverse Problems, 26 (2010).
		
\bibitem{CaCriGa} L. Cardoulis, M. Cristofol, and P. Gaitan, \textit{Inverse problem for the Schr\"{o}dinger operator in an unbounded strip}, J. Inverse Ill-Posed Probl., 16 (2008), 127--146.

\bibitem{CaCaFuNa} M.M. Cavalcanti, V.N. Domingos Cavalcanti, R. Fukuoka, F. Natali, \textit{Exponential stability for the $2-D$ defocusing Schr\"odinger equation with locally distributed damping}, Differential Integral Equations, 22 (7--8) (2009), 617--636.

\bibitem{cazenave_book} T. Cazenave, Semilinear Schr\"{o}dinger Equations, \textit{Amer. Math. Soc.}, Providence, RI, 2003.
		
\bibitem{cazenave} T. Cazenave, F. B. Weissler, \textit{The Cauchy problem for the critical nonlinear Schr\"{o}dinger equation in $H^{s}$}, Nonlinear Anal.: Theory Methods Appl., 14 (1990), 807--836.
		
\bibitem{CaWe} T. Cazenave and F. B. Weissler, Some remarks on the nonlinear Schr\"odinger equation in the critical case, Nonlinear semigroups, Partial Differential Equations and Attractors, \textit{Lecture Notes in Math.}, (1989), 18--29.
		
\bibitem{Co} J. Colliander, M. Keel, G. Staffilani, H. Takaoka, and T. Tao, \textit{Global well-posedness and scattering for the energy-critical nonlinear Schr\"odinger equation in $\mathbb{R}^3$}, Annals of Mathematics, 167 (2008), 767--865.
		
\bibitem{DeGeLe} B. Dehman, P. G\'erard, G. Lebeau, S\textit{tabilization and control for the nonlinear Schr\"{o}dinger equation on a compact surface}, Mathematische Zeitschrift, 254 (2006) 729--749.
		
\bibitem{DeLeZu} B. Dehman, G. Lebeau, and E. Zuazua, \textit{Stabilization and control for the subcritical semilinear wave equation}, Annales scientifiques de l'École Normale Supérieure, 4:36 (2003), 525--551.
		
\bibitem{GalGe} I. Gallagher and P. G\'{e}rard, \textit{Profile decomposition for the wave equation outside a convex obstacle}, J. Math. Pures Appl., 80:1 (2001), 1--49.
		
\bibitem{GiVe} J. Ginibre and G. Velo, \textit{The global Cauchy problem for the nonlinear Schr\"odinger equation revisited}, Ann. Inst. H. Poincar\'e Anal. Non Lin\'eaire, 2 (1985), 309--327.
		
\bibitem{Give1} J. Ginibre and G. Velo, \textit{Scattering theory in the energy space for a class of nonlinear Schr\"odinger
equations}, J. Math. Pure. Appl., 64 (1985), 363--401.
		
\bibitem{Gri} M. Grillakis, \textit{On nonlinear Schr\"odinger equations}, Comm. Partial Differential Equations, 9-10:25 (2000), 1827--1844.
		
\bibitem{Gross} E. Gross, \textit{Hydrodynamics of a superfluid condensate}, J. Math. Phys. 4 (1963), 195--207.
		
\bibitem{Gerard} P. Gérard, \textit{Microlocal defect measures}, Commun. Partial Diff. Eq., 16 (1991), 1762--1794.
		
\bibitem{Kato} T. Kato, \textit{On nonlinear Schr\"odinger equations}, Ann. Inst. H. Poincare Phys. Theor., 46 (1987), 113--129.
		
		
\bibitem{KeelTao} M. Keel and T. Tao, \textit{Endpoint Strichartz estimates},  Am. J. Math., 120 (1998), 955--980.
		
\bibitem{KeMe} C. E. Kenig and F. Merle, \textit{Global well-posedness, scattering, and blow-up for the energy-critical, focusing, non-linear Schr\"{o}dinger equation in the radial case}, Invent. Math, 166 (2006), 645--675.
		
\bibitem{keraani}  S. Keraani, \textit{On the defect of compactness for the Strichartz estimates of the Schr\"{o}dinger equations}, J. Differential Equations, 175:2 (2001), 353--392. 

\bibitem{Laurent} C. Laurent, \textit{Global controllability and stabilization for the nonlinear Schr\"{o}dinger equation on an interval, ESAIM Control Optimisation and Calculus of Variations},  2 (2010), 356--379.
		
\bibitem{Laurent1} C. Laurent, \textit{Global controllability and stabilization for the nonlinear Schrödinger equation on some compact manifolds of dimension 3}, SIAM Journal on Mathematical Analysis, 42 (2010) 785--832.
		
\bibitem{Laurent2} C. Laurent, \textit{On stabilization and control for the critical Klein-Gordon equation on a 3-D compact manifold}, Journal of Functional Analysis 260, 5 (2011), 1304-1368.
		
\bibitem{Laurent3} C. Laurent, \textit{Internal control of the Schrödinger equation}, Mathematical Control and Related Fields, 4:2 (2014).
		
\bibitem{LaTriZhang} I. Lasiecka, R. Triggiani and X. Zhang, \textit{Carleman estimates at the $H^1(\Omega)$- and $L^2(\Omega)$-level for nonconservative Schr\"{o}dinger equations with unobserved Neumann B.C}, Arch. Inequal. Appl., 2 (2004), 215--338.
		
\bibitem{merle} F. Merle and L. Vega, \textit{Compactness at Blow-up time for $L^{2}$ solutions of the critical nonlinear Schrödinger equations in 2D}, Internat. Math. Res. Notices, 8 (1998), 399--425.
		
\bibitem{Miller} L. Miller, \textit{How violent are fast controls for Schr\"{o}dinger and plate vibrations?}, Arch. Rational Mech. Anal., 172 (2004), 429--456.
	
\bibitem{Pitaevski} L. P. Pitaevski, J. Exptl. Theoret. Phys. 13 (1961) p. 646. (Translation:  \textit{Vortex lines in an
imperfect Bose gas}, Soviet Phys. JETP 40 (1961), p. 451--454.
		
\bibitem{Phung} K.-D. Phung, \textit{Observability and controllability for Schr\"{o}dinger equations}, SIAM J. Control Optim., 40 (2001), 211--230.
		
\bibitem{RaTaTeTu} K. Ramdani, T. Takahashi, G. Tenenbaum, and M. Tucsnak, \textit{A spectral approach for the exact observability of infinite-dimensional systems with skew-adjoint generator,} J. Funct. Anal., 226 (2005), 193--229.

\bibitem{RoZhaSIAM} L. Rosier and B.-Y. Zhang, \textit{Local exact controllability and stabilizability of the nonlinear Schr\"{o}dinger equation on a bounded domain}, SIAM J. Control Optim., 48 (2009), 972--992.
		
\bibitem{SuSu} C. Sulem and P.-L. Sulem, The Nonlinear Schr\"odinger Equation: Self-Focusing and Wave Collapse. \textit{Springer-Verlag, Applied Mathematical Sciences}, (1999).
		
\bibitem{RoMe}	A. Mercado, A. Osses, and L. Rosier, \textit{Inverse problems for the Schr\"odinger equation via Carleman inequalities with degenerate weights}, Inverse Problems, 24 (2008), 1--18.
		
\bibitem{RoZhaJDE} L. Rosier and B.-Y. Zhang, \textit{Exact boundary controllability of the nonlinear Schr\"{o}dinger equation}, J. Differential Equations, 246 (2009), 4129--4153.
		
\bibitem{RoZhaMMM} L. Rosier and B.-Y. Zhang, \textit{Control and Stabilization of the Nonlinear Schr\"{o}dinger Equation on Rectangles,} Mathematical Models, and Methods in Applied Sciences, 12 (2010), 2293--2347.

\bibitem{YuYa} G. Yuan and M. Yamamoto, \textit{Carleman estimates for the Schr\"{o}dinger equation and applications to an inverse problem and an observability inequality}, Chinese Annals of Mathematics,  31:4 (2010), 555--578.
			
\bibitem{Zuazua} E. Zuazua, \textit{Remarks on the controllability of the Schr\"{o}dinger equation, Quantum Control: Mathematical and Numerical Challenges}, 193--211, CRM Proc. Lecture Notes, Vol. 33 (Amer. Math. Soc., 2003).
\end{thebibliography}
\end{document}